\documentclass[11pt]{amsart}

\usepackage{amsfonts,amsmath,latexsym,amssymb,verbatim,amsbsy,amsthm,mathtools}
\usepackage{dsfont}
\mathtoolsset{showonlyrefs}

\usepackage{graphicx}

\usepackage{caption}
\usepackage{float}
\usepackage{wrapfig}

\usepackage[top=1in, bottom=1in, left=1in, right=1in]{geometry}

\usepackage[dvipsnames]{xcolor}

\usepackage[colorlinks=true, pdfstartview=FitV, linkcolor=RoyalBlue,citecolor=ForestGreen, urlcolor=blue]{hyperref}

\def\red{\color{red}}

\theoremstyle{plain}
\newtheorem{THEOREM}{Theorem}[section]

\newtheorem{theorem}[THEOREM]{Theorem}
\newtheorem{corollary}[THEOREM]{Corollary}
\newtheorem{lemma}[THEOREM]{Lemma}
\newtheorem{proposition}[THEOREM]{Proposition}

\theoremstyle{definition}

\newtheorem{definition}[THEOREM]{Definition}

\theoremstyle{remark}

\newtheorem{remark}[THEOREM]{Remark}

\def \a {\alpha}

\def \g {\gamma}
\def \d {\delta}
\def \k {\kappa}
\def \e {\varepsilon}

\def \l {\lambda}

\def \th {\theta}
\def \o {\omega}
\def \w {\omega}

\def \br {{\bf r}}

\def \bw {{\bf w}}
\def \bx {{\bf x}}
\def \by {{\bf y}}
\def \bz {{\bf z}}

\def \bE {{\bf E}}

\def \cF {\mathcal{F}}

\def \cM {\mathcal{M}}

\def \cR {\mathcal{R}}

\def \cW {\mathcal{W}}

\newcommand{\R}{\ensuremath{\mathbb{R}}}   %%% reals
\newcommand{\C}{\ensuremath{\mathbb{C}}}   %%% complex
\def \ds  {\, \mbox{d}s}

\def \ddt  {\frac{\mbox{d\,\,}}{\mbox{d}t}}

\def \Re {\mathrm{Re}}
\def \Im {\mathrm{Im}}

\def \i {\mathrm{i}} %% FT: We should consider using a special symbol for the imaginary unit, as it might get confused with index i of the oscillator system

\title[Synchronization of coupled Stuart-Landau oscillators]{Synchronization of coupled Stuart-Landau oscillators: How heterogeneity can facilitate synchronization}

\author{Ana P Millán}
\address{Department of Electromagnetism and Matter Physics and Institute \emph{Carlos I} of Theoretical and Computational Physics, University of Granada}
\email{apmillan@ugr.es}
\author{David Poyato}
\address{Department of Applied Mathematics and Research Unit ``Modeling Nature'' (MNat), Faculty of Sciences, University of Granada, 18071 Granada, Spain}
\email{davidpoyato@ugr.es}
\author{David N Reynolds}
\address{University of Warsaw, Institute of Applied Mathematics and Mechanics}
\email{reyndn12@go.ugr.es}
\author{Francesco Tudisco}
\address{School of Mathematics and Maxwell Institute for Mathematical Sciences,
University of Edinburgh; and Gran Sasso Science Institute}
\email{f.tudisco@ed.ac.uk}

\usepackage{pdfcomment,todonotes}
\begin{document}

\maketitle

\begin{abstract}
We study the collective dynamics of coupled Stuart--Landau oscillators, which model limit-cycle behavior near a Hopf bifurcation and serve as the amplitude-phase analogue of the Kuramoto model. Unlike the well-studied phase-reduced systems, the full Stuart--Landau model retains amplitude dynamics, enabling the emergence of rich phenomena such as amplitude death, quenching, and multistable synchronization. We provide a complete analytical classification of asymptotic behaviors for identical natural frequencies, but heterogeneous inherent amplitudes in the finite-$N$ setting. In the two-oscillator case, we classify the asymptotic behavior in all possible regimes including heterogeneous natural frequencies and inherent amplitudes, and in particular we identify and characterize a novel regime of \emph{leader-driven synchronization}, wherein one active oscillator can entrain another regardless of frequency mismatch. For general $N$, we prove exponential phase synchronization under sectorial initial data and establish sharp conditions for global amplitude death. Finally, we analyze a real-valued reduction of the model, connecting the dynamics to nonlinear opinion formation and consensus processes. Our results highlight the fundamental differences between amplitude-phase and phase-only Kuramoto models, and provide a new framework for understanding synchronization in heterogeneous oscillator networks.
\end{abstract}

\section{Introduction}

The Stuart--Landau (SL) coupled oscillatory system 
\begin{equation}\label{eq:SL1}
    \dot{z}_j = (\alpha_j + i\omega_j - |z_j|^2)z_j + \frac{\kappa}{N} \sum_{l=1}^N (z_l - z_j), \quad j = 1, \dots, N.
\end{equation}
provides a canonical model for a limit-cycle oscillator near a Hopf bifurcation, and plays a central role in the derivation of the celebrated Kuramoto model of synchronization~\cite{Kur}. Despite this, the full coupled SL oscillator system, retaining both amplitude and phase dynamics, has received comparatively limited analytical attention in the mathematical literature. Yet, SL oscillators appear widely in physics and neuroscience, modeling mesoscopic brain activity and spatially extended nonlinear systems \cite{DN1,DN2,Muñ,vNMD,PD,Sahoo,SETKD,Stam}.

Kuramoto derived his model of coupled phase oscillators from \eqref{eq:SL1} via a phase reduction where oscillators are assumed to never leave a common limit-cycle and therefore the coupled dynamics of phases and amplitudes just reduces to the dynamics of phases \cite{Kur}. The resulting Kuramoto model has become one of the most studied paradigms for synchronization, with a vast literature spanning mathematics \cite{ABPRS,CCHKK,Chiba,BK,DV,DF,DFG,HKR,HR,MP,PPS}, physics \cite{MTB,MS,Str}, and engineering \cite{FNP,JSP,NHSNNN}. However, the original SL system of coupled amplitude-phase oscillators has not received the same degree of analytical scrutiny. Notable exceptions include the work in~\cite{PLE}, which proved a practical synchronization result (i.e., synchronization in the limit $\kappa \to \infty$) over heterogeneous networks, and~\cite{MS}, which analyzed the system in the mean-field limit $N \to \infty$ under the assumption the assumption of amplitude homogeneity (i.e., $\alpha_j=1$ for all $j=1,\ldots,N$). Other studies have examined variations of the SL model \eqref{eq:SL1} with repulsive coupling~\cite{WZ}, star network structures~\cite{CCSL,CLLZ}, or alternative coupling schemes. Nevertheless, the foundational ODE system \eqref{eq:SL1} with general parameters in finite dimension remains largely unexplored.

This paper presents the first comprehensive study of the collective dynamics of the finite-$N$ SL system \eqref{eq:SL1}. 
We analyze the system under heterogeneous parameters $\alpha_j$ and $\omega_j$, with a particular focus on the effect of amplitude heterogeneity. Our contributions include:
\begin{comment}
\begin{itemize}
    \item Characterization of a novel regime of \textit{leader-driven synchronization}, where a single supercritical oscillator entrains all others regardless of frequency mismatch.
    \item A complete classification of asymptotic states in the $N=2$ case across parameters $(\alpha_j, \omega_j, \kappa)$, including active/inactive and coherent/incoherent states.
    \item Extension to arbitrary $N$ with homogeneous natural frequencies, showing exponential phase synchronization and conditions for amplitude death or persistence.
    \item A reduction to a novel nonlinear real-valued system, establishing connections to opinion dynamics and nonlinear consensus models.
\end{itemize}
\end{comment}

\begin{itemize}
\item Complete classification of asymptotic states in the $N=2$ case across all possible values of the parameters $\alpha_j$, $\omega_j$ (both homogeneous and heterogeneous) and $\kappa$ (both small and large). We showcase the different regimes, including active/inactive and coherent/incoherent states.
\item Classification of asymptotic states for arbitrary $N\geq 2$ in the case of identical natural frequencies $\omega_j=0$, heterogeneous inherent amplitudes $\alpha_j$, and general $\kappa$. We show that sectorial solutions tend to coherent states. More specifically, we prove the exponential phase synchronization of these sectorial and we give conditions for the convergence to active/inactive states.
\end{itemize}

In doing so, we find new phenomena that was not found in the cases analyzed in \cite{PLE,MS}. In particular, we remark the following two new findings:

\begin{itemize}
\item {\it Leader-driven synchronization}. In the case of $N=2$ oscillators, we characterize a novel regime of leader-driven synchronization, where a single supercritical oscillator entrains the other regardless of the frequency mismatch.

\item {\it Nonlinear opinion dynamics model}. In the case $N\geq 2$, with homogeneous natural frequencies, we find a reduction of the SL model on the complex plane to a novel nonlinear opinion dynamics model on the real line where the inherent amplitudes model the effect of stubbornness of individuals with heterogeneous inherent opinions.
\end{itemize}

These results demonstrate new forms of synchronization and phase transitions in the SL model beyond those observed in phase-reduced models such as the Kuramoto model, and they also extends the partial results in \cite{PLE, MS}. Our analysis reveals parameter-dependent bifurcations, discontinuities, and transitions to amplitude death and coherent motion. The findings also clarify the role of amplitude dynamics in shaping global synchronization.

The remainder of the paper is structured as follows. Section~\ref{sec:preliminaries} reviews background material and provides an overview of the main results of the paper, which are then expanded and proved in the subsequent sections. Section~\ref{S:twin} treats the homogeneous two-oscillator case and proves Theorem~\ref{t:N=2,a=0}. Section~\ref{S:ahet} explores all heterogeneous two-oscillator regimes, establishing the structure seen in Figures~\ref{fig:Kur2a2<a1<0}--\ref{fig:Kur2a1>a2>0}. %Section~\ref{S:Strong} discusses strong-coupling convergence. 
Section~\ref{S:hom} proves Theorem~\ref{t:N>2synch} for $N \geq 2$ oscillators with amplitude heterogeneity and zero frequency heterogeneity. Section~\ref{S:opinion} presents Theorem~\ref{t:opinion1} for the real-valued reduction modeling nonlinear opinion dynamics.

\section{Preliminaries and Statement of Results} \label{sec:preliminaries}

The equation for a single Stuart--Landau oscillator is
\begin{equation}\label{SLo}
    \dot{z}(t)=\left(\alpha+\i\omega-|z(t)|^2\right)z(t),
\end{equation}
where $z(t)\in \C, \a,\o\in\R$ and $\i=\sqrt{-1}$ denotes the imaginary unit. For $z(t)=r(t)e^{\i\phi(t)}$, the equations for the magnitude $|z(t)|=r(t)$ and phase $\phi(t)$ of the oscillator are given by
\begin{align}
    \ddt r(t)&=\left(\a-r^2(t)\right)r(t),\\
    \ddt \phi(t)&=\omega.
\end{align}
From these equations, we notice that for $\a>0$ the oscillator converges to a stable limit cycle with amplitude $\sqrt{\a}$ and speed $\omega$. For this reason, we call $\a$ the inherent amplitude (or Hopf parameter) and $\omega$ the natural frequency. Meanwhile, for $\a<0$, the oscillator is subcritical and undergoes asymptotic amplitude death as $r \to 0$.

The coupled SL system is given by
\begin{equation}\label{eq:SL}
    \dot{z}_j=(\alpha_j+\i\omega_j-|z_j|^2)z_j+\frac{\kappa}{N}\sum_{l=1}^N(z_l-z_j),
\end{equation}
for coupling strength $\k\geq 0$, inherent amplitudes $\a_j \in \R$, natural frequencies, $\w_j\in \R$, and finitely many oscillators $j=1,...,N$. Using polar coordinates $z_j=r_je^{{\rm i}\phi_j}$ the system can be reformulated as the following coupled system for the phases $\phi_j$ and amplitudes $r_j$:
\begin{equation}\label{eq:SL-split}
\begin{aligned}
&\dot{r}_j=(\alpha_j-r_j^2)r_j+\frac{\kappa}{N}\sum_{l=1}^N(r_l\cos(\phi_l-\phi_j)-r_j),\\
&\dot{\phi}_j=\omega_j+\frac{\kappa}{N}\sum_{\ell=1}^N\frac{r_l}{r_j}\sin(\phi_l-\phi_j).
\end{aligned}
\end{equation}
In this form, the coupled SL system \eqref{eq:SL-split} can be regarded as a Kuramoto model over a heterogeneous and dynamical network, whose weights $\frac{r_l}{r_l}$ co-evolve with the phases $\phi_j$. See \cite{BGKKY-23,BVSY-21,HNP-16} and references therein for related literature about models of coupled oscillators on adaptive networks.

In order to study this fundamental ODE system and investigate the role each parameter plays in leading to synchronous outcomes, we start off by introducing the next two notions of heterogeneity for the system.

\begin{definition}[Amplitude and Frequency Heterogeneity]
    Letting $\a_j \in \R$, a measure of inherent amplitude heterogeneity is given by
    \begin{align}
        a=\max_{i,j=1,...,N}|\a_i-\a_j|.
    \end{align}
    Similarly, letting $\w_j \in \R$, a measure of heterogeneity within natural frequencies is defined as
    \begin{align}
        \g=\max_{i,j=1,...,N}|\w_i-\w_j|.\\ \nonumber
    \end{align}
\end{definition}

Historically, synchronous behavior has been linked with the homogeneity of oscillators. Indeed, the Kuramoto model is derived in a limiting regime where $a \equiv 0$. Within this regime, the relationship between coupling strength $\k$ and natural frequency heterogeneity $\g$ drives whether the system converges to a phase-locked state or not. In particular, for $\k>\k^*(\g)$, the Kuramoto system enjoys a convex gradient flow structure such that a large class of initial data converges to the phase-locked state. The SL system does not enjoy such a structure, and as will be seen in Section \ref{S:hom}, a perturbed gradient flow structure can only be recovered in the case of natural frequency homogeneity $(\g=0)$, and even still, it will be a nonconvex flow.

The main contribution of this paper is to highlight the effects of introducing amplitude heterogeneity $(a>0)$ into the Stuart-Landau system \eqref{eq:SL}. To this end, a major focus of the paper will be the case of $N=2$ oscillators. By varying the parameters $\k,a,\g,\a_j$ we can observe various continuous and discontinuous phase transitions. Of particular interest will be a new region for which phase-locking occurs for intermediate values of $\k$ regardless of the natural frequency heterogeneity $\g$. We call this phenomenon \textit{leader-driven synchronization}.

We introduce here two useful functions for observing the synchrony of the system.

\begin{definition}[Phase Difference and Average]
    Letting $\phi_j \in [-\pi,\pi)$, we define the relative phase difference by
    \begin{align}
        \Phi_{jk}=\phi_j-\phi_k
    \end{align}
    and in particular for $N=2$ we drop the indices and refer to this as
    \begin{align}
        \Phi=\phi_1-\phi_2.
    \end{align}
    Similarly, we define the phase average as
    \begin{align}
        \Psi=\frac{1}{N}\sum_{j=1}^N\phi_j.\\
        \nonumber
    \end{align}
\end{definition}

A quantity of further use with regard to the amplitude values is the ratio of amplitude variables.
\begin{definition}[Amplitude Ratios]
    Letting $r_j(t)=|z_j(t)|$, let the amplitude ratios be defined as
    \begin{align}
        R_{jk}=\frac{r_j}{r_k},
    \end{align}
    while for the case $N=2$, we drop the indices and refer to the ratio as
    \begin{align}
        R=\frac{r_1}{r_2}.
    \end{align}
\end{definition}

Last, in order to understand the asymptotic states of the system, we will define the four phenomena we witness when $N=2$, although we state the definitions for general $2\leq N<\infty$. Due to the reliance on both amplitude and phases, there are two potential states that stem from the analysis of the amplitudes, and two that stem from the phases.

\begin{definition}[Active State versus Amplitude Death]
The following two states are observed via the analysis of the amplitude variables $|z_j|=r_j$:

  \begin{itemize}
      \item We say that the system tends to Amplitude Death if for all $j=1,...,N$ each $r_j \to 0.$
      \item On the other hand, we say the oscillators remain in an Active State if again for all $j=1,...,N$, each $r_j \not\to 0$.
  \end{itemize}  
\end{definition}

\begin{definition}[Phase-Locking versus Incoherence (Periodic Orbit)]
The following two states are observed via the analysis of the phase variables $\Phi_{jk}$, $\Psi$:

  \begin{itemize}
      \item We say that the system tends to Phase-Locking if for all $j,k=1,...,N$ each $\Phi_{jk} \to c_{jk}$, a fixed constant.
      \item On the other hand, we say the oscillators have Incoherent dynamics if Phase-locking does not occur: $\Phi_{jk}\not\to c_{jk}$ for some fixed constant. In this case, for $N=2$, at times we can show that this state is further representative of a Periodic orbit characterized by $\Psi\to 0$ while $\ddt \Phi>0$.
  \end{itemize}  
\end{definition}
With all the relevant definitions in hand the rest of this section will be devoted to stating the main results of the paper.

\subsection{Two oscillator results}
In order to highlight the effects of heterogeneity, we begin with the homogeneous case of $N=2$ and $a=0$ with $\a_1=\a_2=\a\leq0$. 
\begin{align}
    \ddt z_1=(\a+\i\w-|z_1|^2)z_1+\frac{\k}{2}(z_2-z_1), \label{eq:N=2,a=0,1}\\
    \ddt z_2=(\a-\i\w-|z_2|^2)z_2+\frac{\k}{2}(z_1-z_2) \label{eq:N=2,a=0,2},
\end{align}
where we choose $\w=\w_1=-\w_2\geq 0$ for the rotational invariance of the system.

Due to $\a\leq0$, the asymptotic outcome will always be amplitude death at an exponential rate (algebraic if $\a=0$), however depending on the relationship between $\k$ and $\g$, the phases of the system can tend towards a periodic orbit, or a phase-locked state, throughout the convergence to amplitude death.

The phase diagram for \eqref{eq:N=2,a=0,1}-\eqref{eq:N=2,a=0,2}, with $\a\leq0$ can be seen in Figure \ref{fig:Kur2a<0}.

The Hopf parameter, $\a$, represents the desired amplitude of each oscillator. As is the case for an individual Stuart-Landau oscillator, $\a\leq 0$ yields a stable fixed point at $z_j=0$. There is a bifurcation as $\a$ passes from negative to positive where stable active states $(z_j\neq 0)$ arise depending on the parameters $\k$ and $\g$.

\begin{figure}[t]
    \centering
    \includegraphics[width=12cm]{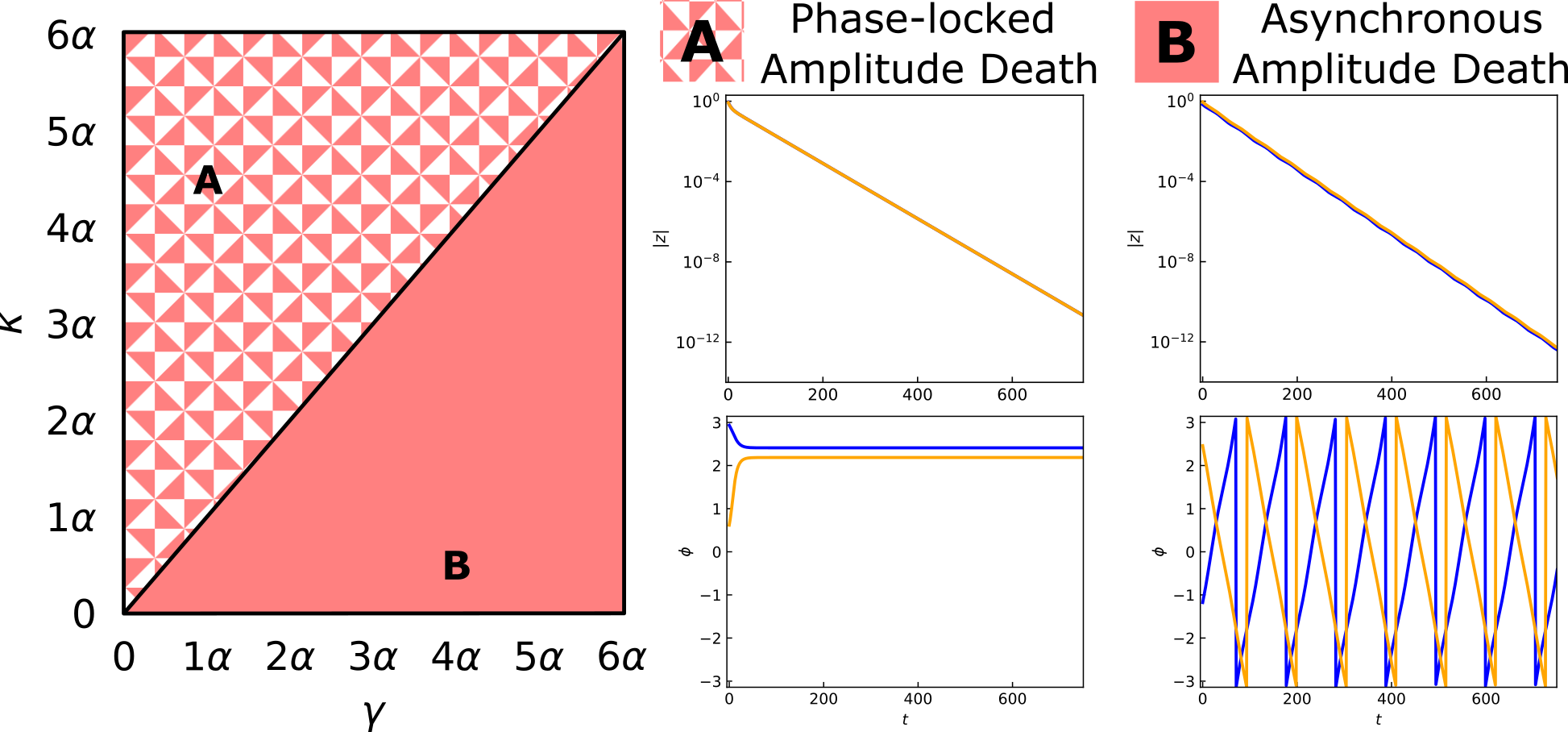}
    \caption{Phase diagram for the Stuart-Landau system with $N=2$ oscillators with identical subcritical Hopf parameter $\a_1=\a_2=\a<0$. The line $\k=\g$ is the onset of synchronization where along this line convergence to phase-locking occurs at an algebraic rate, above the line (Tessellated pattern) phase-locking occurs at an exponential rate, and below the line (Solid) the system tends towards a periodic orbit. As $\a\leq0$ both oscillators tend to Amplitude Death (Red) at an exponential rate (algebraic at $\a=0$). The letters A and B signify the dynamics for the particular choice of parameters in the phase diagram. The first row shows the amplitude death behavior as $|z|\to 0$. The second row shows the phases of each oscillator $\phi_j$ in the first case where Phase-Locking occurs, and the second giving periodic motion.
    }
    \label{fig:Kur2a<0}
\end{figure}

The asymptotic outcomes for model \eqref{eq:N=2,a=0,1}-\eqref{eq:N=2,a=0,2} with $\a>0$ are seen in the phase-diagram Figure \ref{fig:Kur2a>0}. The following theorem formalizes this for all $\a\in \R$.
% FT: The theorem is presented in an informal way and without a proof. Its proof as well as details will be the subject of section xy and specifically theorems wz.

\begin{theorem}\label{t:N=2,a=0}
Let $N=2$, $\a_1=\a_2=\a$, with $\w_1=-\w_2=\w\geq 0$. Then for any parameter configuration $\a \in \R,\k>0,\g\geq 0$ the following represents an invariant manifold:
\begin{align}
   \cM= \{z_1(t),z_2(t) \in \C: R(t)=\frac{r_1}{r_2}=\frac{|z_1|}{|z_2|}=1\}.
\end{align}
Furthermore, for any initial data $\{z_j(0)\}_{j=1}^2\in \cM$, solutions to \eqref{eq:N=2,a=0,1}-\eqref{eq:N=2,a=0,2}  converge to one of the 4 types of asymptotic states seen in Figures \ref{fig:Kur2a<0}-\ref{fig:Kur2a>0} depending on the parameter configuration of $(\a,\k,\g)$. For $\a>0$, the system remains Active for all $\k,\g<2\a$ while if $\k>2\a$ there is an Active/Amplitude Death curve $\k^*(\g)$ for which the active state is recovered above this curve. It is given by
\begin{align}\label{eq:AAcurve1}
    \k^*(\g)=\frac{4\a^2+\g^2}{4\a},  \ \g \in (2\a,\infty).
\end{align}
    If the system parameters lie above the curve $\k^*(\g)$, then the asymptotic state is Active-Phase-Locking. If the parameters are found in between $\k^*(\g)$ and above the line $\k=\g$, then the asymptotic state is Amplitude Death-Phase-Locking. If $2\a<\k<\g$, then the asymptotic state is Amplitude Death-Incoherence. The remaining parameter configuration gives Active-Incoherence. For $\a\leq 0$, the phase transition from periodic to phase-locked always occurs at $\k=\g$, while Amplitude Death always occurs at an exponential rate (algebraic at $\a=0)$.

    The manifold $\cM$ is stable and thus so are the asymptotic states derived in Figures \ref{fig:Kur2a<0}-\ref{fig:Kur2a>0}. Furthermore, for $\k>2\a$, convergence to asymptotic states for a.e. set of initial data is guaranteed.
    
\end{theorem}

The behavior of model \eqref{eq:N=2,a=0,1}-\eqref{eq:N=2,a=0,2} is similar to the behavior of the Kuramoto model of phase synchronization with $N=2$ oscillators. Indeed, the same onset of synchronization at $\k=\g$ occurs with an algebraic convergence rate occurring at the phase-transition. However, without the phase-reduction which yields the Kuramoto model, the amplitude parameter $\a \in \R$ plays a role in whether or not the oscillators remain active $r_j(t)>0$, or converge to amplitude death $r_j \to 0$. Of interesting note is that even in amplitude death regimes, the phase behavior can still be recovered, giving rise to both periodic phase behavior and phase-locking, depending on the parameters.

\begin{figure}[t]
    \centering
    \includegraphics[width=12cm]{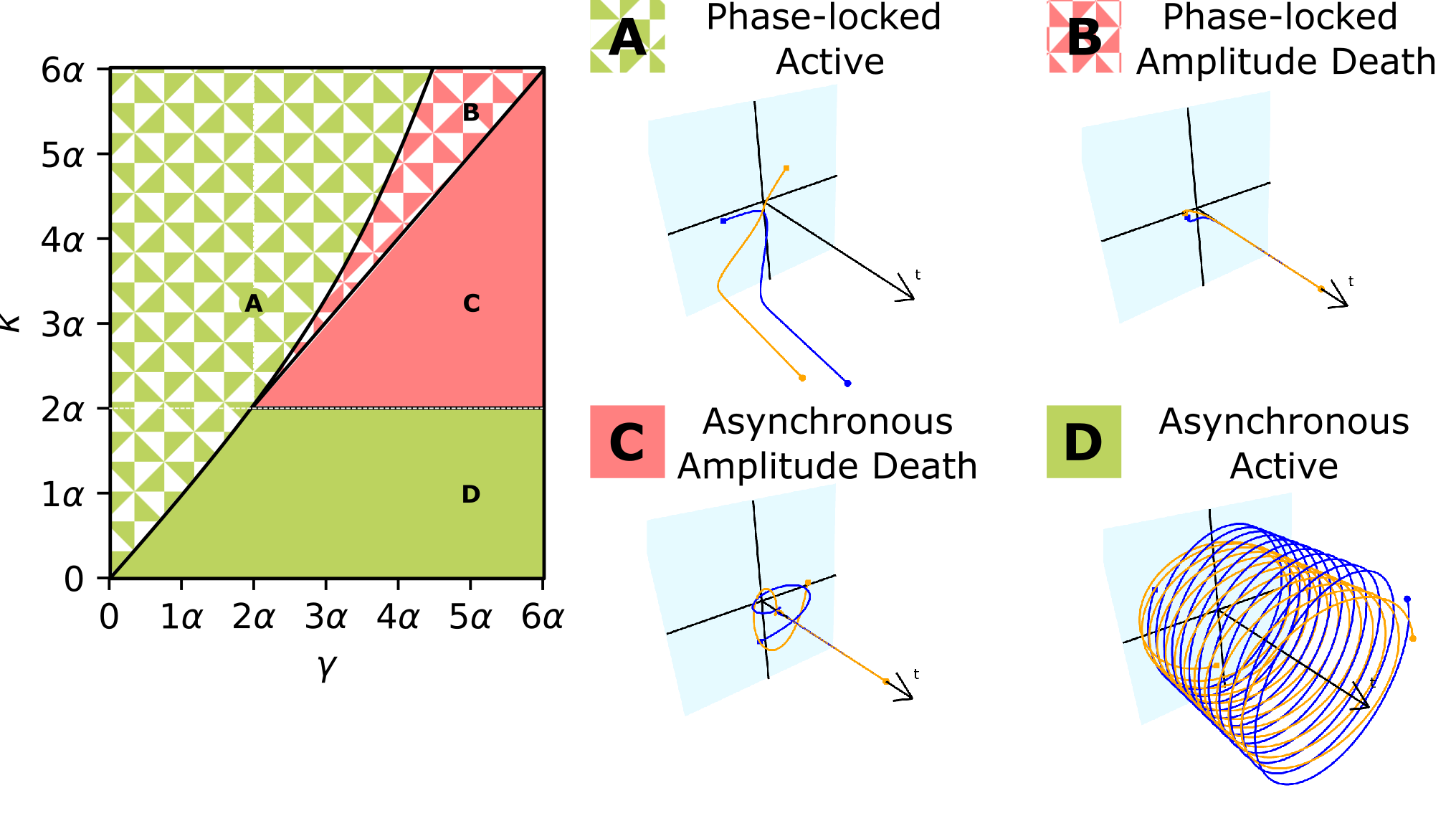}
    \caption{Phase diagram for the Stuart-Landau system with $N=2$ oscillators with identical supercritical Hopf parameter $\a_1=\a_2=\a>0$. The line $\k=\g$ is the onset of synchronization where along this line convergence to phase-locking occurs at an algebraic rate, above the line (Tessellated) phase-locking occurs at an exponential rate, and below the line (Solid) the system tends towards a periodic orbit. The curve $\k^*(\g)$ as defined in \eqref{eq:AAcurve1}, and the line $\k=2\a$ for $\g>2\a$, determine whether or not oscillators remain Active (Green) or tend to Amplitude Death (Red). The letters A,B,C,D signify a particular
choice of parameters in the phase diagram for which the dynamics are shown. }
    \label{fig:Kur2a>0}
\end{figure}

The phase-reduction process which achieves the Kuramoto model in \cite{Kur} removes the dependence upon the amplitudes of the oscillators. This, of course, is not a requirement and even in the simple case of Theorem \ref{t:N=2,a=0} we can see the effects of retaining the amplitude dependence of the Stuart-Landau system. In order to study further the effects of amplitude dependence, we introduce heterogeneity into the parameters $\a_j$. In the $N=2$ case, in order to investigate all possible configurations we begin with the $\a_1=\a_2<0$ case seen in Figure \ref{fig:Kur2a<0}. From there we leave $\a_2$ fixed and begin to increase $\a_1$. This first increase provides a discontinuous bifurcation within the system. The incoherent state seen in Figure \ref{fig:Kur2a<0} disappears entirely and the only asymptotic state is given by Amplitude Death and Phase-Locking for all $\k>0$, $\g\geq0$, as seen in Figure \ref{fig:Kur2a2<a1<0}.  

\begin{figure}[t]
    \centering
    \includegraphics[width=12cm]{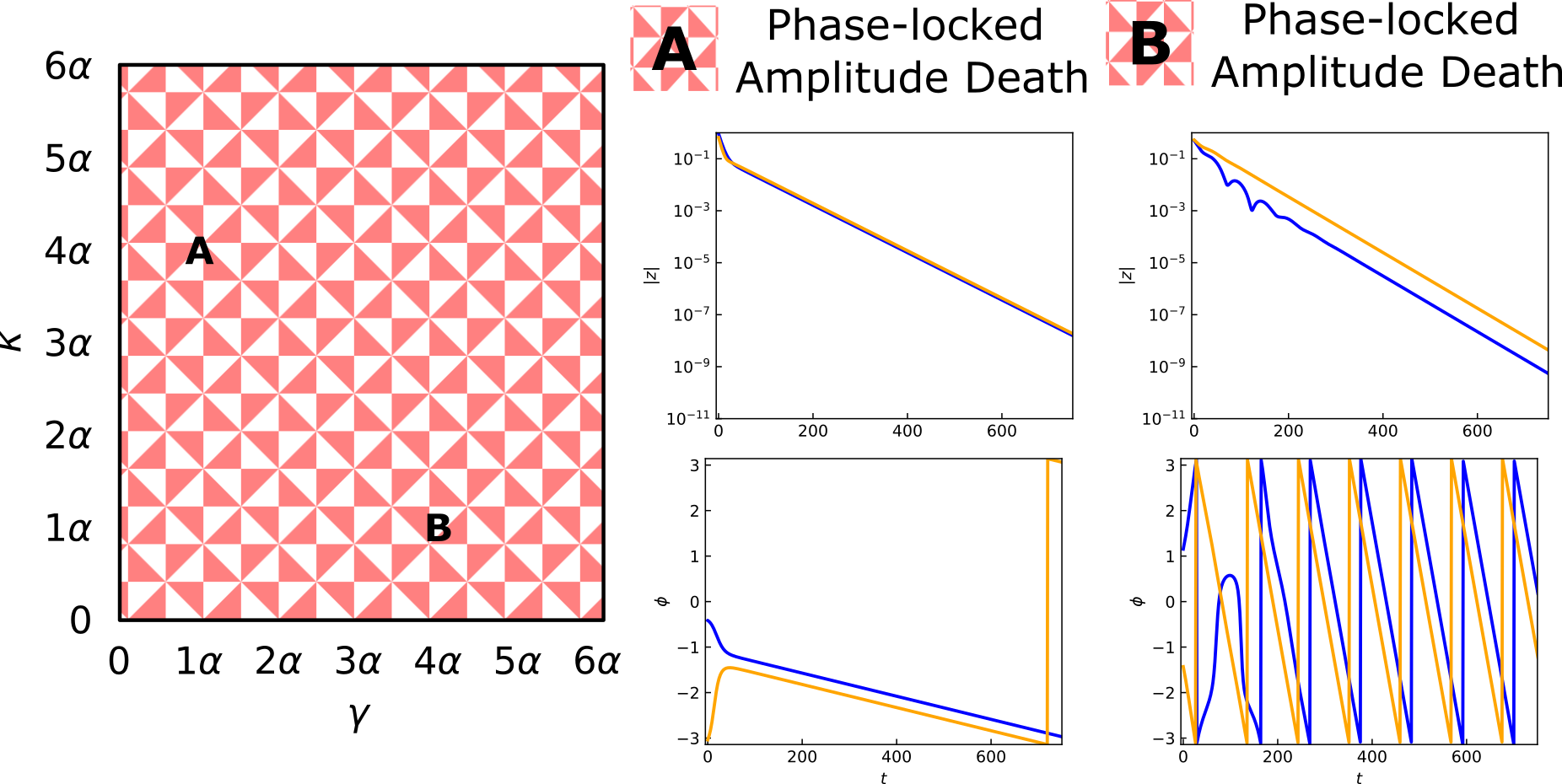}
    \caption{Phase diagram for the Stuart-Landau system with $N=2$ oscillators with nonidentical subcritical Hopf parameters $\a_2<\a_1\leq0$. The only asymptotic state is Phase-Locked--Amplitude Death (Tessellated, Red). The dynamics for the two points A and B are given where both lead to Amplitude Death and Phase-Locking, but the increased rotation speed in B can be seen.
}
    \label{fig:Kur2a2<a1<0}
\end{figure}

Continuing to increase $\a_1$ until $\a_1>0$ yields another phase transition. Once $\a_1>0$, the possibility of Active States is recovered, while still the phase behavior remains Phase-Locked. The curve which separates the Active and Amplitude Death States within Figure \ref{fig:Kur2-a2>a1>0} can be computed explicitly and is given by
\begin{align}\label{intro:curve}
     f(\a_1,\a_2,\k,\g):=\a_1+\a_2-\k+\sqrt{\frac{1}{2}(\sqrt{4a^2\g^2+(a^2-\g^2+\k^2)^2}+a^2-\g^2+\k^2)}=0, \ \ \ \k>0,\g\geq0
\end{align}
and within this regime it can be shown that
\begin{align}
    f(\a_1,\a_2,\k,0) \implies \k^*(0)&=\frac{2\a_1\a_2}{\a_1+\a_2}, \label{eq:2asym}
\end{align}
while for $\g \to \infty$ there is a horizontal asymptote at $\k=2\a_1$. Indeed, this implies that for any $\k<2\a_1$, and any $\g>0$ there exists a unique stable phase-locked state. We dub this phenomenon \textit{leader-driven synchronization} due to the fact that the larger oscillator (corresponding to $\a_1>0$) usurps the dynamics so that both oscillators oscillate close to the natural frequency $\w_1$, with $r_1^2 \sim \a_1$ and $r_2 \sim 0$. Within this regime, letting $\g \to \infty$, we see $r_1^2\to \a_1$ and $r_2\to 0$ which makes the second oscillator experience the phenomena known as Quenching \cite{KVK}. This is in direct opposition to what happens in oscillatory models, which have homogeneity of amplitudes, whether by assumption \cite{MS} or via a phase-reduction \cite{Kur}. 
\begin{figure}[t]
    \centering
    \includegraphics[width=12cm]{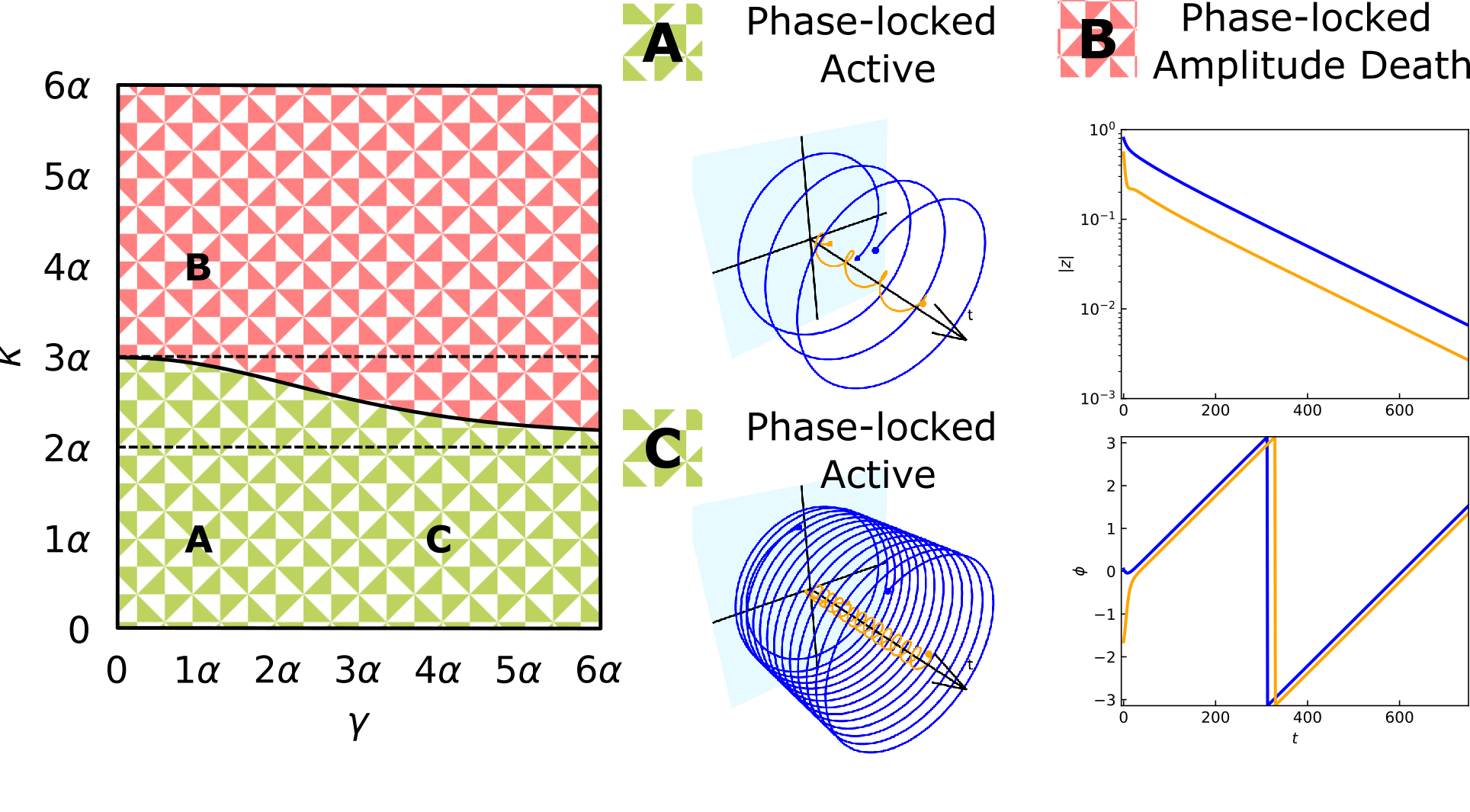}
    \caption{Phase diagram for the Stuart-Landau system with $N=2$ oscillators with nonidentical, on average subcritical, Hopf parameters $-\a_2>\a_1>0$. The phase behavior remains Phase-Locked (Tessellated), while the Active state (Green) is recovered for low coupling values such that $(\k,\g)$ is below the curve $f(\a_1,\a_2,\k,\g)=0$, and Amplitude Death (Red) above the curve. The three points A,B,C give dynamics for those particular parameter configurations.
    Notably one can see the increased rotation speed of the Phase-Locked states when comparing C (small $\k$, large $\g$) to A (small $\k$, small $\g$). Point B highlights the ability to track the phase-locking behavior through amplitude death.}
    \label{fig:Kur2-a2>a1>0}
\end{figure}

Continuing to increase $\a_1$ until $\a_1=-\a_2>0$ provides the next transition. The formulation for the Amplitude Death transition \eqref{intro:curve} remains the same but at this value we see that the domain in $\g$ for the curve is restricted to $\g>2\a_1$.
From \eqref{eq:2asym} we can see that as $\a_1\nearrow -\a_2$ the value $\k^*(0)\to \infty$. Thus, within Figure \ref{fig:Kur2a1=-a2>0}, the symmetry of $\a_1=-\a_2$ provides a vertical asymptote at $\g=2\a_1$ and a horizontal asymptote at $\k=2\a_1$. Therefore, the \textit{leader-driven synchronization} regime is preserved for $\k<2\a_1$, and a new Active Phase-Locked regime arises for weak natural frequency heterogeneity $\g<2\a_1$.

\begin{figure}[t]
    \centering
    \includegraphics[width=12cm]{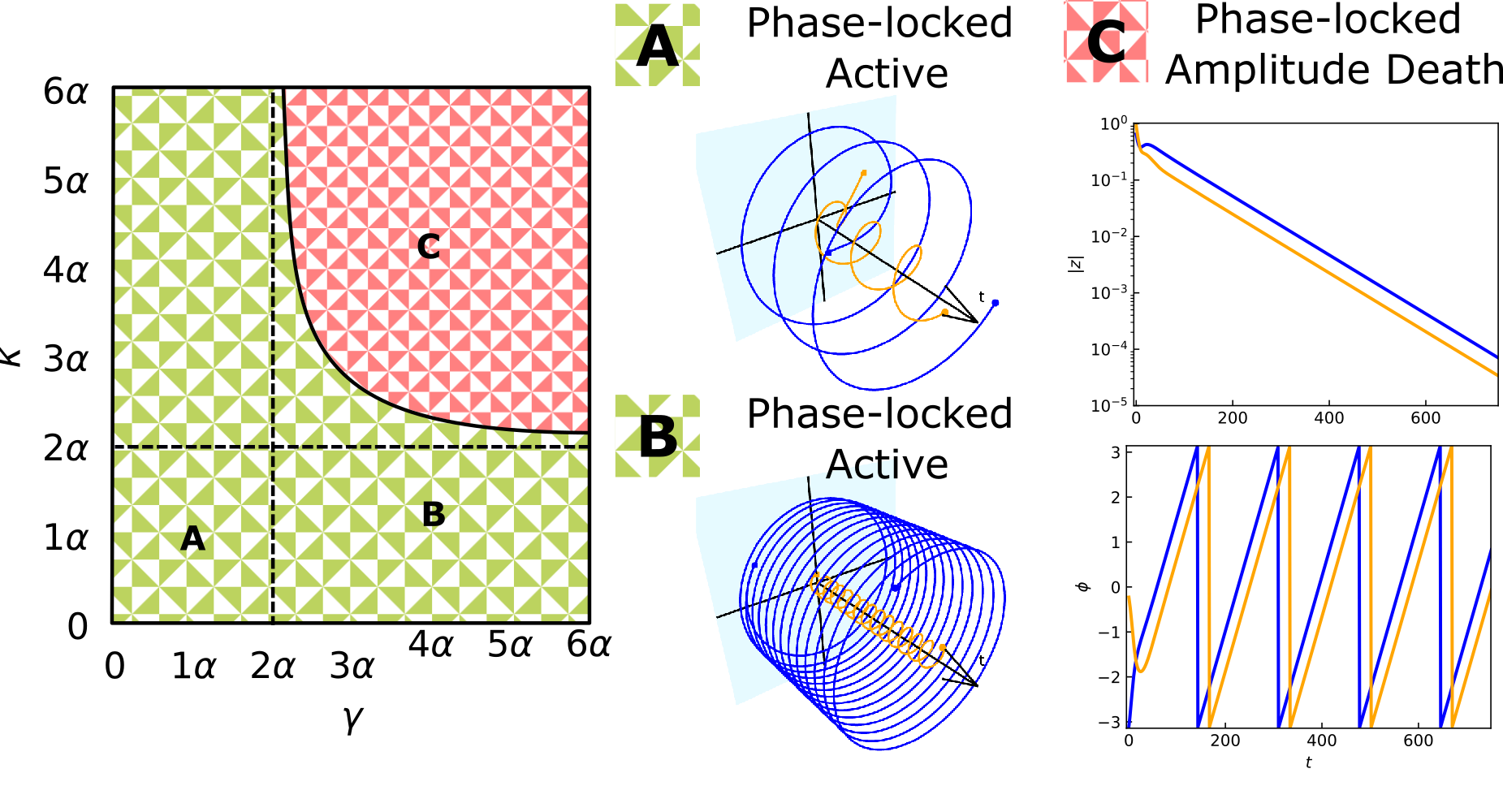}
    \caption{Phase diagram for the Stuart-Landau system with $N=2$ oscillators with nonidentical, on average critical, Hopf parameters $\a_1=-\a_2>0$. The phase behavior remains Phase-Locked (Tessellated), while the Active state (Green) is recovered for $(\k,\g)$ below the curve $f(\a_1,\a_2,\k,\g)=0$, and Amplitude Death (Red) above the curve. The points A,B,C give dynamics for the chosen parameters in the phase diagram. Note the qualitative differences: (A-small $\k$, small $\g$) slow rotations with medium difference in amplitudes, (B-small $\k$, large $\g$) fast rotations, large difference in amplitudes, (C-large $\k$, large $\g$) fast amplitude death.
    }
    \label{fig:Kur2a1=-a2>0}
\end{figure}

From $\a_1=-\a_2>0$, we begin to increase $\a_2$ which immediately gives the next transition at $\a_1>-\a_2\geq 0$. Breaking the symmetry further shrinks the Amplitude Death regime so that $f(\a_1,\a_2,\k,\g)=0$ no longer can be viewed as a function $\k(\g)$ as it becomes multi-valued. However, we can still see in Figure \ref{fig:Kur2a1>-a2>0} the persistence of the horizontal asymptote at $\k=2\a_1$, and thus the preservation of the \textit{leader driven synchronization} regime. The weak natural frequency heterogeneity regime $\g<\g'$ for $\g'>2\a_1$ (computed in Section \ref{ss:I2}) grows, while the limiting behavior in the vertical direction can be computed to be quadratic $\k \sim \g^2$, so that for any $\g>\g'$. There exist two values $\k_1(\g)>\k_2(\g)$ that satisfy the Amplitude Death curve \eqref{intro:curve}. If $\k>\k_1(\g)$, then the Active Phase-Locked state returns in the strong coupling regime.

\begin{figure}[t]
    \centering
    \includegraphics[width=12cm]{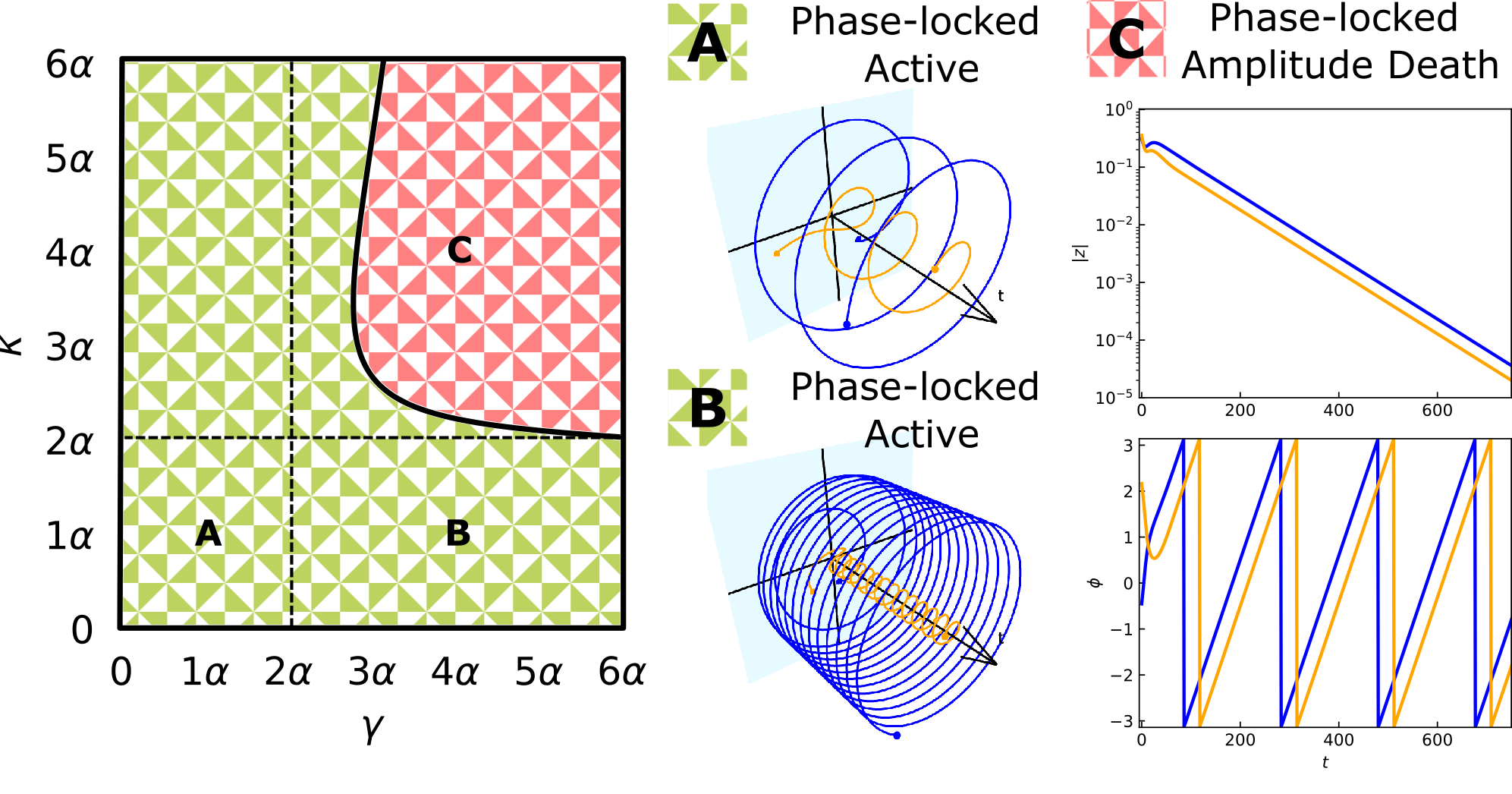}
    \caption{Phase diagram for the Stuart-Landau system with $N=2$ oscillators with nonidentical, supercritical on average, Hopf parameters $\a_1>-\a_2\geq 0$. The phase behavior remains Phase-Locked (Tessellated), while the Amplitude Death regime (Red) above the curve $f(\a_1,\a_2,\k,\g)=0$ shrinks further, compared to the Active state (green). The points A,B,C give dynamics for the chosen parameters in the phase diagram. Note the qualitative differences: (A-small $\k$, small $\g$) slow rotations with medium difference in amplitudes, (B-small $\k$, large $\g$) fast rotations, large difference in amplitudes, (C-large $\k$, large $\g$) fast amplitude death.}
    \label{fig:Kur2a1>-a2>0}
\end{figure}

As $\a_2$ becomes positive, we enter the parameter configuration $\a_1>\a_2>0$, which once more recovers the incoherent state for low coupling strengths $\k$. This curve can be explicitly computed as well for $\g$ as a function of $\k$:
\begin{align}\label{intro:curve2}
    \g^*(\k)=\k\left(\frac{\a_1+\a_2-\k}{\sqrt{(2\a_1-\k)(2\a_2-\k)}}\right), \ \ \k \in (0,2\a_2).
\end{align}
The \textit{leader-driven synchronization} regime is now moved up to intermediate values of  $\k \in [2\a_2,2\a_1]$ where we can again see that a unique Active Phase-Locked state exists for any $\g>0$, wedged in between the two curves \eqref{intro:curve2} and \eqref{intro:curve}. The dotted line seen in Figure \ref{fig:Kur2a1>a2>0} serves to delineate the smooth qualitative transition from \textit{leader-driven synchronization} to standard phase-locking behavior. This qualitative distinction can be seen via analysis of the $|z|$ order parameter.

Of further importance for understanding Figure \ref{fig:Kur2a1>a2>0} is the Northeast region, which undergoes Amplitude Death and Phase-Locking. Indeed, in all configurations when $\a_1\neq\a_2$ we see the correspondence of Amplitude Death always experiencing Phase-Locking, in contrast to Figure \ref{fig:Kur2a>0}. As $\a_2 \nearrow \a_1$ we see two transitions. This can be seen by comparing Figure \ref{fig:Kur2a1>a2>0} and Figure \ref{fig:Kur2a>0}, as well as looking at the limit of $f(\a_1,\a_2,\k,\g)$. First, the \textit{leader-driven synchronization} region shrinks and disappears as at $\a_1=\a_2$ there can be no leader. Second, there is a discontinuous bifurcation where an entire region of Amplitude Death with Incoherent dynamics (Periodic Orbit) emerges. Indeed, the curve $\g^*(\k)$ which determines the phase-transition from Incoherence to Phase-Locking within Figure \ref{fig:Kur2a1>a2>0}, jumps to the straight line $\k=\g$ seen in Figure \ref{fig:Kur2a>0} exactly when $\a_1=\a_2=\a$.

\begin{figure}[t]
    \centering
    \includegraphics[width=12cm]{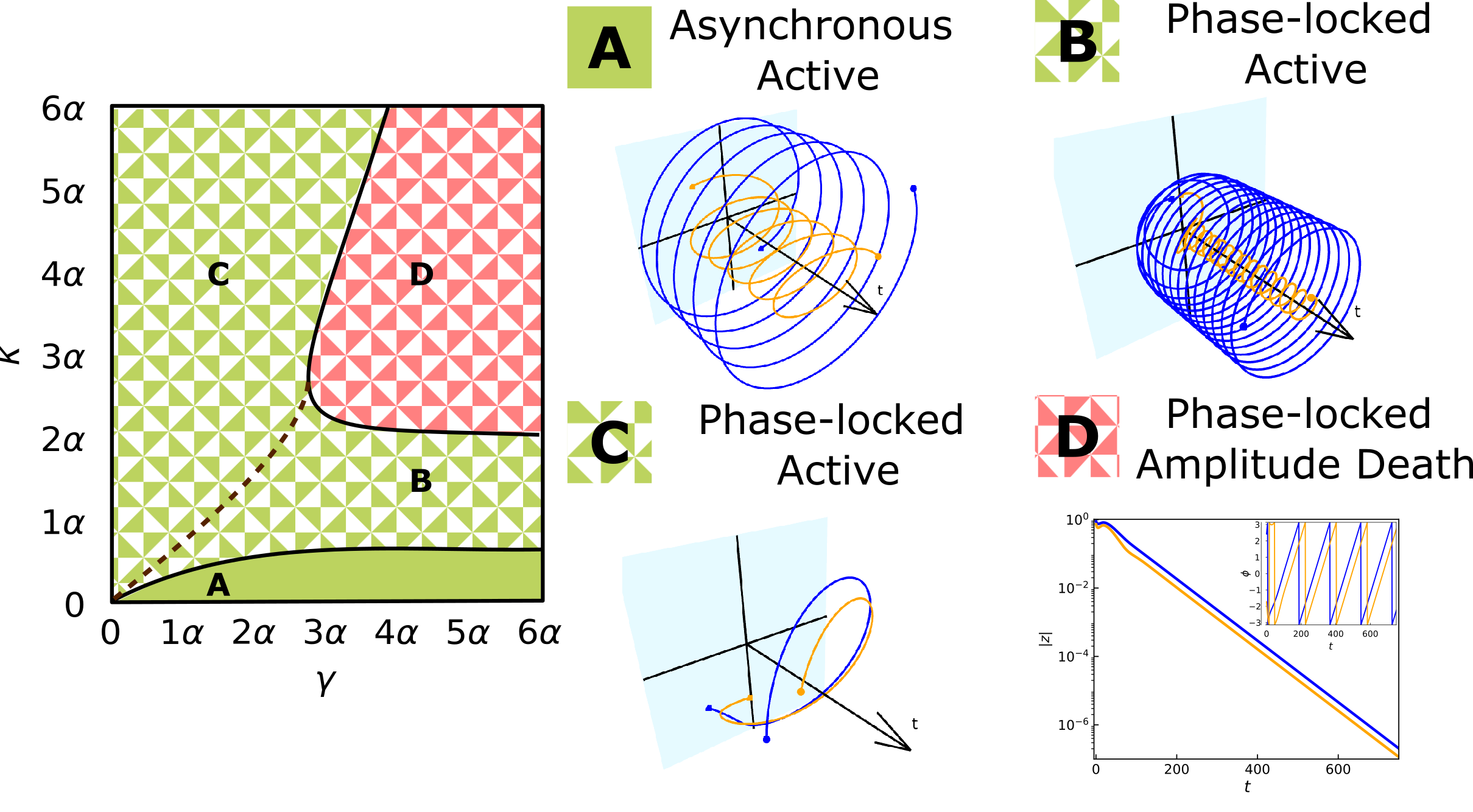}
    \caption{Phase diagram for the Stuart-Landau system with $N=2$ oscillators with nonidentical supercritical Hopf parameters $\a_1>\a_2>0$. The phase behavior is incoherent (Solid) for weak coupling values, below the lower curve and Phase-locked (Tessellated) above. The Amplitude Death regime (Red) can still be found in the Northeast region above the curve $f(\a_1,\a_2,\k,\g)=0$, while the system remains Active (Green) below. Intermediate values of $\k \in [2\a_2,2\a_1]$ beyond the dotted line correspond to the new area of Active, Phase-Locking named \textit{leader-driven synchronization}. The points A,B,C,D give dynamics for the chosen parameters in the phase diagram. Note the qualitative differences: (A-small $\k$) oscillators rotate in opposite directions with medium difference in amplitudes, (B-intermediate $\k$, large $\g$) fast rotations with large difference in amplitudes, (C-large $\k$, small $\g$) slow rotations, small difference in amplitudes, (D-large $\k$, large $\g$) fast amplitude death.
    }
    \label{fig:Kur2a1>a2>0}
\end{figure}

\subsection{Finitely many oscillator case}
Beyond the two oscillator case, we address the general $2\leq N<\infty$ case with heterogeneous amplitude parameters $a>0$, and homogeneous natural frequency parameters $\g=0$. The main result for this case is for a large class of initial data known as \textit{sectorial solutions}, complete frequency synchronization $\Phi_{jk}\to 0$ with all amplitudes Active, but differing in asymptotic value, or all amplitudes converging to Amplitude Death at the same exponential rate is guaranteed. The results are provided in the following theorem.

\begin{theorem}\label{t:N>2synch}
    Let $2\leq N<\infty$, $\a_j \in \R$ and $\w_j\equiv 0$ for each $j=1,...,N$. Let $\{z_j\}_{j=1}^N$ be strictly sectorial solutions to \eqref{SLhom}, then for all $\k>0$, $\max_{j,k} \Phi_{jk}\to 0$ exponentially fast, and the amplitude dynamics have each $r_j \to r_j^{\infty}$, where one of the two cases can occur:
    \begin{itemize}
        \item Amplitude death: $r_j^{\infty}=0$ for all $j=1,...,N$ if all the following conditions hold
        \begin{itemize}
        \item $\sum_{j=1}^N\a_j<0$,
        \item $\k>\a_j$, $\forall j=1,...,N$
        \item $\sum_{j=1}^N \frac{\a_j}{\k-\a_j}<0$
        \end{itemize}
        \item Active: $r_j^{\infty}>0$ for each $j=1,...,N$ if any of the above conditions fails.
    \end{itemize}
    Further, $r_{j,\infty}\leq r_{j+1,\infty}$ for all $j$ and if all $\a_j>0$, then $r_j^{\infty} \in [\min_j \sqrt{\a_j},\max_j \sqrt{\a_j}]$.
\end{theorem}

The above theorem proves that in the Active case, all $z_j$ converge to the same ray emanating from the origin at an exponential rate. Meaning the dynamics are essentially one dimensional plus an exponentially decaying term. Considering the system reduced to the one-dimensional setting gives rise to a nonlinear model of opinion dynamics that has been treated in \cite{LRS,RT}. Within said works the parameter and initial data configurations were such that $\a_j>0$ and $x_j(0)>0$. However, the current study of Stuart-Landau oscillators for which it is natural to consider Hopf-parameter $\a_j<0$ has inspired us to further investigate the one-dimensional reduction in this case as well. Allowing for both $\a_j\leq 0$ and initial data $x_j(0)\in \R$ gives rise to previously not seen asymptotic states including stable disagreement, and recovery of a consensus state despite the heterogeneous nonlinear stubbornness effects. Let us rewrite the equations for the Real-line setting and give a few definitions in order to state the main result for the reduced system.
\begin{align}\label{eq:SLopinion}
    \ddt x_j=(\a_j-x_j^2)x_j+\frac{\k}{N}\sum_{l=1}^N(x_l-x_j), \ \ \k>0, \  \a_j,x_{j,0}\in \R, \ \ j=1,...,N.
\end{align}

\begin{definition}[Disagreement, Compromise, and Consensus]
    The asymptotic states $x_j^{\infty} \in \R$ of \eqref{eq:SLopinion} can be characterized as
    \begin{itemize}
        \item Disagreement: If there exists $j,k$ such that $x_j^{\infty}<0$ and $x_k^{\infty}>0$.
        \item Compromise: If for all $j=1,..,N$, $x_j^{\infty}>0$ $(x_j^{\infty}<0)$ and there exists $k,l$ such that $x_k^{\infty}\neq x_l^{\infty}$.
        \item Consensus: If for all $j=1,...,N$, $x_j^{\infty}\equiv c$.
            \begin{itemize}
                \item Balanced Consensus: if $c=0$
            \end{itemize}
    \end{itemize}
\end{definition}

The following theorem highlights the new results for the Stuart-Landau model on the Real line.
\begin{theorem}\label{t:opinion1}
    Let $2\leq N<\infty$, $\a_j \in \R$ and further suppose $\w_j\equiv 0$. If $z_j(0)=x_{j,0}\in \R$, then the dynamics preserve this fact.
    
    If $a=0$, so that $\a_j\equiv \a$ then the following can occur:
    \begin{itemize}
        \item Weak coupling: There exists a $\k^*>0$ such that for $0<\k<\k^*$ there exists both stable Disagreement and stable Active Consensus states.
        \item Strong coupling: There exists a $\k^{**}\geq\k^*>0$ such that for $\k>\k^{**}$, only one of the following holds:
        \begin{itemize}
            \item Active Consensus: $x_j^{\infty}\equiv \sqrt{\a}$ or $x_j^{\infty}\equiv-\sqrt{\a}$ if $\a>0$.
            \item Balanced Consensus: $x_j^{\infty}\equiv 0$ if $\a\leq 0$.
        \end{itemize}
    \end{itemize}

If $a>0$, so that there exist $j,k$ where $\a_j\neq \a_k$, then the following can occur:
\begin{itemize}
        \item Weak coupling: There exists a $\k^*>0$ such that for $0<\k<\k^*$ there exists both stable Disagreement and stable Compromise states.
        \item Strong coupling: There exists a $\k^{**}\geq\k^*>0$ such that for $\k>\k^{**}$, only one of the following holds:
        \begin{itemize}
            \item Active Compromise: $x_1^{\infty}\geq ...\geq x_N^{\infty}>0$ or $x_1^{\infty}\leq ...\leq x_N^{\infty}<0$ if $\sum_{j=1}^N \a_j>0$.
            \item Balanced Consensus: $x_j^{\infty}\equiv 0$ if $\sum_{j=1}^N \a_j<0$.
        \end{itemize}
    \end{itemize}

In every one of the above cases, the dynamics are governed by a gradient flow and hence converge to a fixed point.
\end{theorem}

As a model of opinion dynamics, these results represent intuitive outcomes for cooperatively interacting stubborn agents. For weak coupling strengths, the stubborn nonlinear forcing dominates and if the initial data is close to disagreement fixed points, then this is preserved, while for initial data close to compromise/consensus states again this will be conserved instead. Meanwhile, if the coupling is strong enough, then the ability to disagree disappears via a phase transition and the system can only converge to Compromise or Consensus states depending on the parameters $\a_j$.

The remainder of the paper is outlined as follows. Section \ref{S:twin} will be dedicated to the $N=2$ homogeneous case ($a=0$) and the proof of Theorem \ref{t:N=2,a=0}. Section \ref{S:ahet} will provide existence, uniqueness, and stability of all states depicted in Figures \ref{fig:Kur2a2<a1<0}-\ref{fig:Kur2a1>a2>0} for the $N=2$ heterogeneous cases ($a>0$). %Section \ref{S:Strong} will provide convergence results for the strong coupling $N=2$ regimes. 
Section \ref{S:hom} will treat the general $N\geq2$ case with heterogeneity in amplitudes ($a\geq 0$) and homogeneity in natural frequencies ($\g=0$), proving Theorem \ref{t:N>2synch}. Section \ref{S:opinion} considers the one-dimensional reduction to the real line for the $N\geq2$ system studied in Section \ref{S:hom} which gives rise to a model of opinion dynamics, proving Theorem \ref{t:opinion1}.

\section{The twin oscillators case}\label{S:twin}
As a first investigation of coupled Stuart-Landau oscillators we begin with the simplest regime of $N=2$ oscillators with $\a_1=\a_2=\a$ so that the two oscillators have identical inherent amplitude parameters. Further, for the rotational invariance and symmetry we pick $\o_1=-\o_2=\o\geq 0$. This setting will prove to be the most similar to the Kuramoto model for $N=2$ oscillators, however differences still arise due to the amplitude dependence of the Stuart-Landau oscillators. The main work of this section will be the proof of Theorem \ref{t:N=2,a=0}, which we restate here.

\begin{theorem}\label{2t:N=2,a=0}
Let $N=2$, $\a_1=\a_2=\a$, with $\w_1=-\w_2=\w\geq 0$. Then for any parameter configuration $\a \in \R,\k>0,\g\geq 0$ the following represents an invariant manifold:
\begin{align}
   \cM= \{z_1(t),z_2(t) \in \C: R(t)=\frac{r_1}{r_2}=\frac{|z_1|}{|z_2|}=1\}.
\end{align}
Furthermore, for any initial data $\{z_j(0)\}_{j=1}^2\in \cM$, solutions to \eqref{eq:N=2,a=0,1}-\eqref{eq:N=2,a=0,2}  converge to one of the 4 types of asymptotic states seen in Figures \ref{fig:Kur2a<0}-\ref{fig:Kur2a>0} depending on the parameter configuration of $(\a,\k,\g)$. For $\a>0$, the system remains Active for all $\k,\g<2\a$ while if $\k>2\a$ there is an Active/Amplitude Death curve $\k^*(\g)$ for which the active state is recovered above this curve. It is given by
\begin{align}\label{eq:AAcurve1}
    \k^*(\g)=\frac{4\a^2+\g^2}{4\a},  \ \g \in (2\a,\infty).
\end{align}
    If the system parameters lie above the curve $\k^*(\g)$, then the asymptotic state is Active-Phase-Locking. If the parameters are found in between $\k^*(\g)$ and above the line $\k=\g$, then the asymptotic state is Amplitude Death-Phase-Locking. If $2\a<\k<\g$, then the asymptotic state is Amplitude Death-Incoherence. The remaining parameter configuration gives Active-Incoherence. For $\a\leq 0$, the phase transition from periodic to phase-locked always occurs at $\k=\g$, while Amplitude Death always occurs at an exponential rate (algebraic at $\a=0)$.

    The manifold $\cM$ is stable and thus so are the asymptotic states derived in Figures \ref{fig:Kur2a<0}-\ref{fig:Kur2a>0}. Furthermore, for $\k>2\a$, convergence to asymptotic states for a.e. set of initial data is guaranteed.
    
\end{theorem}
\begin{proof}
Let us begin with the invariance of the manifold $\cM$.
As we only have two oscillators we rewrite the equations here.
\begin{align}
    \ddt z_1&=(\a+\i\w-|z_1|^2)z_1+\frac{\k}{2}(z_2-z_1), \label{2eq:N=2,a=0,1}\\
    \ddt z_2&=(\a-\i\w-|z_2|^2)z_2+\frac{\k}{2}(z_1-z_2).\label{2eq:N=2,a=0,2}
\end{align}
The equations for the amplitudes and phases are similarly given by
\begin{align}
    \ddt r_1&=(\a-r_1^2)r_1+\frac{\k}{2}(\cos(\phi_1-\phi_2)r_2-r_1),\\
    \ddt r_2&=(\a-r_2^2)r_2+\frac{\k}{2}(\cos(\phi_1-\phi_2)r_1-r_2),\\
    \ddt \phi_1&=\w+\frac{\k}{2}\frac{r_2}{r_1}\sin(\phi_2-\phi_1),\\
    \ddt \phi_2&=-\w+\frac{\k}{2}\frac{r_1}{r_2}\sin(\phi_1-\phi_2).
\end{align}
As we are only looking at two oscillators we track the average phase, $\Psi=\frac{1}{2}(\phi_1+\phi_2)$, the difference in the phases, $\Phi=\phi_1-\phi_2$,
and the ratio of the amplitudes, $R:=\frac{r_1}{r_2}$. Further with $\g:=\w_1-\w_2=2\w$ we get the following equations
\begin{align}\label{eq:2rat}
    \ddt R=R(r_2^2-r_1^2)+\frac{\k}{2}\cos(\Phi)\left(1-R^2\right).
\end{align}
\begin{align}\label{eq:Phi}
    \ddt \Phi=\g-\frac{\k}{2}\sin(\Phi)\left(\frac{1}{R}+R\right)
\end{align}
\begin{align}\label{eq:Psi}
    \ddt \Psi=\frac{\k}{4}\sin(\Phi)\left(R-\frac{1}{R}\right)
\end{align}
 Then it is clear that $R=1$, yields an invariant manifold for the system.

 Let us prove that within the manifold $\cM$, convergence to each of the asymptotic outcomes dictated by Figures \ref{fig:Kur2a<0}-\ref{fig:Kur2a>0} are guaranteed.

We begin with the phase behavior.

 Let $\{z_j(0)\}_{j=1}^2\in \cM$. Then as $R=1$, we have $\ddt \Psi=0$ by \eqref{eq:Psi}. Therefore without loss of generality let us assume $\Psi(0)=0$ so that $\phi_1=-\phi_2$ and $\Phi=2\phi_1$. Then \eqref{eq:Phi} becomes
 \begin{align}\label{eq:PT}
     \ddt \Phi=\gamma-\kappa\sin(\Phi).
 \end{align}
 The phase transition is now clear. For $\k<\gamma$ we have $\ddt \Phi>0$ and it is impossible for a fixed point of \eqref{eq:PT} to exist and the dynamics are periodic in $\Phi$. While for $\k>\gamma$ the value $\Phi=\arcsin(\frac{\g}{\k})$ is an asymptotically stable fixed point. This provides the line $\k=\g$ in both Figures \ref{fig:Kur2a<0} and \ref{fig:Kur2a>0}. At the line $\k=\g$ we see that $\Phi \nearrow \frac{\pi}{2}$ at an algebraic rate.\\

 To see the amplitude behavior first let $\a\leq 0$ to complete Figure \ref{fig:Kur2a<0}. As $R=1$ we need only observe one amplitude $r_1$,
 \begin{align}
     \ddt r_1&=(\a-r_1^2)r_1+\frac{\k}{2}(\cos(\Phi)r_2-r_1),\\
     &=(\a-r_1^2+\frac{\k}{2}(\cos(\Phi)-1))r_1,\\
     &\leq (C-r_1^2)r_1
 \end{align}
 where $C$ is an upper bound on $\a+\frac{\k}{2}(\cos(\Phi)-1)$. If $\a<0$, then $C<0$ and if $\a=0$, then $C=0$ if and only if $\cos(\Phi)=1$, which can only occur in the special case of $\g=0$ as well. Thus if $\a<0$ or if $\a=0$ and $\g>0$, then $r_j \to 0$ exponentially fast by Gr\"onwall's inequality. In the case $\a=0, \g=0$ the inequality becomes
 \begin{align}
     \ddt r_1\leq -r_1^3,
 \end{align}
 in which case amplitude death still occurs, but at an algebraic rate $r_j \sim t^{-1/2}$.

 Now for $\a>0$, in order to complete Figure \ref{fig:Kur2a>0}, we separate into several cases.
 
 \textit{Case 1:} Let $0\leq \g < 2\a$. This corresponds to the vertical rectangular portion of Figure \ref{fig:Kur2a>0} where we will show the oscillators always remain active.
 
  \textit{Case 1a:} Let $\k<\g<2\a$.
We get the same equation for the amplitude of an oscillator
\begin{align}\label{eq:rHopf}
    \ddt r_1&=(\a-r_1^2+\frac{\k}{2}(\cos(\Phi)-1))r_1,
\end{align}
we see that since $\k<\g$, the term $\cos(\Phi)$ will behave periodically according to $\dot{\Phi}=\g-\k\sin(\Phi)$, with average value of zero, $\langle \cos(\Phi)\rangle=0$,  due to the symmetry of cosine and the symmetry of the dynamics  around the points $\Phi=\frac{\pi}{2}$ and $\Phi=\frac{3\pi}{2}$. 

In order to extract the active state within this regime, we observe that equation \eqref{eq:rHopf} corresponds to a Stuart-Landau oscillator with time periodic Hopf-parameter given by $\a + \frac{\k}{2}(\cos(\Phi)-1)$. Thus the system will remain active if on average this parameter remains positive. Thus the time-average Hopf-parameter is
\begin{align}
    \a+\frac{\k}{2}\left(\langle \cos(\Phi)\rangle-1\right)=\a-\frac{\k}{2}>0,
\end{align}
since $\k<\g<2\a$.

\textit{Case 1b:} Let $\g=\k<2\a$.

Then the amplitude equation becomes
  \begin{align}
     \ddt r_1&=(\a-r_1^2-\frac{\k}{2})r_1 + A(t),
\end{align}
 where $A(t) \to 0$ at an algebraic rate. As $\k<2\a$, we again see that the system is bounded below by the equation for a Hopf-bifurcation with $\a-\frac{\k}{2}>0$ and hence the oscillator must remain active. In particular the limiting amplitude value is given by $r_j^2 \to \a-\frac{\k}{2}$ at an algebraic rate.

 \textit{Case 1c:} Let $\g<\k$.

 Then due to the phase-locking at an exponential rate, the amplitude equation becomes
  \begin{align}
     \ddt r_1&=(\a-r_1^2+\frac{\k}{2}\left(\sqrt{1-\left(\frac{\g}{\k}\right)^2}-1\right))r_1 + E(t),\\
     &=(\a-r_1^2+\frac{1}{2}\left(\sqrt{\k^2-\g^2}-\k\right))r_1 + E(t),\label{eq:Hopf1}
\end{align}
where $E(t)$ is an exponentially decaying quantity. In order to see that the system always remains active within this regime, we differentiate the Hopf-bifurcation value with respect to the coupling strength.
\begin{align}\label{eq:diffk}
    \frac{d}{d\k}\left(\a+\frac{1}{2}\left(\sqrt{\k^2-\g^2}-\k\right)\right)=\frac{1}{2}\left(\frac{\k-\sqrt{\k^2-\g^2}}{\sqrt{\k^2-\g^2}}\right).
\end{align}
From \eqref{eq:diffk} we can see that the Hopf-bifurcation parameter of \eqref{eq:Hopf1} is monotonically increasing as $\k$ increases away from the value $\k=\g$. This implies that within the regime $\g<2\a$, if $\k\geq \g$, then the system converges to the active phase-locked state seen in Figure \ref{fig:Kur2a>0}.\\

\textit{Case 2:} Let $\g>2\a$. This corresponds to the other vertical rectangular half of the phase diagram Figure \ref{fig:Kur2a>0}.

\textit{Case 2a:} Let $\k<2\a<\g$.

The same argument as \textit{Case 1a} provides the Active state.

\textit{Case 2b:} Let $2\a<\k<\g$.

For this case, as $\k>2\a$ we cannot make the same conclusion immediately as we did in Case 1a and 2a. However, the same initial computations hold so that we can view the time-average Hopf-parameter as
\begin{align}
    \a+\frac{\k}{2}\left(\langle \cos(\Phi)\rangle -1\right)=\a-\frac{\k}{2}<0.
\end{align}
Hence we have Amplitude Death in this regime.

\textit{Case 2c:} Let $\g\leq \k< \frac{4\a^2+\g^2}{4\a}$.

As $\k \geq \g$ we have phase-locking so that the amplitude equation becomes
\begin{align}
     \ddt r_1&=(\a-r_1^2+\frac{\k}{2}\left(\sqrt{1-\left(\frac{\g}{\k}\right)^2}-1\right))r_1 + E(t),
\end{align}
where $E(t) \to 0$ exponentially fast for $\k>\g$ and at an algebraic rate for $\k=\g$. From this we can see that the Hopf-bifurcation parameter occurs at the value 
\begin{align}\label{eq:Hopf2}
    \a^*(\k)=\frac{\k}{2}\left(1-\sqrt{1-\left(\frac{\g}{\k}\right)^2}\right).
\end{align}
Solving \eqref{eq:Hopf2} for $\k$ yields the condition $\k^*=\frac{4\a^2+\g^2}{4\a}$. Then equation \eqref{eq:diffk} again provides the monotonicity in $\k$ implying that for $\g\leq \k< \frac{4\a^2+\g^2}{4\a}$ the system decays to amplitude death.

\textit{Case 2d:} Let $\g< \frac{4\a^2+\g^2}{4\a}<\k$.
The same argument for \textit{Case 2c} except we are on the other side of the Hopf-bifurcation and we conclude the convergence to the active state.\\

\textit{Stability of $\cM$}.\\

Again, we separate into two cases, this time differentiating between the relationship of $\k$ and $\a$.

\textit{Case 1:} $\k>2\a$.

Let us define a new variable $l=r_1^2-r_2^2$. This value appears in the equation for the ratio $R$, \eqref{eq:2rat}. Indeed, $R=1$ is equivalent to $l=0$ within the Active regime so these quantities are closely related. Computing the derivative of $l$ yields
\begin{align}\label{eq:l1}
    \ddt l=l(2\a-\k-2(r_1^2+r_2^2)).
\end{align}
Indeed, for $\k>2\a$, we can immediately conclude by Gr\"onwall's inequality that $l \to 0$. Now this is not quite enough to guarantee that $R\to 1$ as there exist Amplitude Death regimes when $\k>2\a$ where it is feasible that $r_1,r_2\to 0$, but $R\to R^*\neq 1$. However, letting $Z=\frac{z_1}{z_2}$ we can extract the following Riccati equation
    \begin{align}\label{eq:Z}
        \ddt Z&=\frac{\k}{2}\left(1+\frac{2}{\k}(\i\g-l)Z-Z^2\right),\\
        &=\frac{\k}{2}\left(1+\frac{2}{\k}\i\g Z-Z^2\right)+E(t)
    \end{align}
    where $E(t)$ is an exponentially decaying term from the fact that $l \to 0$ exponentially fast. Indeed, the equation without the exponential perturbation,
    \begin{align}\label{eq:Y}
        \ddt Y=\frac{\k}{2}\left(1+\frac{2}{\k}\i\g Y-Y^2\right),
    \end{align}
    can be explicitly solved, when $\k<\g$, the solution is periodic with $R\to1$ while for $\k>\g$ we have 
    \begin{align}
    Y(t)=\frac{Y_{\infty}+\overline{Y}_{\infty}\frac{Y(0)-\overline{Y}_{\infty}}{Y(0)+\overline{Y}_{\infty}}e^{-\sqrt{\k^2-\g^2}t}}{1-\frac{Y(0)-Y_{\infty}}{Y(0)+Y_{\infty}}e^{-\sqrt{\k^2-\g^2}t}},
    \end{align}
    where $Y_{\infty}=\sqrt{1-\frac{\g^2}{\k^2}}+\i\frac{\g}{\k}$. By Duhamel's Principle, solutions to the original equation \eqref{eq:Z} will also exponentially converge to $Y_{\infty}.$ If $\k=\g$, then the solution to \eqref{eq:Y} is given by
    \begin{align}
        Y(t)=\i+\left(\frac{\k}{2}t+\frac{1}{Y(0)-\i}\right)^{-1}
    \end{align}
    and solutions will tend to $Y_{\infty}=i$ at an algebraic rate.

    In fact, this proves that for all $\k>2\a$, then for almost all initial conditions solutions converge to the asymptotic states described in Figures \ref{fig:Kur2a<0}-\ref{fig:Kur2a>0}. As $\k>2\a$ holds trivially for $\a\leq 0$, Figure \ref{fig:Kur2a<0} is completely proven.

    In order to complete Figure \ref{fig:Kur2a>0} we move to the next case.

    \textit{Case 2:} $\k<2\a$.

    We return to the equation for the $l$ variable.
    \begin{align}
        \ddt l=l(2\a-\k-2(r_1^2+r_2^2))
    \end{align}

\textit{Case 2a:} $\k\geq \g$.

As $\k\geq \g$ we know that the phase equation \eqref{eq:Phi} always has stationary states. Further we can see an invariant region for the phase difference $\Phi$ given by $\Phi \in [0,\frac{\pi}{2}]$. Indeed,
\begin{align}
    \ddt \Phi|_{\Phi=0}=\g>0, \ \ \ \ddt \Phi|_{\Phi=\frac{\pi}{2}}=\g-\frac{\k}{2}\left(\frac{1}{R}+R\right)\leq \g-\k\leq0.
\end{align}
Therefore, if $\Phi \in[0,\frac{\pi}{2}]$, then $\cos(\Phi)\geq0$ and the equation for $R$ gives monotonic convergence towards $R=1$,
\begin{align}
    \ddt R=-lR+\frac{\k}{2}\cos(\Phi)(1-R^2)<0,
\end{align}
within the invariant region, providing stability of the manifold $\cM$. Further note that the invariant phase region is absorbing for almost all initial data as well.

    \textit{Case 2b:} $\k<\g$.

    As $\k<\g$ we know that for $R=1$ no fixed point can exist. Therefore we cannot have an invariant region for the phase variable $\Phi$. On the manifold $\cM$ we saw that the time average $\langle \cos(\Phi)\rangle|_\cM=0$. Off of $\cM$, (but near the manifold) let us see that $\langle \cos(\Phi)\rangle \geq 0$.

    Without loss of generality suppose that $R>1$.  Then the equation,
    \begin{align}
        \ddt R=-lR+\frac{\k}{2}\cos(\Phi)(1-R^2)<0,
    \end{align}
    holds at least when $\cos(\Phi)\geq 0$, i.e. $R$ is strictly decreasing for $\Phi(t) \in [-\frac{\pi}{2},\frac{\pi}{2}]$. On the other hand, it is possible that when $\Phi \sim \pi$ so that $\cos(\Phi)\sim -1$, that $\ddt R>0$ provides growth.

    Returning to the equation for $\Phi$ we see that
    \begin{align}\label{eq:Phiav>0}
        \ddt \Phi=\g-\frac{\k}{2}\left(\frac{1}{R}+R\right)\sin(\Phi).
    \end{align}
    As $\k<\g$, let us suppose that initially $R_0$ is close enough to $1$ so that there are no critical points of \eqref{eq:Phiav>0}. Indeed, if there were, then up to some finite time $T>0$, there would be an invariant region such that $\cos(\Phi(t))\geq 0$ on $[0,T]$ providing monotonic decay of $R$ until there can be no critical points of \eqref{eq:Phiav>0}.

    Now as $R_0$ is close enough to $1$, we have that $\ddt \Phi>0$. Examining the second derivative will tell us where $\Phi(t) \in [0,2\pi]$ spends more time on average, and hence allow us to determine information on $\langle \cos(\Phi)\rangle$.
\begin{align}
    \ddt\dot{\Phi}=-\frac{\k}{2}\sin(\Phi)\ddt\left(R+\frac{1}{R}\right)-\frac{\k}{2}\left(R+\frac{1}{R}\right)\cos(\Phi)\ddt \Phi.
\end{align}

As $\ddt\Phi>0$, we can see the following:
\begin{align}
    \ddt\dot{\Phi}(0)<0, \ \ \ddt\dot{\Phi}(\frac{\pi}{2})>0, \ \ \ddt\dot{\Phi}(\pi)>0, \ \ \ddt\dot{\Phi}(\frac{3\pi}{2})<0.
\end{align}
Hence due to the dependence on $R$ we can see that in each pass from $\Phi=0$ to $\Phi=2\pi$, there are two inflection points that can be found in the first $(\Phi^* \in (0,\frac{\pi}{2})$) and third quadrants $\Phi^{**} \in (\pi,\frac{3\pi}{2})$). This breaks the symmetry that we see on the manifold $\cM$ where each inflection point is found at $\Phi=\frac{\pi}{2}$ and $\Phi=\frac{3\pi}{2}$, and will provide us with the condition $\langle \cos(\Phi)\rangle\geq 0$. Indeed, looking at Figure \ref{fig:Manifold_stability}, one can see the broken symmetry showing that $R$ is decreasing on more of the dynamics than it increases, as well as the acceleration of the phase behavior being shifted off center. This implies that off the manifold $\cM$, the phase variable $\Phi$ spends more time close to zero, $(\Phi \sim 0)$, than it does close to $\pi$, $(\Phi \sim \pi)$, allowing us to conclude that the time average of the cosine of the phase variable is nonnegative, $\langle \cos(\Phi)\rangle\geq 0$.

\begin{figure}
    \centering
    \includegraphics[width=0.5\linewidth]{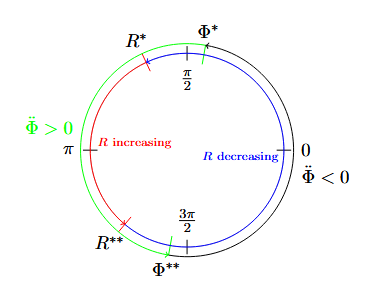}
    \caption{Schematic showing the concavity and inflection points $\Phi^*, \Phi^{**}$ of the phase behavior $\Phi$ and the growth and decay of the ratio $R$ over one $[0,2\pi)$ revolution of the dynamics off of the manifold $\cM$.}
    \label{fig:Manifold_stability}
\end{figure}

In order to yield the stability of the manifold, we must show that $R \to 1$ or equivalently, as we know this regime remains active, $l \to 0$. Let $\e>0$ be arbitrary and assume $l(t)\geq \e$ for all $t\geq 0$. Returning to the equation on $R$ we get the inequality
\begin{align}
    \ddt (R-1)&\leq -\e (R-1)-\e+\frac{\k}{2}\cos(\Phi)(1-R^2),\\
&=-\e +(R-1)(-\e-\frac{\k}{2}\cos(\Phi)(1+R))    
\end{align}
    By Gr\"onwall's inequality we get
    \begin{align}
        R(t)-1 \leq (R(0)-1)\exp\left(\int_0^t-\e-\frac{\k}{2}\cos(\Phi)(1+R(s)) \ds\right).
    \end{align}
Given that $\langle \cos(\Phi)\rangle\geq 0$, we know that $\lim_{t\to\infty} \int_0^t-\e-\frac{\k}{2}\cos(\Phi)(1+R(s)) \ds=-\infty$, implying $R\to 1$. However, this also implies $l \to 0$, contradicting the assumption that $l \geq \e$. Hence $l \to 0$ and the manifold $\cR$ is stable in all cases.
    
  \end{proof}

The model described in Theorem \ref{2t:N=2,a=0} is the most similar to the Kuramoto model, where the same condition for the existence of stable phase-locked states is seen. However, even in this simplest version, the inclusion of the amplitudes into the model lead to different outcomes. The difference here is that how close to the Hopf-bifurcation the oscillators are relative to the coupling strength and natural frequencies determines whether or not amplitude death occurs. Indeed, we see regimes where amplitude death and active states can occur in both phase-locked and incoherent states.

\section{Introducing $\alpha$-Heterogeneity}\label{S:ahet}

Within this section we will systematically go through the passage from Figure \ref{fig:Kur2a<0} to Figure \ref{fig:Kur2a>0} via the introduction of $\alpha$-heterogeneity. 

The relevant equations are then
\begin{align}
    \ddt z_1&=(\a_1+\i\w_1-|z_1|^2)z_1+\frac{\k}{2}(z_2-z_1), \label{eq:SL2het1}\\
    \ddt z_2&=(\a_2+\i\w_2-|z_2|^2)z_2+\frac{\k}{2}(z_1-z_2). \label{eq:SL2het2}
\end{align}

Before we investigate each particular configuration, let us also write down each of the equations that will prove useful in the analysis. The phase difference equation does not see the Hopf-parameter and is therefore unchanged, however, the equation for the ratio of amplitudes is effected, and in turn affects the phase-difference dynamics as the two equations cannot be decoupled
\begin{align}
    \ddt \Phi&=\g-\k\left(R+\frac{1}{R}\right)\sin(\Phi), \label{eq:SLphihet} \\
    \ddt R&=(a-l)R+\frac{\k}{2}\cos(\Phi)(1-R^2). \label{eq:SLRhet}
\end{align}
Recalling that $a=\a_1-\a_2>0$ and $l=r_1^2-r_2^2$. We further have the equation for $l$ and the full complex ratio $Z=\frac{z_1}{z_2}$
\begin{align}
    \ddt l&= 2(\a_1r_1^2-\a_2r_2^2)-l(\k+2(r_1^2+r_2^2)), \label{eq:SLlhet} \\
    \ddt Z&=(a+\i\g-l)Z+\frac{\k}{2}(1-Z^2). \label{eq:SLZhet}
\end{align}
Each of these equations will prove useful in the following subsections, however of particular importance is \eqref{eq:SLZhet} which is a complex-valued Riccati equation with a time dependent coefficient, $l$, in the linear term. Indeed, any situation in which we can first provide a convergence of the variable $l$, the Riccati theory will grant us the rest of the necessary results.\\

Each of the following subsection will cover the asymptotic states of the respective phase diagrams for the different configurations of $\a_1$ and $\a_2$.

\subsection{Subcritical Oscillators} We begin with the first bifurcation where we move $\a_1$ up so that $\a_2<\a_1\leq 0$. Their is only one stable asymptotic state for this system given by Figure \ref{fig:Kur2a2<a1<0}. The result can be given by the following theorem
\begin{theorem}\label{t:SLa2<a1<0}
    Let $\a_2<\a_1\leq 0$, then solutions to the system \eqref{eq:SL2het1}-\eqref{eq:SL2het2} converge to Amplitude Death at an exponential rate, however through this process the phase difference variable $\Phi$ converges to a fixed value providing Phase-Locking, for a.e. initial data.
\end{theorem}

\begin{proof}
    Let us begin with equation \eqref{eq:SLRhet}. As $\a_1>\a_2$, indeed we expect that in general the dynamics would produce $r_1(t)\geq r_2(t)$. Checking at $r_1=r_2$ so that $R=1$ we get
    \begin{align}
        \ddt R|_{R=1}=a>0
    \end{align}
    Hence there is an invariant region such that $r_1\geq r_2$ where $R\geq 1$. Further, by continuity this can be extended to a region such that $R\geq 1+\e$ for some $\e>0$ which depends on the values $a$ and $\g$.

    Supposing initially that $R(0)\geq 1+\e$, this is preserved and we examine the amplitude equation for $z_1$
    \begin{align}
        \ddt r_1&=(\a_1-r_1^2)r_1+\frac{\k}{2}(\cos(\Phi)r_2-r_1),\\
        &=(\a_1+\frac{\k}{2}(\cos(\Phi)\frac{1}{R}-1)-r_1^2)r_1,\\
        &\leq (\a_1-\frac{\k\e}{2(1+\e)}-r_1^2)r_1.
    \end{align}
Thus for $\a_1\leq 0$ we have $\a_1-\frac{\k\e}{2(1+\e)}<0$ and both $r_1,r_2\to 0$ exponentially fast. In particular this also proves $l\to 0$ at an exponential rate.

Returning to equation \eqref{eq:SLZhet} we have
\begin{align}
    \ddt Z=(a+\i\g)Z+\frac{\k}{2}(1-Z^2)+E(t)
\end{align}
    where $E(t)\to 0$ exponentially fast due to the convergence of $l\to 0$ at an exponential rate. Thus the function $Z(t)$ satisfies an exponential perturbation of a complex-valued Riccati equation with constant coefficients. Solving the stationary solution for the unperturbed equation amounts to solving
    \begin{align}
        0=(a+\i\g)Y+\frac{\k}{2}(1-Y^2)
    \end{align}
    which yields two solutions via the quadratic formula
    \begin{align}
        Y_{\infty}^+&=\frac{a+\i\g+\sqrt{(a+\i\g)^2+\k^2}}{\k},\\
        Y_{\infty}^-&=\frac{a+\i\g-\sqrt{(a+\i\g)^2+\k^2}}{\k}.
    \end{align}

    \textit{Linear Stability:}

    Let $F(Y)=(a+\i\g)Y+\frac{\k}{2}(1-Y^2)$ and compute the linearization around the fixed point
    \begin{align}
        DF(Y)&=(a+\i\g)-\k Y,\\
        DF(Y_{\infty}^+)&=-\sqrt{(a+\i\g)^2+\k^2},\\
        DF(Y_{\infty}^-)&=\sqrt{(a+\i\g)^2+\k^2}.
    \end{align}
    Now to see that $\Re{\sqrt{(a+\i\g)^2+\k^2}}>0$ we expand the square to have $(a+\i\g)^2=a^2-\g^2+2a\g\i$. As $a,\g>0$ we have $2a\g>0$, and therefore $(a+\i\g)^2+\k^2$ is found in the upper half plane with positive imaginary part. Writing $\sqrt{(a+\i\g)^2+\k^2}$ in polar form gives
    \begin{align}
        \sqrt{(a+\i\g)^2+\k^2}=\sqrt{r}e^{\i\theta/2}.
    \end{align}
    As $(a+\i\g)^2+\k^2$ is found in the upper half plane, we have $\th \in (0,\pi)$ and hence $\theta/2 \in (0,\frac{\pi}{2})$ and $\Re{\sqrt{(a+\i\g)^2+\k^2}}>0$. Thus $Y_{\infty}^+$ is a stable equilibrium and $Y_{\infty}^-$ is repelling.

    For the unperturbed equation
    \begin{align}
        \ddt Y=(a+\i\g)Y+\frac{\k}{2}(1-Y^2),
    \end{align}
    since there are exactly two fixed points, one attracting and one repelling and the dynamics remain bounded, convergence to $Y_{\infty}^+$ would occur for all initial data not on the unstable manifold of $Y_{\infty}^-$. The same holds for the original equation for $Z$ with the exponentially small perturbation.

    In this way, although the system converges to amplitude death at an exponential rate, we can still see that phase-locking occurs as the quotient variable $Z=Re^{i\Phi} \to Y_{\infty}^+$ exponentially fast. 
\end{proof}

In this way we can see that the variable $R$ remains bounded along the dynamics and converges to the particular value $|Y_{\infty}^+|$ and $\Phi_{\infty}$ can also be extracted from $Y_{\infty}^+$. In fact, this proof yields an interesting fact for the $\a$-heterogeneous system highlighted in the following theorem.

\begin{theorem}[Amplitude Death implies Phase-Locking]\label{t:ADtoPL}
    Let $\a_1>\a_2$, then the solutions to the system \eqref{eq:SL2het1}-\eqref{eq:SL2het2} which undergo Amplitude Death also converge to a Phase-Locked state.
\end{theorem}

\subsection{Interlude 1: Obtaining the Active/Amplitude Death curve}\label{ss:i1}

Now that we have an explicit description of the asymptotic phase-locked state within Amplitude Death, we can extract relations between the system parameters which provide the boundaries of the various phase-transitions seen in Figures \ref{fig:Kur2-a2>a1>0}-\ref{fig:Kur2a1>a2>0}
Let us first find the curve which delineates between the Active State and Amplitude Death. To see how the curve is obtained, we recall the equation for the amplitude of the second oscillator,
\begin{align}
    \ddt r_2=(\a_2-r_2^2)r_2+\frac{\k}{2}(\cos(\Phi)r_1-r_2).
\end{align}
However we can write the above equation in the standard form of a Hopf-bifurcation,
\begin{align}
    \ddt r_2&=(\a_2+\frac{\k}{2}(\cos(\Phi)R-1)-r_2^2)r_2,\\
    &=(\a_2+\frac{\k}{2}(\Re(Z)-1)-r_2^2)r_2,
\end{align}
where we have used the fact that $\cos(\Phi)R=\Re(Z)$.

As we have just shown, Amplitude Death implies phase-locking in $\a$-heterogeneous systems, therefore the boundary of the Amplitude Death region is provided by
\begin{align}
    &\a_2+\frac{\k}{2}(\Re(Y_{\infty})-1)=0,\\
    &\Re(Y_{\infty}^+)=\frac{a+\sqrt{\frac{1}{2}(\sqrt{4a^2\g^2+(a^2-\g^2+\k^2)^2}+a^2-\g^2+\k^2)}}{\k}.
\end{align}
This provides the following relation between $\a_1,\a_2, \k, \g$,
\begin{align}\label{curve1}
  f(\a_1,\a_2,\k,\g):=\a_1+\a_2-\k+\sqrt{\frac{1}{2}(\sqrt{4a^2\g^2+(a^2-\g^2+\k^2)^2}+a^2-\g^2+\k^2)}=0, \ \ \ \k>0,\g\geq0
\end{align}

The equation $f(\a_1,\a_2,\k,\g)=0$ provides the transition from Active to Amplitude deaths in each of Figures \ref{fig:Kur2-a2>a1>0}-\ref{fig:Kur2a1>a2>0} so that if $f(\a_1,\a_2,\k,\g)\geq0$, the only asymptotic state of the system is Amplitude Death, and therefore also phase-locking by Theorem \ref{t:ADtoPL}.\\

Of interesting note is to consider what happens to this curve in each of the transitions between diagrams. Let us begin with Figure \ref{fig:Kur2-a2>a1>0}.
Fixing $\a_1$ and $\a_2$ within this regime we can see that the Active state occurs only for weak coupling values such that $f(\a_1,\a_2,\k,\g)<0$. Further, for $-\a_2>\a_1>0$, we see that $\a_1+\a_2-\k<0$ and hence for any $\g>0$, there exists a $\kappa>0$ satisfying \eqref{curve1}. Therefore we can compute the limiting behavior as $\gamma \to 0$ and as $\gamma \to \infty$ in the above relation.
\begin{lemma}
    Let $-\a_2>\a_1>0$. Then the curve $f(\a_1,\a_2,\k,\g)=0$ which yields the phase transition from Active to Amplitude Death states satisfies the following relations in $\g$.
    \begin{align}
        & f(\a_1,\a_2,\k,0)=0 \implies \k= \frac{2\a_1\a_2}{\a_1+\a_2},\\
        &f(\a_1,\a_2,\k,+\infty)=0  \implies \k=2\a_1.
    \end{align}
\end{lemma}
First, plugging $\g=0$ into the above \eqref{curve1} yields
\begin{align}
    &\a_1+\a_2-\k+\sqrt{a^2+\k^2}=0,\\
    &\k=\frac{2\a_1\a_2}{\a_1+\a_2}.
\end{align}
To compute the asymptotic value of $\k$ as $\g \to \infty$ we first rearrange the terms to isolate the part depending on $\g$,
\begin{align}
    &\lim_{\g \to \infty} f(\a_1,\a_2,\k,\g)=0,\\
    &\lim_{\g \to \infty} \a_1+\a_2-\k+\sqrt{\frac{1}{2}(\sqrt{4a^2\g^2+(a^2-\g^2+\k^2)^2}+a^2-\g^2+\k^2)}=0,\\
    &\lim_{\g \to \infty} \sqrt{\frac{1}{2}(\sqrt{4a^2\g^2+(a^2-\g^2+\k^2)^2}+a^2-\g^2+\k^2)}=\k-\a_1-\a_2,\\
    &\lim_{\g \to \infty} \sqrt{4a^2\g^2+(a^2-\g^2+\k^2)^2}-\g^2=2(\k-\a_1-\a_2)^2-(a^2+\k^2)
\end{align}
As the right hand-side does not depend on $\g$ we need only compute the limit of the left hand side and solve for $\k$. Indeed, multiplying the LHS by the conjugate in the numerator and denominator allows one to yield
\begin{align}
    \lim_{\g \to \infty} \sqrt{4a^2\g^2+(a^2-\g^2+\k^2)^2}-\g^2=a^2-\k^2.
\end{align}
Equating to the above and solving for $\k$ gives $\k=2\a_1$.\\

The shape of the curve $f(\a_1,\a_2,\k,\g)$ changes as we increase $\a_1$ until we hit $\a_1=-\a_2>0$. In particular, we see in Figure \ref{fig:Kur2a1=-a2>0} that with the balance of the supercritical oscillator and the subcritical oscillator, for weak heterogeneity in $\g$, the increased synchronization leads to the active state being recovered in large $\k$ regimes. Thus the domain of permissible $\gamma$ values shrinks. Now due to the symmetry we compute the asymptotic values as $\g \to \infty$ and $\k\to \infty$ respectively.

\begin{lemma}
    Let $\a_1=-\a_2>0$. Then the curve $f(\a_1,\a_2,\k,\g)=0$ which yields the phase transition from Active to Amplitude Death states satisfies the following relations in $\g$ and $\k$.
    \begin{align}  
         f(\a_1,-\a_1,\k,\infty)=0 \implies \k=2\a_1,\\
        f(\a_1,-\a_1,\infty,\g)=0  \implies \g= 2\a_1.        
    \end{align}
\end{lemma}

The first limit is computed exactly as before utilizing $\lim_{\g \to \infty} \sqrt{4a^2\g^2+(a^2-\g^2+\k^2)^2}-\g^2=a^2-\k^2$. Then as $\a_1=-\a_2$ we have $2\a_1=a$ and
\begin{align}
    \k^2-a^2=\lim_{\g \to \infty} \sqrt{4a^2\g^2+(a^2-\g^2+\k^2)^2}-\g^2=a^2-\k^2 \implies \k=a=2\a_1.
\end{align}

The second limit is obtained symmetrically using $\lim_{\k \to \infty} \sqrt{4a^2\g^2+(a^2-\g^2+\k^2)^2}-\k^2=a^2-\g^2$. Therefore we get
\begin{align}
   \g^2-a^2=\lim_{\k \to \infty} \sqrt{4a^2\g^2+(a^2-\g^2+\k^2)^2}-\k^2=a^2-\g^2 \implies \g=a=2\a_1.
\end{align}

\begin{lemma}
    Let $\a_1>-\a_2$. Then the curve $f(\a_1,\a_2,\k,\g)=0$ which yields the phase transition from Active to Amplitude Death states satisfies the following relations in $\g$ and $\k$.
    \begin{align}  
        &f(\a_1,\a_2,\k,\infty)=0 \implies \k=2\a_1,\\
        &f(\a_1,\a_2,\infty,\g)=0 \implies \lim_{\gamma \to \infty}\k(\g) \sim \frac{(\a_1+\a_2)^2+\g^2-a^2}{2(\a_1+\a_2)}.        
    \end{align}
\end{lemma}

The first limit is identical to the previous two cases. The second limit provides the shape of the upper part of the Amplitude Death transition curve seen in Figures \ref{fig:Kur2a1>-a2>0}-\ref{fig:Kur2a1>a2>0}. Indeed, it is clear that the curve diverges to infinity, but the limit provides us with the fact that $\k(\g) \sim \g^2$ as $\k \to \infty$. Again utilizing $\lim_{\k \to \infty} \sqrt{4a^2\g^2+(a^2-\g^2+\k^2)^2}-\k^2=a^2-\g^2$ we get
\begin{align}
    \lim_{\k \to \infty} -4\k(\a_1+\a_2)+2(\a_1+\a_2)^2-a^2+\g^2=\lim_{\k \to \infty} \sqrt{4a^2\g^2+(a^2-\g^2+\k^2)^2}-\k^2=a^2-\g^2.
\end{align}
solving for $\k(\g)$ gives the asymptotic behavior as
\begin{align}
    \k(\g) \sim \frac{(\a_1+\a_2)^2+\g^2-a^2}{2(\a_1+\a_2)}.
\end{align}

Finally,we can see that the limit $\a_2\nearrow \a_1$ yields
\begin{align}
    f(\a,\a,\k,\g)=2\a-\k+\sqrt{\k^2-\g^2}=0,
\end{align}
which provides the exact condition for $\k\geq 2\a$, the top part of the Amplitude Death transition curve is given by $\k^*(\g)=\frac{4\a^2+\g^2}{4\a}$ in the $\a$-homogeneous regime. However, due to this transition being discontinuous, we see that we lose the lower part of the boundary where $\k=2\a$ for $\k\leq \g$.

Further, knowledge of this curve allows us to prove the following result.
\begin{theorem}\label{t:ADstab}
    If $\a_1>\a_2$ and $f(\a_1,\a_2,\k,\g)>0$, then there exists two Amplitude Death Phase-Locked states of the system \eqref{eq:SL2het1}-\eqref{eq:SL2het2} given by $r_1=0=r_2$ and
     \begin{align}
    &Y_{\infty}^+=\frac{a+\i\g+\sqrt{(a+\i\g)^2+\k^2}}{\k},\\
        &Y_{\infty}^-=\frac{a+\i\g-\sqrt{(a+\i\g)^2+\k^2}}{\k}.
    \end{align}
    where $Y_{\infty}^+$ is stable and $Y_{\infty}^-$ is unstable.
\end{theorem}
\begin{proof}
    First, note that Amplitude Death, $r_1=0=r_2$, always represents a fixed point of the system \eqref{eq:SL2het1}-\eqref{eq:SL2het2}. Therefore, as in the previous section $Y_{\infty}^+$ and $Y_{\infty}^-$ both represent phase-locked states for the variable $Z=\frac{z_1}{z_2}$. In order to prove the stability/instability of the two fixed points let us write the four equations we need in order to study the linear stability of the fixed points. Letting $RZ=\Re(Z)$ and $IZ=\Im(Z)$
    \begin{align}
        &\ddt r_1=(\a_1+\frac{\k}{2}(\frac{RZ}{RZ^2+IZ^2}-1)-r_1^2)r_1,\\
        &\ddt r_2=(\a_2+\frac{\k}{2}(RZ-1)-r_2^2)r_2,\\
        &\ddt RZ=(a-r_1^2+r_2^2)RZ-\g IZ+\frac{\k}{2}(1-(RZ)^2+(IZ)^2),\\
        &\ddt IZ=(a-r_1^2+r_2^2)IZ+\g IZ-\k (RZ)(IZ).
    \end{align}

    Now let $F(r_1,r_2,RZ,IZ)$ be the fixed point map of the above equations. Computing the Jacobian of $F$ at the fixed points will give us the linear stability.
\[
J=\begin{bmatrix}
\a_1+\frac{\k}{2}(\Re(1/Y_{\infty})-1) & 0 & 0 & 0  \\
0 & \a_2+\frac{\k}{2}(\Re(Y_{\infty})-1) & 0 & 0 \\
0 & 0 & a-\k\Re(Y_{\infty}) & -\g+\k\Im(Y_{\infty}) \\
0 & 0 & \g-\k\Im(Y_{\infty}) & a-\k\Re(Y_{\infty}) \\
\end{bmatrix}
\]
The condition $f(\a_1,\a_2,\k,\g)>0$ is exactly the condition that guarantees both $\a_1+\frac{\k}{2}(\Re(1/Y_{\infty}^+)-1)<0$ and $\a_2+\frac{\k}{2}(\Re(Y_{\infty}^+)-1)<0$. Therefore the first two eigenvalues are negative and the last two are given by the same stability analysis in the previous section guaranteeing the stability of $Y_{\infty}^+$ and the instability of $Y_{\infty}^-$.
\end{proof}

Note that the above proof also shows that if $f(\a_1,\a_2,\k,\g)<0$, then Amplitude death is unstable and the system remains active. Let us now proceed to the next section.

\subsection{Mixed oscillators--Stability of the Active Phase-Locked states}\label{ss:mixed} Now that we have the equation for the curve separating the Active and Amplitude Death states we can provide a rigorous analysis of the existence, uniqueness and stability of Active Phase-Locked state when we have one supercritical oscillator $(\a_1>0)$ and one subcritical oscillator $(\a_2\leq 0)$. This subsection will provide the proofs of stability of the Active Phase-Locked regimes found in Figures \ref{fig:Kur2-a2>a1>0}-\ref{fig:Kur2a1>-a2>0}.  In particular, with $\a_1>0\geq\a_2$ the active state is recovered whenever $f(\a_1,\a_2,\k,\g)<0$. However, due to the subcriticality of the second oscillator, $\a_2\leq0$, the phase behavior will be asymptotic phase-locking for all parameter values $\k,\g>0$. We have already seen that Phase-Locking is stable in the Amplitude Death regime where $f(\a_1,\a_2,\k,\g)>0$. We begin with existence and uniqueness of the Active Phase-Locked state whenever $f(\a_1,\a_2,\k,\g)<0$.

\begin{theorem}\label{t:mixact}
    Let $\a_1>0\geq \a_2$. If $f(\a_1,\a_2,\k,\g)<0$, then there exists a unique linearly stable Active Phase-Locked State to \eqref{eq:SL2het1}-\eqref{eq:SL2het2} within the following region:
    \begin{align}
        l_{\infty} \in (0,a), \ \ R_{\infty} \in(R_{\infty}^-,R_{\infty}^+), \ \ \Phi_{\infty}\in (0,\frac{\pi}{2}),
    \end{align}
    with $1\leq R_{\infty}^-\leq R_{\infty}^+\leq \infty$ to be defined later.
If $f(\a_1,\a_2,\k,\g)>0$, then the Amplitude Death Phase-Locked State is stable.
\end{theorem}
\begin{proof}
    Theorem \ref{t:ADstab} provides stability of the Amplitude Death Phase-Locked state in the case of $f(\a_1,\a_2,\k,\g)>0$. Now suppose $f(\a_1,\a_2,\k,\g)<0$. Again, by Theorem \ref{t:ADstab} we know that the Amplitude Death state is unstable and hence we are searching for the existence of Active Phase-Locked states.\\

    \textit{Existence and Uniqueness:}\\

    We proceed by noticing that the relative dynamics of each $z_1$ and $z_2$ satisfying \eqref{eq:SL2het1}-\eqref{eq:SL2het2} can be characterized completely by three variables: $R=\frac{r_1}{r_2},\Phi=\phi_1-\phi_2,l=r_1^2-r_2^2$, with the following equations
    \begin{align}
        \ddt R&=(a-l)R+\frac{\k}{2}\cos(\Phi)(1-R^2),\label{eq:Rhet}\\
        \ddt \Phi&=\g-\frac{\k}{2}(R+\frac{1}{R})\sin(\Phi),\label{eq:Phihet}\\
        \ddt l&=2(\a_1r_1^2-\a_2r_2^2)-l(\k+2(r_1^2+r_2^2)). \label{eq:lhet}
    \end{align}
    Indeed, we can see that $l$ and $R$ characterize $r_1$ and $r_2$ as $r_1=r_2R$ and $r_2^2=\frac{l}{R^2-1}$. Therefore obtaining $l$ and $R$ yield $r_1$ and $r_2$. Last, $\Phi$ grants the phase difference between the two oscillators which is sufficient to fully characterize the phase-locked state.

    We begin with the equation on $l$, \eqref{eq:lhet}. As $f(\a_1,\a_2,\k,\g)<0$, we seek steady state solutions to \eqref{eq:lhet} such that $l>0$. Indeed, as the system is active and $\a_1>\a_2$, then a steady state of \eqref{eq:Rhet} with $l=0$ implies that $R>1$ and hence $r_1=r_2=0$ an amplitude death state. Setting the right side of \eqref{eq:lhet} equal to zero gives an equation of a hyperbola in the variables $r_1^2, r_2^2$
    \begin{align}
        H(r_1^2,r_2^2):=(2\a_1-\k)r_1^2-2r_1^4-(2\a_2-\k)r_2^2+2r_2^4=0.\label{eq:Hyper}
    \end{align}
    Only a particular portion of this hyperbola is relevant to the fixed points we are investigating. Indeed, any steady state within the current regime requires that $r_2^2<r_1^2<\a_1$. Hence the critical points of $l$ are those such that $0<r_2^2<r_1^2<\a_1$ and $H(r_1^2,r_2^2)=0$.\\

    Note that this provides bounds on the asymptotic value $l$ can take.
    \begin{lemma}\label{l:lbound}
        Let $\a_1>0\geq \a_2$. Then any Phase-Locked state of \eqref{eq:SL2het1}-\eqref{eq:SL2het2} has a bound on the final state of the variable $l_{\infty}=r_{1,\infty}^2-r_{2,\infty}^2$,
        \begin{align}
            l_{\infty} \in [0,\a_1) \subset [0,a),
        \end{align}
        where $l_{\infty}=0$ only in the case of $f(\a_1,\a_2,\k,\g)>0$ which coincides with Amplitude Death.
    \end{lemma}

    Now for the variables $l_{\infty}, R_{\infty}, \Phi_{\infty}$, we are seeking solutions to the following three equations

    \begin{align}
        0&=(a-l_{\infty})R_{\infty}+\frac{\k}{2}\cos(\Phi_{\infty})(1-R_{\infty}^2),\label{eq:Rhetstat}\\
        0&=\g-\frac{\k}{2}(R_{\infty}+\frac{1}{R_{\infty}})\sin(\Phi_{\infty}),\label{eq:Phihetstat}\\
        0&=2(\a_1r_{1,\infty}^2-\a_2r_{2,\infty}^2)-l_{\infty}(\k+2(r_{1,\infty}^2+r_{2,\infty}^2)). \label{eq:lhetstat}
    \end{align}

    Now, starting with \eqref{eq:lhetstat} we plug in for $r_{1,\infty}^2=\frac{l_{\infty}R_{\infty}^2}{R_{\infty}^2-1}$ and $r_{2,\infty}^2=\frac{l_{\infty}}{R_{\infty}^2-1}$. Taking advantage of the hyperbola formulation in \eqref{eq:Hyper} we can solve for $l_{\infty}$ in terms of $R_{\infty}$
    \begin{align}
        (2\a_1-\k)\frac{l_{\infty}R_{\infty}^2}{R_{\infty}^2-1}-2\left(\frac{l_{\infty}R_{\infty}^2}{R_{\infty}^2-1}\right)^2&=(2\a_2-\k)\frac{l_{\infty}}{R_{\infty}^2-1}-2\left(\frac{l_{\infty}}{R_{\infty}^2-1}\right)^2,\\
        l_{\infty}&=\frac{(2\a_1-\k)R_{\infty}^2+(\k-2\a_2)}{2(R_{\infty}^2+1)}.\label{eq:ltoR}
    \end{align}

    Note, that depending on the relation between $\k$ and $\a_1$, the equation \eqref{eq:ltoR} can provide an upper bound on the value $R_{\infty}$ can take. In Lemma \ref{l:lbound} we see that $l_{\infty} \in [0,a)$ and hence we can see that the following relation must hold,
    \begin{align}
        0\leq \frac{(2\a_1-\k)R_{\infty}^2+(\k-2\a_2)}{2(R_{\infty}^2+1)}<a.
    \end{align}
    If $\k\leq 2\a_1$, then clearly $0\leq \frac{(2\a_1-\k)R_{\infty}^2+(\k-2\a_2)}{2(R_{\infty}^2+1)}$ holds. And further solving the second half of the inequality yields the condition
    \begin{align}
        R_{\infty}^2>\frac{2\a_1-\k}{2\a_2-\k},
    \end{align}
    which always holds as the RHS is nonpositive. However, if $\k>2\a_1$, then we get an upper bound from the condition
    \begin{align}
        0\leq (2\a_1-\k)R_{\infty}^2+(\k-2\a_2) \implies R_{\infty}\leq \sqrt{\frac{\k-2\a_2}{\k-2\a_1}}.
    \end{align}
Therefore let us define
    \[
    R_{\infty}^+=\begin{cases}
        \infty & \text{if} \ \k \leq 2\a_1, \\
       \sqrt{\frac{\k-2\a_2}{\k-2\a_1}}  & \text{if} \ \k > 2\a_1.
    \end{cases}
    \]

Continuing to the phase equation, let us observe that \eqref{eq:Phihetstat} gives us the following relation
\begin{align}
    \cos(\Phi_{\infty})=\sqrt{1-\left(\frac{2\g}{\k(R_{\infty}+\frac{1}{R_{\infty}})}\right)^2}:=g(R_{\infty}).\label{eq:PhitoR}
\end{align}
where we have chosen the positive square root as we are seeking solutions such that $\Phi_{\infty}\in(0,\frac{\pi}{2})$. Further, note that as we are seeking solutions such that $R_{\infty}>1$, we potentially have a further refinement on the domain depending on the relationship between $\k$ and $\g$ in order to guarantee that the square root is well-defined. Indeed,

\begin{align}
    1-\left(\frac{2\g}{\k(R_{\infty}+\frac{1}{R_{\infty}})}\right)^2>0
\end{align}
is required to hold. Rearranging terms gives the quadratic inequality,
\begin{align}
    R_{\infty}^2-\frac{2\g}{\k}R_{\infty}+1>0.
\end{align}
Hence if $\k\geq \g$, then the inequality is always satisfied. If $\k<\g$, then $R_{\infty}$ must be increased beyond the second root of the quadratic to guarantee that the square root is well-defined. Indeed, we get a new minimum value for which $R_{\infty}>R_{\infty}^-$ must hold. We define
\[
    R_{\infty}^-=\begin{cases}
        1 & \text{if} \ \k \geq \g, \\
       \frac{\g}{\k}+\frac{1}{\k}\sqrt{\g^2-\k^2}  & \text{if} \ \k <\g.
    \end{cases}
    \]

    Now let us use $g(R_{\infty})$ and the equation \eqref{eq:ltoR} within \eqref{eq:Rhetstat} to get
\begin{align}
    0=R_{\infty}^3(\k-2\a_2)+R_{\infty}(2\a_1-\k)+\k g(R_{\infty})(1-R_{\infty}^4). \label{eq:Rquadr}
\end{align}
In order to prove existence of a solution to \eqref{eq:Rquadr} we isolate $g(R_{\infty})$,
\begin{align}\label{eq:gh}
    g(R_{\infty})=\frac{R_{\infty}^3(\k-2\a_2)+R_{\infty}(2\a_1-\k)}{\k(R_{\infty}^4-1)}:=h(R_{\infty}).
\end{align}
Therefore we wish to prove the existence of an $R_{\infty} \in (R_{\infty}^-,R_{\infty}^+)$ such that $g(R_{\infty})=h(R_{\infty})$.
The following lemma gives us useful qualitative information about the two functions $g(x)$ and $h(x)$.
\begin{lemma}\label{l:gh}
    If $g(x):=\sqrt{1-\left(\frac{2\g}{\k(x+\frac{1}{x})}\right)^2}$, and $h(x):=\frac{x^3(\k-2\a_2)+x(2\a_1-\k)}{\k(x^4-1)}$. Then letting
    \[
    x^-=\begin{cases}
        \frac{\g}{\k}+\frac{1}{\k}\sqrt{\g^2-\k^2} & \text{if} \ \k < \g, \\
        1 & \text{if} \ \k\geq \g,
    \end{cases}
    \]
    and
     \[
    x^+=\begin{cases}
        \infty & \text{if} \ \k\leq 2\a_1, \\
        \sqrt{\frac{\k-2\a_2}{\k-2\a_1}}  & \text{if} \ \k > 2\a_1,
    \end{cases}
    \]
    then the following holds
    \begin{align}
        &\lim_{x\to (x^-)^+}g(x)=c_1< c_2=\lim_{x\to (x^-)^+}h(x), \label{leftlim}\\
        &\lim_{x \to (x^+)^-} g(x)=c_3>c_4=\lim_{x \to (x^+)^-} h(x). \label{rightlim}
    \end{align}
    further $g'(x)>0$ and $h'(x)<0$ on the domain $(x^-,x^+)$.
\end{lemma}
Note that proving Lemma \ref{l:gh} automatically grants existence (via proving the limits) and uniqueness (via the strict derivative behavior) of an Active Phase-Locked state.
\begin{proof}
Let us begin with the derivative behavior first. As $x+\frac{1}{x}$ is strictly increasing and differentiable on $x>1$ it is clear that $g(x)$ is strictly increasing on its domain $(x^-,x^+)$ and hence $g'(x)>0$. To see the decrease of $h(x)$ let us differentiate.
\begin{align}\label{eq:hder}
    h'(x)=\frac{1}{\k(x^4-1)^2}\left((\k-2\a_2)(-x^6-3x^2)+(2\a_1-\k)(-3x^4-1)\right).
\end{align}
Now if $\k \leq 2\a_1$, then $h'(x)<0$ for all $x>1$. Now suppose $\k>2\a_1$ which gives the wrong sign for the quartic and constant coefficients. However, for $\k>2\a_1$ we have $x^+=\sqrt{\frac{\k-2\a_2}{\k-2\a_1}}$ which implies that for $x<x^+$ we have the following bound,
\begin{align}\label{in:xbound1}
    x^2(2\a_1-\k)+(\k-2\a_2)>0.
\end{align}
Further, we also always have
\begin{align}\label{in:xbound2}
    (2\a_1-\k)+x^2(\k-2\a_2)>0.
\end{align}

Therefore we can group the terms in \eqref{eq:hder} and use \eqref{in:xbound1} and \eqref{in:xbound2} to obtain,
\begin{align}
    h'(x)=\frac{1}{\k(x^4-1)^2}\left(\left((2\a_1-\k)+x^2(\k-2\a_2) \right)(-x^4-1)+\left(x^2(2\a_1-\k)+(\k-2\a_2)\right)(-2x^2)\right)<0.\\
    \nonumber
\end{align}
Continuing to the limiting behavior to guarantee the existence of a point where $g(x)=h(x)$.

Let us begin with the left limits. Suppose $\k \geq \g$, then $x^-=1$ and
\begin{align}
    &\lim_{x \to 1^+} g(x)=c_1=\sqrt{1-\left(\frac{\g}{\k}\right)^2} \in [0,1),\\
    &\lim_{x \to 1^+} h(x)=\lim_{x \to 1^+}\frac{x^3(\k-2\a_2)+x(2\a_1-\k)}{\k(x^4-1)}=\lim_{x \to 1^+} \frac{2a}{\k(x^4-1)}=+\infty>c_1.
\end{align}

On the other hand if $\k<\g$, then $x^-=\frac{\g}{\k}+\frac{1}{\k}\sqrt{\g^2-\k^2}>1$ and we can simply plug in to get
\begin{align}
    &\lim_{x \to (x^-)^+} h(x)=\frac{(x^-)^3(\k-2\a_2)+(x^-)(2\a_1-\k)}{\k((x^-)^4-1)}=c_2>0,\\
    &\lim_{x \to (x^-)^+} g(x)=c_1=0<c_2
\end{align}
where $c_2>0$ holds from \eqref{in:xbound2}, and $c_1=0$ by construction of $x^-$ being the larger root of the quadratic $x^2-\frac{2\g}{\k}x+1=0$. Therefore \eqref{leftlim} is proved.\\

Now if $\k \leq 2\a_1$, then $x^+=\infty$ and the limits $\lim_{x\to \infty} g(x)=c_3=1$ and $\lim_{x\to \infty} h(x)=c_4=0$ are trivial. Thus the last regime to check is when $\k>2\a_1$, so that $x^+=\sqrt{\frac{\k-2\a_2}{\k-2\a_1}}$.

As $x^+<\infty$ and $g(x)$ and $h(x)$ are both continuous we can simply plug $x^+$ into each of the functions, and what is left to be shown is the inequality
\begin{align}\label{in:ghgoal}
    g(x^+)>h(x^+).
\end{align}

Within this regime we must take advantage of the fact that $f(\a_1,\a_2,\k,\g)<0$ to achieve \eqref{in:ghgoal}.

Let us compute
\begin{align}
    &g(x^+)^2=1-\frac{\g^2}{\k^2}\frac{(\k-2\a_1)(\k-2\a_2)}{(\k-(\a_1+\a_2))^2},\label{g2}\\
    &h(x^+)^2=\frac{1}{\k^2}(\k-2\a_1)(\k-2\a_2)
\end{align}

From this we can see that we need an upper bound on $\g^2$ in order to guarantee \eqref{in:ghgoal}. Indeed, inspection of Figures \ref{fig:Kur2-a2>a1>0}-\ref{fig:Kur2a1>-a2>0} one can see that for $\k>2\a_1$ there is a maximal $\g^*$ such that $f(\a_1,\a_2,\k,\g^*)=0$ where the transition to amplitude death occurs.

Therefore we can find the $\g$ bound using $f(\a_1,\a_2,\k,\g)<0$. We have
\begin{align}
    \k-(\a_1+\a_2)>\sqrt{\frac{1}{2}(\sqrt{4a^2\g^2+(a^2-\g^2+\k^2)^2}+a^2-\g^2+\k^2)}.
\end{align}
Squaring both sides and isolating the interior square root
\begin{align}
    2(\k-(\a_1+\a_2))^2-(a^2-\g^2-\k^2)>\sqrt{4a^2\g^2+(a^2-\g^2+\k^2)^2}.
\end{align}
Squaring again yields
\begin{align}
    4(\k-(\a_1+\a_2))^4-4((\k-(\a_1+\a_2))^2(a^2-\g^2-\k^2)>4a^2\g^2.
\end{align}
Isolating $\g^2$ yields the bound
\begin{align}\label{in:gbound}
    \g^2<\frac{(\k-(\a_1+\a_2))^2}{(\k-2\a_1)(\k-2\a_2)}\left(\k^2-(\k-2\a_1)(\k-2\a_2)\right).
\end{align}
Plugging this bound into \eqref{g2} yields exactly $g(x^+)^2>h(x^+)^2$ establishing \eqref{in:ghgoal} and thus \eqref{rightlim}.
\end{proof}

Lemma \ref{l:gh} establishes the existence and uniqueness of an Active Phase-Locked state whenever $f(\a_1,\a_2,\k,\g)<0$ in the domain
\begin{align}
    &R_{\infty} \in (R_{\infty}^-,R_{\infty}^+),\\
    &l_{\infty} \in (0,a),\\
    &\Phi_{\infty} \in (0,\frac{\pi}{2}),
\end{align}
which satisfy equations \eqref{eq:Rhetstat}-\eqref{eq:Phihetstat}. Finally we must establish the stability of this fixed point via analyzing the Jacobian.\\

\textit{Stability:}\\

Let $F(l_{\infty},R_{\infty},\Phi_{\infty})=(F_l,F_R,F_\Phi)$ with
\begin{align}
    F_l&=\frac{2l_{\infty}}{R_{\infty}^2-1}\left(\a_1R_{\infty}^2-\a_2-l_{\infty}(R_{\infty}^2+1)\right)-\k l_{\infty}\\
    F_R&=(a-l_{\infty})R_{\infty}+\frac{\k}{2}\cos(\Phi_{\infty})(1-R_{\infty}^2),\\
    F_{\Phi}&=\g-\frac{\k}{2}\left(R_{\infty}+\frac{1}{R_{\infty}}\right)\sin(\Phi_{\infty}).
\end{align}

Then $F(l_{\infty},R_{\infty},\Phi_{\infty})$ is a closed system in $l_{\infty},R_{\infty},\Phi_{\infty})$ and computing the Jacobian will allow us to analyze the stability of the unique fixed point found above.

Let us compute the needed derivatives first,
\begin{align*}
    \partial_{l_{\infty}}F_l&=\frac{2}{R_{\infty}^2-1}\left(\a_1R_{\infty}^2-\a_2-l_{\infty}(R_{\infty}^2+1)\right)-\frac{2l_{\infty}}{R_{\infty}^2-1}\left(R_{\infty}^2+1\right)-\k=-\frac{2l_{\infty}}{R_{\infty}^2-1}\left(R_{\infty}^2+1\right),\\
    \partial_{R_{\infty}}F_l&=2l_{\infty}\left(\frac{-2R_{\infty}}{(R_{\infty}^2-1)^2}\right)\left(\a_1R_{\infty}^2-\a_2-l_{\infty}(R_{\infty}^2+1)\right)+\frac{2l_{\infty}}{R_{\infty}^2-1}(2\a_1R_{\infty}-2l_{\infty}R_{\infty}),\\
    &=\frac{2l_{\infty}R_{\infty}}{R_{\infty}^2-1}(2\a_1-2l_{\infty}-\k),\\
    \partial_{\Phi_{\infty}}F_l&=0,\\
    \partial_{l_{\infty}}F_R&=-R_{\infty},\\
    \partial_{R_{\infty}}F_R&=(a-l_{\infty})-\k\cos(\Phi_{\infty})R_{\infty}=-\frac{\k}{2}\left(R_{\infty}+\frac{1}{R_{\infty}}\right)\cos(\Phi_{\infty}),\\
    \partial_{\Phi_{\infty}}F_R&=-\frac{\k}{2}\sin(\Phi_{\infty})(1-R_{\infty}^2),\\
    \partial_{l_{\infty}}F_{\Phi}&=0,\\
    \partial_{R_{\infty}}F_{\Phi}&=-\frac{\k}{2}\sin(\Phi_{\infty})(1-\frac{1}{R_{\infty}^2}),\\
    \partial_{\Phi_{\infty}}F_{\Phi}&=-\frac{\k}{2}\left(R_{\infty}+\frac{1}{R_{\infty}}\right)\cos(\Phi_{\infty}),
\end{align*}
where any simplifications are achieved by plugging in the values at the fixed point.
Now the Jacobian matrix is given by
\[
DF(l_{\infty},R_{\infty},\Phi_{\infty})=\begin{bmatrix}
-\frac{2l_{\infty}}{R_{\infty}^2-1}\left(R_{\infty}^2+1\right) & \frac{2l_{\infty}R_{\infty}}{R_{\infty}^2-1}(2\a_1-2l_{\infty}-\k) & 0 \\
-R_{\infty} & -\frac{\k}{2}\left(R_{\infty}+\frac{1}{R_{\infty}}\right)\cos(\Phi_{\infty}) & -\frac{\k}{2}\sin(\Phi_{\infty})(1-R_{\infty}^2) \\
0 & -\frac{\k}{2}\sin(\Phi_{\infty})(1-\frac{1}{R_{\infty}^2}) & -\frac{\k}{2}\left(R_{\infty}+\frac{1}{R_{\infty}}\right)\cos(\Phi_{\infty})
\end{bmatrix}
\]
\[=\begin{bmatrix}
    d_{11} & d_{12} & d_{13} \\
    d_{21} & d_{22} & d_{23} \\
    d_{31} & d_{32} & d_{33}
\end{bmatrix}
\]
The characteristic polynomial is then given by
\begin{align}
    \chi(\l)=-\l^3+\mathrm{Tr}(DF)\l^2-\frac{1}{2}(\mathrm{Tr}^2(DF)-\mathrm{Tr}(DF^2))\l+\det(DF).
\end{align}
First, let us compute $\det(DF)$. Since $d_{13}=d_{31}=0$ the computation simplifies to
\begin{align}
    \det(DF)=d_{11}(d_{22}d_{33}-d_{32}d_{23})-d_{12}d_{21}d_{33}
\end{align}
Note that $-d_{11}d_{32}d_{23}<0$ and let us focus on
\begin{align}
    d_{33}(d_{11}d_{22}-d_{12}d_{21})&=\frac{-\k l_{\infty}R_{\infty}^2}{R_{\infty}^2-1}(R_{\infty}+\frac{1}{R_{\infty}})\cos(\Phi_{\infty})\left(\frac{\k}{2}(R_{\infty}+\frac{1}{R_{\infty}})\cos(\Phi_{\infty})+2\a_1-2l_{\infty}-\k\right)\\
    &-\frac{l_{\infty}\k^2}{2(R_{\infty}^2-1)}\left(R_{\infty}+\frac{1}{R_{\infty}}\right)^2\cos^2(\Phi_{\infty})
\end{align}
As the second term is strictly negative we focus on the first product. Further since $\frac{-\k l_{\infty}R_{\infty}^2}{R_{\infty}^2-1}(R_{\infty}+\frac{1}{R_{\infty}})\cos(\Phi_{\infty})<0$ the second half of the product will determine the sign of $d_{33}(d_{11}d_{22}-d_{12}d_{21})$. Plugging in for $l_{\infty}$ and $\cos(\Phi_{\infty})$ at the fixed point using \eqref{eq:ltoR} and \eqref{eq:gh}.
\begin{align}
    &\left(\frac{\k}{2}(R_{\infty}+\frac{1}{R_{\infty}})\cos(\Phi_{\infty})+2\a_1-2l_{\infty}-\k\right)\\
    & \ \ \ \ \ \ =2\a_1-\k+\frac{(\k-2\a_2)R_{\infty}^2+2\a_1-\k}{2(R_{\infty}^2-1)}-\frac{(2\a_1-\k)R_{\infty}^2+\k-2\a_2}{R_{\infty}^2+1},\\
    & \ \ \ \ \ \ =\frac{1}{2(R_{\infty}^4-1)}\left((2\a_1-\k)(3R_{\infty}^2-1)+(\k-2\a_2)(R_{\infty}^4-R_{\infty}^2+2)\right).
\end{align}
Grouping appropriately to utilize \eqref{in:xbound1} and \eqref{in:xbound2} we have
\begin{align}
    =\frac{1}{2(R_{\infty}^4-1)}\left(\left((2\a_1-\k)+(\k-2\a_2)R_{\infty}^2\right)(R_{\infty}^2-1)+2\left((2\a_1-\k)R_{\infty}^2+(\k-2\a_2)\right)\right)>0.
\end{align}
Therefore we can conclude that $d_{33}(d_{11}d_{22}-d_{12}d_{21})<0$ and $\det(DF)<0$.\\

Moving on to compute the linear coefficient
\begin{align}
    -\frac{1}{2}(\mathrm{Tr}^2(DF)-\mathrm{Tr}(DF^2))=-d_{11}d_{22}-d_{11}d_{33}-d_{22}d_{33}+d_{23}d_{32}+d_{12}d_{21}.
\end{align}
As we just showed that $d_{33}(d_{11}d_{22}-d_{12}d_{21})<0$ and $d_{33}<0$ we know that $-d_{11}d_{22}+d_{12}d_{21}<0$ and further each of $-d_{11}d_{33}<0, -d_{22}d_{33}<0$ and $d_{23}d_{32}<0$ so that the linear coefficient $-\frac{1}{2}(\mathrm{Tr}^2(DF)-\mathrm{Tr}(DF^2))<0$.\\

Last the trace is negative as each $d_{ii}<0$. Thus the characteristic polynomial is given by
\begin{align}\label{eq:charac}
    \chi(\l)=a_0\l^3+a_1\l^2+a_2\l+a_3, \ \ a_i<0, \  i=0,...,3.
\end{align}
Therefore by the Routh–Hurwitz criterion every eigenvalue is found strictly in the negative half-plane and the fixed point is stable.
\end{proof}

As can be seen in the passage from Figure \ref{fig:Kur2-a2>a1>0} to Figure \ref{fig:Kur2a1=-a2>0} to Figure \ref{fig:Kur2a1>-a2>0} a phase transition in the curve $f(\a_1,\a_2,\k,\g)=0$ occurs when the $\a$ parameters go from subcritical on average to supercritical on average. The next subsection investigates what occurs during this transition.

\subsection{Interlude 2: The subcritical on average to supercritical on average transition}\label{ss:I2}
When $\a_1+\a_2<0$ it is clear that the curve $f(\a_1,\a_2,\k,\g)=0$ can be viewed as a function $\k(\g)$ for $\g \in \R_+$. When $\a_1+\a_2=0$ the domain of the function discontinuously jumps to $\g>2\a_1$, while for $\a_1+\a_2>0$, the curve $f(\a_1,\a_2,\k,\g)=0$ fails to be a function $\k(\g)$ as it becomes multivalued.
It can further be observed in Figures \ref{fig:Kur2a1>-a2>0} and \ref{fig:Kur2a1>a2>0} that there exists a minimal value $\g'$ such that for $\g<\g'$ the Active State is guaranteed, while a phase transition occurs at the value $\g'$ so that Amplitude death becomes possible on the other side of the curve $f(\a_1,\a_2,\k,\g)=0$. We can explicitly compute this value $\g'$ by minimizing over this curve. Indeed, in the regime where $\a_1+\a_2>0$ are fixed, we see that the curve $f(\a_1,\a_2,\k,\g)=0$ can instead be seen as a function $\g(\k)$ within the domain $\k>2\a_1$. Solving \eqref{curve1} for $\g^2$ gives
\begin{align}\label{gammacurve}
    \g^2(\k)=\frac{-2(\k-(\a_1+\a_2))^2(2\a_1\a_2-\k(\a_1+\a_2))}{\k^2-2(\a_1+\a_2)\k+4\a_1\a_2}.
\end{align}
Minimizing with respect to $\k$ yields solving the following quartic equation to obtain the critical points of $\g(\k)$.
\begin{align}
    (\k-(\a_1+\a_2))\left(\k^3-3(\a_1+\a_2)\k^2+12\a_1\a_2-16\frac{\a_1^2\a_2^2}{\a_1+\a_2}\right)=0
\end{align}
As $\a_1>|\a_2|$ the first root at $\k=\a_1+\a_2$ is outside the domain $\k>2\a_1$. Thus we seek solutions to the cubic part $g(\k)=\k^3-3(\a_1+\a_2)\k^2+12\a_1\a_2-16\frac{\a_1^2\a_2^2}{\a_1+\a_2}$. Differentiating the cubic part yields the quadratic function
\begin{align}
    g'(\k)=3\k^2-6(\a_1+\a_2)\k+12\a_1\a_2=3(\k-2\a_2)(\k-2\a_1),
\end{align}
and hence the cubic $g(\k)$ is strictly increasing on the domain $\k>2\a_1$. Further, $g(2\a_1)=-4\a_1^2(\a_1-\a_2)^2<0$ and hence the cubic has exactly one real root for $\k>2\a_1$. Indeed, one can check that the discriminant of the cubic is negative, $\Delta<0$, and thus the only real root is given by Cardano's formula.
\begin{align}
    \k=\a_1+\a_2+\sqrt[3]{-\frac{q}{2}+\sqrt{\frac{q^2}{4}+\frac{p^3}{27}}}+\sqrt[3]{-\frac{q}{2}-\sqrt{\frac{q^2}{4}+\frac{p^3}{27}}}
\end{align}
for $p=12\a_1\a_2-3(\a_1+\a_2)^2$ and $q=-2(\a_1+\a_2)^3+12\a_1\a_2(\a_1+\a_2)-\frac{16\a_1^2\a_2^2}{\a_1+\a_2}$. Plugging this into \eqref{gammacurve} yields the critical value $\g'$.

This analysis provides that when oscillators are supercritical on average, the Amplitude Death state can only arise if there is both sufficient $\o$-heterogeneity and sufficient coupling strength $\k>2\a_1$.\\

 Of particular importance is that due to the amplitude dependence of the system, whenever $\a_2\leq 0$, for any $\k>0$, the oscillators are able to reach a phase-locked state regardless of the natural frequencies of the oscillators. As $\a_2\leq 0$, the subcritical oscillator is able to have small amplitudes in order to accommodate a potentially weak coupling within the phase equation. Within certain regimes this leads to the supercritical oscillator remaining relatively unaffected by the subcritical oscillator while forcing the subcritical oscillator to adopt the rotation speed of the supercritical one, thus attaining an active phase-locked state despite weak coupling and strong heterogeneity of parameters. This so-called \textit{leader-driven synchronization} can be more clearly illustrated in the following subsections when the smaller oscillator is also supercritical. This completes the analysis of coupling one subcritical oscillator with a supercritical oscillator.

\subsection{Interlude 3: Obtaining the Phase-Locking/Incoherence Curve}\label{ss:i3}

Increasing $\a_2$ until we have $\a_1>\a_2>0$ provides the recovery of incoherent dynamics as can be seen in Figure \ref{fig:Kur2a1>a2>0}. Within this figure we still see the same curve $f(\a_1,\a_2,\k,\g)=0$ which provides the transition from Active to Amplitude Death. To obtain the equation for the transition between the Phase-Locking and Incoherent states we note that the curve arises from a coincidence of events.

In the previous subsections when $\a_2\leq 0$, the heterogeneity parameter $a=\a_1-\a_2\geq \a_1$ and hence the automatic bound of $r_j^2 \leq \a_1$ implies that $l<  \a_1\leq a$. However now that $\a_2>0$, this is no longer guaranteed as $a=\a_1-\a_2<\a_1$. Therefore the bound $l< \a_1$ does not guarantee that $a-l>0$. Let us observe what occurs if $l_{\infty}=a$. We plug this into equation \eqref{eq:ltoR} to get
\begin{align}
    a=\frac{(2\a_1-\k)R_{\infty}^2+(\k-2\a_2)}{2(R_{\infty}^2+1)} \implies R_{\infty}=\sqrt{\frac{2\a_1-\k}{2\a_2-\k}}.
\end{align}
From this we return to the fixed point equation for $R_{\infty}$, \eqref{eq:Rhetstat} to get
\begin{align}
    0=\frac{\k}{2}\cos(\Phi_{\infty})(1-R_{\infty}^2).
\end{align}
As $R_{\infty}\neq 1$, we see that for a fixed point to exist when $l_{\infty}=a$, we must have $\cos(\Phi_{\infty})=0$. Utilizing this in \eqref{eq:Phihetstat} we get the equation
\begin{align}
    0=\g-\frac{\k}{2}(R_{\infty}+\frac{1}{R_{\infty}}) \implies \g^*(\k)=\k\left(\frac{\a_1+\a_2-\k}{\sqrt{(2\a_1-\k)(2\a_2-\k)}}\right), \ \ \k \in (0,2\a_2).
\end{align}

Along this curve, the stability analysis provided in the previous section still holds so that the Active Phase-Locked State $l_{\infty}=a, R_{\infty}=\sqrt{\frac{2\a_1-\k}{2\a_2-\k}},\Phi_{\infty}=\frac{\pi}{2}$ is stable along the curve $\g^*(\k)$. Indeed the Jacobian is given by

\[
DF(l_{\infty},R_{\infty},\Phi_{\infty})=\begin{bmatrix}
-\frac{2a}{R_{\infty}^2-1}\left(R_{\infty}^2+1\right) & \frac{2aR_{\infty}}{R_{\infty}^2-1}(2\a_1-2a-\k) & 0 \\
-R_{\infty} & 0 & -\frac{\k}{2}(1-R_{\infty}^2) \\
0 & -\frac{\k}{2}(1-\frac{1}{R_{\infty}^2}) & 0 \\
\end{bmatrix}
\]
The characteristic polynomial is given by
\begin{align}
    \chi(\l)=-\l^3+d_{11}\l^2+(d_{23}d_{32}+d_{12}{d_{21}})\l-d_{11}d_{23}{d_{32}}
\end{align}
Again, by the Routh-Hurwitz criterion we can conclude the stability of this fixed point since $d_{11}<0$, $d_{23}>0$, $d_{32}<0$, $d_{21}<0$ and
\begin{align}
    d_{12}=\frac{2aR_{\infty}}{R_{\infty}^2-1}(2\a_2-\k)>0, \ \ \text{for} \ \k\in (0,2\a_2),
\end{align}
which is the whole domain for $\g^*(\k)$. Hence the transition curve retains linear stability in contrast to both the $\a$-homogeneous case and the $N=2$ Kuramoto model.

\subsection{Supercritical Oscillators}
We reach the final heterogeneous configuration for $N=2$ oscillators where both oscillators are supercritical so that $\a_1>\a_2>0$. The asymptotic states are depicted in Figure \ref{fig:Kur2a1>a2>0}, and the following theorem provides stability for each of the asymptotic states.

\begin{theorem}\label{t:a1>a2>0}
    Let $\a_1>\a_2>0$. If $f(\a_1,\a_2,\k,\g)>0$, then the Amplitude Death Phase-Locked state given in Theorem \ref{t:ADstab} is stable. If $f(\a_1,\a_2,\k,\g)<0$, then Amplitude Death is unstable and if further either $\k\geq 2\a_2$ or if $\k<2\a_2$ and $\g\leq\g^*(\k)$, there exists a unique stable phase-locked state in the domain
    \begin{align}
    &R_{\infty} \in (R_{\infty}^-,R_{\infty}^+),\\
    &l_{\infty} \in (0,a],\\
    &\Phi_{\infty} \in (0,\frac{\pi}{2}].
\end{align}
If $\k<2\a_2$ and $\g>\g^*(\k)$, then the system remains Active, however no stable phase-locked state exists.
\end{theorem}

Much of the heavy lifting for this theorem has been proved in the previous Theorem \ref{t:ADstab} subsections \ref{ss:i1}, \ref{ss:mixed}, \ref{ss:i3}. Indeed, Theorem \ref{t:ADstab} covers the case of $f(\a_1,\a_2,\k,\g)>0$. Further, the analysis completed in \ref{ss:i1} and \ref{ss:i3} shows that the two curves given by $f(\a_1,\a_2,\k,\g)=0$ and $\g^*(\k)$ never intersect. Further, notice that if $\k\geq 2\a_2$, then for the characterization of $l_{\infty}$ obtained earlier,
\begin{align}
    l_{\infty}=\frac{(2\a_1-\k)R_{\infty}^2+(\k-2\a_2)}{2(R_{\infty}^2+1)},\label{eq:ltoR2}
\end{align}
the condition $R_{\infty}>1$ automatically guarantees that $l_{\infty} \in (0,a)$. Hence the solution to the stationary equation for $R_{\infty}$,
\begin{align}
     0&=(a-l_{\infty})R_{\infty}+\frac{\k}{2}\cos(\Phi_{\infty})(1-R_{\infty}^2),\label{eq:Rhetstat2}
\end{align}
implies that a solution satisfies $\cos(\Phi_{\infty}) \in (0,\frac{\pi}{2})$. Therefore the definitions of $R_{\infty}^-$ and $R_{\infty}^+$ are identical and the argument for existence, uniqueness and stability of the Active Phase-Locked state is identical to that of Theorem \ref{t:mixact}.\\

Now assume $\k<2\a_2$, and $\g<\g^*(\k)$. Then from \eqref{eq:ltoR2} we again seek the bounds $l_{\infty} \in (0,a)$.

\begin{align}
    0< \frac{(2\a_1-\k)R_{\infty}^2+(\k-2\a_2)}{2(R_{\infty}^2+1)}<a \implies \sqrt{\frac{2\a_2-\k}{2\a_1-\k}}<R_{\infty}<\sqrt{\frac{2\a_1-\k}{2\a_2-\k}}
\end{align}
As $\k<2\a_2<2\a_1$ we see that for $R_{\infty}>1$ the first bound automatically holds, and the second inequality provides us with a new $R_{\infty}^+$ within this regime.

We proceed to extend the proof of existence and uniqueness within this region via the same functions $g(x)$ and $h(x)$ as previously defined.

\begin{lemma}\label{l:gh2}
    If $g(x):=\sqrt{1-\left(\frac{2\g}{\k(x+\frac{1}{x})}\right)^2}$, and $h(x):=\frac{x^3(\k-2\a_2)+x(2\a_1-\k)}{\k(x^4-1)}$. Then letting
    \[
    x^-=\begin{cases}
        \frac{\g}{\k}+\frac{1}{\k}\sqrt{\g^2-\k^2} & \text{if} \ \k < \g, \\
        1 & \text{if} \ \k\geq \g,
    \end{cases}
    \]
    and
     \[
    x^+=\begin{cases}
        \sqrt{\frac{\k-2\a_2}{\k-2\a_1}}  & \text{if} \ \k > 2\a_1,\\
        \infty & \text{if} \ 2\a_2\leq \k\leq 2\a_1, \\
           \sqrt{\frac{2\a_1-\k}{2\a_2-\k}}  & \text{if} \ \k < 2\a_2.
    \end{cases}
    \]
    then the following holds
    \begin{align}
        &\lim_{x\to (x^-)^+}g(x)=c_1< c_2=\lim_{x\to (x^-)^+}h(x), \label{leftlim2}\\
        &\lim_{x \to (x^+)^-} g(x)=c_3>c_4=\lim_{x \to (x^+)^-} h(x). \label{rightlim2}
    \end{align}
    further $g'(x)>0$ and $h'(x)<0$ on the domain $(x^-,x^+)$.
\end{lemma}
\begin{proof}
    The derivative behavior as well as the limits for $\k\geq 2\a_2$ are already done. Therefore we begin with the left limits for $\k<2\a_2$. In this case we have to check the two cases depending on if $\k\geq \g$ or $\k<\g$. First, the limits for $\k\geq \g$ are unchanged so that
    \begin{align}
    &\lim_{x \to 1^+} g(x)=c_1=\sqrt{1-\left(\frac{\g}{\k}\right)^2} \in [0,1),\\
    &\lim_{x \to 1^+} h(x)=\lim_{x \to 1^+}\frac{x^3(\k-2\a_2)+x(2\a_1-\k)}{\k(x^4-1)}=\lim_{x \to 1^+} \frac{2a}{\k(x^4-1)}=+\infty>c_1.
\end{align}
For $\k<\g$, we still have $\lim_{x \to (x^-)^+} g(x)=0$. However, now that $\k-2\a_2<0$ the limit for $h(x)$ is more complicated. We still have $\lim_{x \to (x^-)^+} h(x)=h(x^-)$,
\begin{align}
    h(x^-)=\frac{\left(\frac{\g}{\k}+\frac{1}{\k}\sqrt{\g^2-\k^2}\right)}{\k\left(\left(\frac{\g}{\k}+\frac{1}{\k}\sqrt{\g^2-\k^2}\right)^4-1\right)}\left(\left(\frac{\g}{\k}+\frac{1}{\k}\sqrt{\g^2-\k^2}\right)^2(\k-2\a_2)+2\a_1-\k\right).
\end{align}
As $\frac{\g}{\k}+\frac{1}{\k}\sqrt{\g^2-\k^2}>1$, we focus on showing the second half of the product is positive to get $h(x^-)>0$. Now, we use the fact that $\g<\g^*(\k)=\k\left(\frac{\a_1+\a_2-\k}{\sqrt{(2\a_1-\k)(2\a_2-\k)}}\right)$ to get
\begin{align}
    &\left(\left(\frac{\g}{\k}+\frac{1}{\k}\sqrt{\g^2-\k^2}\right)^2(\k-2\a_2)+2\a_1-\k\right)>\\
    &\left(2\frac{(\a_1+\a_2-\k)^2}{(2\a_1-\k)(2\a_2-\k)}+2\frac{\a_1+\a_2-\k}{\sqrt{(2\a_1-\k)(2\a_2-\k)}}\sqrt{\frac{(\a_1+\a_2-\k)^2}{(2\a_1-\k)(2\a_2-\k)}-1}-1\right)(\k-2\a_2)+\\
    & \ \ \ \ +2\a_1-\k=0,\nonumber
\end{align}
where one can verify the final equality via a decent amount of algebra. Thus $\lim_{x \to (x^-)^+} h(x)=h(x^-)=c_2>c_1$.\\

We continue with the right limits. Again we evaluate each of $h$ and $g$ at the end point $x^+$ to get
\begin{align}
    &\lim_{x\to (x^+)^-} h(x)=h(x^+)=0, \label{hRlim}\\
    &\lim_{x \to (x^+)^-} g(x)=g(x^+)=\sqrt{1-\frac{\g^2(2\a_1-\k)(2\a_2-\k)}{\k^2(\a_1+\a_2-\k)^2}}>0.\label{gRlim}
\end{align}
Equation \eqref{hRlim} is attained by construction of $x^+$ to guarantee that $h(x)>0$ in the regime $\k<2\a_2$. Meanwhile \eqref{gRlim} is automatic as $g'(x)>0$ and $g(x^-)\geq0$, however one can also check directly using the fact that $\g<\g^*(\k)$.
\end{proof}

Lemma \ref{l:gh2} provides existence and uniqueness of an Active Phase-Locked state in the domain and regimes defined in Theorem \ref{t:a1>a2>0}. The linear stability of this state can be deduced from the same argument as for Theorem \ref{t:mixact}. This can be seen from the fact that in $\k \geq 2\a_2$, the argument does not change at all, while for $\k<2\a_2$ and $\g<\g^*(\k)$ the restriction of the domain of $R_{\infty}$ given in Lemma \ref{l:gh2} yields the same bounds \eqref{in:xbound1}-\eqref{in:xbound2} that were necessary to guarantee the negativity of all the coefficients of the characteristic polynomial \eqref{eq:charac} so that the Routh-Hurwitz criterion gives eigenvalues only in the left half-plane. Further, in the previous Subsection \ref{ss:i3} we saw that when $\g=\g^*(\k)$ the stability is preserved.\\

Finally, when $\k<2\a_2$ and $\g>\g^*(\k)$ we see that the system remains active since $f(\a_1,\a_2,\k,\g)<0$. Let us continue to prove that any phase-locked state within this regime must be unstable.

\begin{lemma}\label{l:incoherence}
    Let $\a_1>\a_2>0$ and suppose that $\k<2\a_2$ and $\g>\g^*(\k)$. Then any triple $(l_{\infty},R_{\infty},\Phi_{\infty})$ which represents a fixed point of the equations \eqref{eq:Rhetstat}, \eqref{eq:Phihetstat}, \eqref{eq:lhetstat} is not stable.
\end{lemma}

\begin{proof}
First notice that the incoherence curve $\g^*(\k)$ was constructed via finding exactly when $l_{\infty}=a$, therefore the condition $\g>\g^*(\k)$ implies that $a<l_{\infty}<\a_1$ must hold. Therefore from,
\begin{align}
    0=(a-l_{\infty})R_{\infty}+\frac{\k}{2}\cos(\Phi_{\infty})(1-R_{\infty}^2),
\end{align}
we see that $a-l_{\infty}<0$ implies that $\cos(\Phi_{\infty})<0$. Therefore the fixed point equation for $R_{\infty}$ is given by the parabola
\begin{align}
    p(R_{\infty})=c_1R_{\infty}^2-c_2R_{\infty}-c_1, \ \ c_1,c_2>0.
\end{align}
Hence there is a stable negative root and the positive root given by $R_{\infty}=\frac{c_2+\sqrt{c_2^2+4c_1^2}}{2c_1}$ cannot represent a stable equilibrium. Since the variable $R(t)$ represents the ratio of the amplitudes of the oscillators, $R(t)\geq0$ is required, and hence there is no stable equilibrium in this regime.
\end{proof}
An important note about the above lemma is that this technique only works to prove that the equilibrium cannot be stable. Indeed, proving stability of the fixed points required more in depth analysis that took into account all three variables while performing the linearization.\\

This completes the proof of Theorem \ref{t:a1>a2>0}.

\section{Homogeneous natural frequency coupled Stuart-Landau Oscillators}\label{S:hom}
In this section we move beyond the $N=2$ case to provide full synchronization results for coupled Stuart-Landau oscillators with identical natural frequencies. The model for general $N$ is given by
\begin{equation}\label{SLfull}
    \ddt z_j=\left(\a_j+\i\w_j-|z_j|^2\right)z_j+\frac{\k}{N}\sum_{l=1}^N(z_l-z_j),
\end{equation}

Let us first see the conditions necessary for a fixed point to exist. Splitting into  the equations for amplitude and phase,
\begin{align}
    \ddt r_j&=(\a_j-r_j^2)r_j+\frac{\k}{N}\sum_{l=1}^N(\cos(\phi_l-\phi_j)r_l-r_j),\\
    \ddt \phi_j&=\omega_j+\frac{\k}{N}\sum_{l=1}^N\frac{r_l}{r_j}\sin(\phi_l-\phi_j),
\end{align}
let us first set the phase equation to zero, and multiply through by $r_j^2$,
\begin{align}
    0=\omega_jr_j^2+\frac{\k}{N}\sum_{l=1}^Nr_lr_j\sin(\phi_l-\phi_j),
\end{align}
summing over all $j$ yields
\begin{align}
    \bw=\frac{1}{\sum_{j=1}^Nr_{j,\infty}^2}\sum_{j=1}^N\w_jr_{j,\infty}^2=0,
\end{align}
is required in order to have a fixed point. Indeed, if $\bw \neq 0$, this corresponds to a phase locked state rotating at constant speed $\bw$, which depends on the asymptotic amplitude values, and hence is an emergent property. Although a priori it is not known what $r_{j,\infty}$ will be, by the rotational invariance of the system we can still shift to the rotational reference frame so that each $\omega_j \mapsto \omega_j-\bw$. We begin with an analysis of the system with all $\o_j\equiv \o$. With this we see that $\bw=\omega$, which no longer depends on the amplitudes, and that a fixed point only can exist for $\o=0$. This provides the $\o$-homogeneous formulation of the model.

\begin{equation}\label{SLhom}
    \ddt z_j=\left(\a_j-|z_j|^2\right)z_j+\frac{\k}{N}\sum_{l=1}^N(z_l-z_j),
\end{equation}
where $\k>0$ is the coupling strength, and $\a_j\in\R$ determines the desired amplitude of each of the oscillators. In this way all the oscillators want to synchronize to the same phase, while potentially having different amplitudes. A particular version of this setting was treated in \cite{DN1} in the context of large $N$ and $\a_j$ sampled from exactly two values ($a>0$ and $b<0$) where a numerical investigation lead to either full active synchronization or amplitude death (termed aging). In what follows we make no restriction on the finite size of $N$ and let $\a_j\in \R$ take any values, providing an analytic proof of convergence to phase synchronized states as well as an explicit phase transition from the Active state to the Amplitude Death state.

We can again split \eqref{SLhom} into equations for the amplitudes and phases,
\begin{align}
    \ddt r_j&=(\a_j-r_j^2)r_j+\frac{\k}{N}\sum_{l=1}^N\cos(\phi_l-\phi_j)r_l-r_j,\label{eq:shom1}\\
    \ddt \phi_j&=\frac{\k}{N}\sum_{l=1}^N\frac{r_l}{r_j}\sin(\phi_l-\phi_j).\label{eq:shom2}
\end{align}
Notice the similarity of the phase equation to the classical homogeneous Kuramoto model, however the amplitudes now play a role in the synchronization process. Indeed, the ratios of amplitude play the role of an adaptive complete graph network where oscillators with greater amplitudes are less sensitive to the effects of the oscillators with lesser amplitudes. Recall that we have defined $R_{jk}=\frac{r_j}{r_k}$ and $\Phi_{jk}=\phi_j-\phi_k$ so that we can more compactly write the equations as
\begin{align}
    \ddt r_j&=(\a_j-r_j^2)r_j+\frac{\k}{N}\sum_{l=1}^N\cos(\Phi_{lj})r_l-r_j,\label{eq:hom1}\\
    \ddt \phi_j&=\frac{\k}{N}\sum_{l=1}^NR_{lj}\sin(\Phi_{lj}).\label{eq:hom2}
\end{align}

First let us see that the amplitudes of each oscillator are eventually bounded by the largest value $\a_+=\max_j \a_j$.
\begin{align}\label{in:max1}
    \ddt r_+&=(\a_+-r_+^2)r_++\frac{\k}{N}\sum_{l=1}^N\cos(\phi_l-\phi_+)r_l-r_+, \nonumber\\
    &\leq (\a_+-r_+^2)r_+.
\end{align}

Thus all amplitudes are bounded from above, and further any fixed point satisfies $r_{j,\infty}^2\leq \a_+$ for all $j=1,...,N$. However, as for the $N=2$ case, Stuart-Landau oscillators have the potential to experience Amplitude death under certain configurations. Before investigating such conditions, lets us introduce the idea of sectorial solutions.

\subsection{Sectorial Solutions}

Each oscillator is evolving in the complex plane, however, note that if initially all oscillators begin on one side of a half-plane, then this property persists in time. Considering the complex plane as equivalent to $\R^2$, and letting $\ell \in \R^2$ be a unit functional such that its kernel is a line through the origin. Suppose initially that $\ell(z_j(0))\geq 0$ for all $j=1,...,N$, then letting $\ell(z_-(t))=\min_j\ell(z_j(t))$, and differentiating,
\begin{align}
    \ddt \ell(z_-(t))&=(\a_--r_-^2)\ell(z_-(t))+\frac{\k}{N}\sum_{l=1}^N(\ell(z_l(t))-\ell(z_-(t))),\\
    &\geq (\a_--r_-^2)\ell(z_-(t)).
\end{align}
Now integrating the above gives,
\begin{align}\label{invariance}
    \ell(z_-(t))\geq \ell(z_-(0))e^{c(t)}\geq 0.
\end{align}
\begin{lemma}[Sectorial Principle]
    Any solution to \eqref{SLhom} that begins in a sector
    \begin{align}
        \Sigma_{\cF}=\bigcap_{\ell \in \cF} \left\{\bz : \ell(\bz)\geq0\right\},
    \end{align}
    remains in the sector for all time.
\end{lemma}

In particular, if $\ell(z_j(0))>0$ for all $j=1,...,N$, then this is also preserved in time, although the possibility of asymptotic Amplitude Death could yield $\lim_{t\to \infty}\ell(z_j(t))=0$. However, as was seen in the case of $N=2$ oscillators, it is possible to characterize synchronization within Amplitude Death configurations via analysis of the phase difference variable $\Phi$ and the amplitude ratio $R$. Indeed, observing the phase equation \eqref{eq:hom2} we can see that as long as the ratios $R_{jk}=\frac{r_j}{r_k}$ are controlled, then we can continue to analyze the phase behavior.

\begin{lemma}\label{l:maxhom}
    Let $\{z_j\}_{j=1}^N$ be strictly sectorial solutions to \eqref{SLhom}. Then there exist constants $c$ and $C$ such that the ratios $R_{jk}(t)=\frac{r_j}{r_k}$ satisfy $0<c\leq R_{jk}(t)\leq C$, for all $t>0$.
\end{lemma}
\begin{proof}
    First note that from \eqref{in:max1} we have beyond some time $T>0$ each $r_j^2(t)\leq \a_++\e$. Hence we must show that if for some $j$, $r_j\to 0$, then all $r_k \to 0$ at the same rate so that $R_{jk}$ remains bounded.

    Let $R=\max_{j,k} R_{jk}$ and we compute with Rademacher's Lemma
    \begin{align}\label{eq:R1}
    \ddt R=R(\a_+-\a_-+r_-^2-r_+^2)+\frac{\k}{N}R\sum_{l=1}^N\left(\cos(\Phi_{l+})R_{l+}-\cos(\Phi_{l-})R_{l-}\right).
\end{align}

Now picking out the $+$ and $-$ index for $l$ within the sum, we get
\begin{align}
    \ddt R=R(\a_+-\a_-+r_-^2-r_+^2)+\frac{\k}{N}\cos(\phi_--\phi_+)\left(1-R^2\right)+\frac{\k}{N}R\sum_{l\neq -,+}\left(\cos(\Phi_{l+})R_{l+}-\cos(\Phi_{l-})R_{l-}\right).
\end{align}
Now as $\{z_j\}_{j=1}^N$ are strictly sectorial solutions, we have $\cos(\Phi_{jk})\geq c_0>0$ for all $j,k=1,...,N$. Further as $r_+(t)=\max_j r_j$ we have $R_{l+}\leq 1$. Therefore
\begin{align}
    \ddt R\leq R(\a_+-\a_-+r_-^2-r_+^2+\k)+\frac{\k}{N}c_0(1-R^2)
\end{align}
Therefore the negative leading order coefficient guarantees that $R$ is bounded from above and there exists $c,C>0$ such that $0<c\leq R_{jk}(t)\leq C$.
\end{proof}

Now let us compute two more useful equations for the system. First, defining the direction of each oscillator as $\tilde{z}_j=\frac{z_j}{r_j}$, we have
\begin{align}
    \ddt \tilde{z}_j=\frac{\k}{N}\sum_{l=1}^NR_{lj}\left(\tilde{z}_l-\tilde{z}_j\cos(\Phi_{lj})\right).
\end{align}

Further, we can derive an equation for the angle between two individual oscillators as $\cos(\Phi_{ij})=\tilde{z}_i\cdot\tilde{z}_j$,
\begin{align}\label{eq:angles}
    \ddt \cos(\Phi_{ij})&=\frac{\k}{N}\sum_{l=1}^NR_{lj}(\cos(\Phi_{il})-\cos(\Phi_{ij})\cos(\Phi_{lj})),\\
    &+\frac{\k}{N}\sum_{l=1}^NR_{li}(\cos(\Phi_{jl})-\cos(\Phi_{ji})\cos(\Phi_{li})).
\end{align}
With these we can prove full phase synchronization exponentially in time so that $\Phi_{ij}\to 0$ for all $i,j=1,...,N$.
\begin{lemma}\label{l:fullsynch}
    Suppose $\{z_j\}_{j=1}^N$ is a strictly sectorial solution to \eqref{SLhom}. Then the system undergoes full phase synchronization so that $\Phi_{ij}\to 0$ for all $i,j=1,...,N$ exponentially fast.
\end{lemma}
\begin{proof}
Without loss of generality assume all $z_j(0)$ are contained in the upper half-plane. This property is preserved in time, further letting $+,-$ represent the agents such that $\cos(\Phi_{ij})$ is minimized. Differentiating we have,

\begin{align}\label{eq:anglesmax}
    \ddt \cos(\Phi_{+-})&=\frac{\k}{N}\sum_{l=1}^NR_{l-}(\cos(\Phi_{+l})-\cos(\Phi_{+-})\cos(\Phi_{l-})),\\
    &+\frac{\k}{N}\sum_{l=1}^NR_{l+}(\cos(\Phi_{-l})-\cos(\Phi_{+-})\cos(\Phi_{l+})).
\end{align}
Now as
\begin{align}
    \Phi_{+-}=\phi_+-\phi_l+\phi_l-\phi_-=\Phi_{+l}+\Phi_{l-}<\pi-\d,
\end{align}
for some $\d>0$, since $\{z_j\}_{j=1}^N$ is a sectorial solution. Therefore
\begin{align}
    \cos(\Phi_{+l})-\cos(\Phi_{+-})\cos(\Phi_{l-})&=\cos(\Phi_{+-}+\Phi_{-l})-\cos(\Phi_{+-})\cos(\Phi_{l-}),\\
    &=\sin(\Phi_{+-})\sin(\Phi_{l-})\geq 0,
\end{align}
and similarly for the other summand. Therefore both sums are nonnegative, giving monotonic growth of $\cos(\Phi_{+-})$, and further, as each $R_{jk}$ is bounded from above and below by Lemma \ref{l:maxhom} we have for $R=\max_{jk} R_{jk}$
\begin{align}
    \ddt (1- \cos(\Phi_{+-}))\leq -c\frac{\k}{N R}\sum_{l=1}^N(\cos(\Phi_{+l})+\cos(\Phi_{l-}))(1-\cos(\Phi_{+-})),
\end{align}
where,
\begin{align}
   \cos(\Phi_{+l})+\cos(\Phi_{l-})=2\cos\left(\frac{1}{2}\Phi_{+-}\right)\cos\left(\frac{1}{2}(\Phi_{+l}+\Phi_{-l})\right)\geq c>0.
\end{align}
And we have,
\begin{align}\label{sync}
    \ddt (1-\cos(\Phi_{+-}))\leq-C\frac{1}{R}(1-\cos(\Phi_{+-})).
\end{align}
As $R$ is bounded Gr\"onwall's inequality grants $\Phi_{+-}\to 0$ exponentially fast.
\end{proof}

Now let us examine the amplitude behavior. As $\cos(\Phi_{jk})\to 1$ exponentially fast for each $j,k=1,...,N$, the amplitude equation can be written as

\begin{align}
    \ddt r_j=(\a_j-r_j^2)r_j+\frac{\k}{N}\sum_{l=1}^N(r_l-r_j)+E_j(t),\label{eq:amppert}
\end{align}
where $E_j(t)$ are exponentially decaying terms from the phase synchronization. Ignoring the exponentially decaying term for the moment. This system can be viewed as a first order dynamical opinion game which has associated payout function,
\begin{equation}\label{eq:pNFJ}
    p_j(\br)=\frac{\k}{2N}\sum_{l=1}^N (r_l-r_j)^2+\left(\frac{1}{4}r_j^{4}-\frac{1}{2}\a_jr_j^2\right).
\end{equation}
Further the system is governed by a perturbed gradient flow structure
\begin{equation}\label{eq:GF}
\ddt \br=-\nabla P(\br) + \bE(t),
\end{equation}
for $\bE(t)=(E_1(t),...,E_N(t))$ each of the exponentially decaying terms from the phase synchronization, and
\begin{align}
P(\br)=\frac{\k}{4N}\sum_{i,j=1}^N(r_i-r_j)^2+\frac{1}{4}\sum_{i=1}^N r_i^{4}-\frac{1}{2}\sum_{i=1}^N\a_ir_i^2.
\end{align}
Although $P(\br)$ is not globally convex, for the case of $\a_j>0$, for each $j=1,...,N$, existence and uniqueness of a stable fixed point for which any trajectory eventually converges to was shown in \cite{LRS}. Further given the complete graph structure and $\a_j>0$ there is an ordering for the fixed point so that 
\begin{align}
    0<\a_-\leq r_{1,\infty}^2\leq ...\leq r_{j,\infty}^2\leq r_{j+1,\infty}^2\leq ...\leq r_{N,\infty}^2\leq \a_+.
\end{align}

On the other hand, if all $\a_j< 0$, then we see from \eqref{in:max1} and Lemma \ref{l:maxhom} that all $r_j \to 0$ at the same exponential rate. The next subsection is dedicated to dealing with the case of mixed $\a_j$ so that some oscillators are subcritical and some are supercritical.\\

\subsection{Mixed parameters}
In the mixed case where there exist $\a_j<0$ and $\a_k>0$, we still achieve exponential phase synchronization for sectorial solutions via Lemma \ref{l:fullsynch}. However, notice that the Amplitude Death fixed point so that $r_{j,\infty}\equiv0$ for all $j=1,...,N$ always exists. Let us now determine under what conditions is Amplitude Death stable.

First, let us write the model without the exponentially decaying terms.
\begin{align}\label{eq:ramps}
    \ddt r_j=(\a_j-r_j^2)r_j+\frac{\k}{N}\sum_{l=1}^N(r_l-r_j),
\end{align}
for $\a_j \in \R$, and $r_j\in \R_{\geq0}$.

\begin{lemma}\label{msad}
    Let $\{r_j\}_{j=1}^N$ be a solution to \eqref{eq:ramps}. Then the Amplitude Death fixed point 
 so that $r_{j,\infty}\equiv 0$ is stable if all of the following conditions hold:
    \begin{align}
        &\sum_{j=1}^N \a_j<0, \label{cond1}\\
        &\k>\a_j \ \forall j, \label{cond2}\\ 
        &\sum_{j=1}^N \frac{\a_j}{\k-\a_j}<0. \label{cond3}
    \end{align}  
\end{lemma}
\begin{proof}
Notice that the above encompasses the already known situations where all $\a_j<0$. Now let us define the fixed point map
\begin{align}\label{eq:FMap}
    F(\br)_j=\k r_j-\frac{\k}{N}\sum_{l=1}^Nr_l+(r_j^2-\a_j)r_j.
\end{align}
Indeed, $F(\br)=0$ corresponds to a fixed point of the system, where we have changed the sign for convenience. Now, if all the eigenvalues of the linearization have positive real part, then the fixed point is stable.  Computing the Jacobian matrix yields
\begin{align}\label{Jacobian}
    D_{\br}F(\br)=G-M,
\end{align}
where $G$ is a diagonal matrix given by $\{g_j\}_{i=1}^N=\k+3r_j^2-\a_j$ and $M$ is the constant matrix with values $M_{ij}=-\frac{\k}{N}$. Therefore the determinant of the matrix is
\begin{align}\label{eq:det}
    \det{D_{\br}F(\br)}=\prod_{j=1}^Ng_j\left(1-\frac{\k}{N}\sum_{j=1}^N\frac{1}{g_j}\right).
\end{align}
Now to analyze the stability of the zero fixed point, $r_j=0$ for all $j=1,...,N$, we have $g_j=\k-\a_j$. Indeed, $\k>\a_j$ for all $j=1,...,N$, is required to have the main diagonal positive, and in this way  $\prod_{j=1}^Ng_j>0$ so that we need only determine the sign of $\left(1-\frac{\k}{N}\sum_{j=1}^N\frac{1}{g_j}\right)$.

Now let us rewrite
\begin{align}\label{eq:ADcondition}
    \frac{\k}{N}\sum_{j=1}^N\frac{1}{g_j}=\frac{1}{N}\sum_{j=1}^N \frac{\k}{\k-\a_j}=1+\frac{1}{N}\sum_{j=1}^N \frac{\a_j}{\k}+\frac{\a_j^2}{\k^2-\a_j\k}.
\end{align}
As $\k>\a_j$ for all $j$, we have $\frac{\a_j^2}{\k^2-\a_j\k}>0$ and hence to guarantee stability of the zero fixed point we must have $\sum_{j=1}^N \a_j<0$. Further we need
\begin{align}
    \frac{1}{N}\sum_{j=1}^N \frac{\a_j}{\k}+\frac{\a_j^2}{\k^2-\a_j\k}= \frac{1}{N}\sum_{j=1}^N\frac{\a_j}{\k-\a_j}<0,
\end{align}
which is what gives us the final condition that $\sum_{j=1}^N\frac{\a_j}{\k-\a_j}<0$. In fact this shows that the condition \eqref{cond3} is strictly stronger than condition \eqref{cond1}, however, given the necessity of condition \eqref{cond2} to hold to be able to write down \eqref{cond3}, we include all three conditions.

In this regime we see that the main diagonal of the matrix is positive and the determinant is also positive. Performing the same computation on the upper left minors, $n<N$,
\begin{align}
    M_n=\prod_{j=1}^n g_j\left(1-\frac{\k}{N}\sum_{j=1}^n\frac{1}{g_j}\right)>0,
\end{align}
grants the stability of the zero fixed point exactly under the conditions of Lemma \ref{msad}.
\end{proof}

Now, the last regime is proving stability of the Active fully synchronized state when one of the conditions of Lemma \ref{msad} fails to hold.

\begin{lemma}\label{l:actfullsynch}
    Let $\{r_j\}_{j=1}^N$ be a solution to \eqref{eq:ramps}. Then there is a unique stable Active fixed point 
 so that $r_{j,\infty}> 0$ for each $j$, if at least one of the following conditions hold:
    \begin{align}
        &\sum_{j=1}^N \a_j>0, \label{2cond1}\\
        &\k<\a_j \ \ \text{for some } j, \label{2cond2}\\ 
        &\sum_{j=1}^N \frac{\a_j}{\k-\a_j}>0. \label{2cond3}
    \end{align}  
\end{lemma}
\begin{proof}
    First let us see that although Amplitude Death at $r_{j,\infty}\equiv 0$ is always a solution, that under any of the above assumptions \eqref{2cond1}-\eqref{2cond3}, that this state must be unstable.

    Let us first suppose that \eqref{2cond2} holds. Now for this particular $j$ we have
    \begin{align}
        \ddt r_j&=(\a_j-r_j^2)r_j+\frac{\k}{N}\sum_{l=1}^N(r_l-r_j),\\
        &=(\a_j-\k-r_j^2)r_j+\frac{\k}{N}\sum_{l=1}^Nr_l
    \end{align}
    Let $\varepsilon=\frac{1}{2}(\a_j-\k)>0$, then if $r_j^2<\varepsilon$ we have $\ddt r_j>0$. Therefore from Lemma \ref{l:maxhom} all $r_j$ remain active.

    Now suppose that $\k>\a_j$ for all $j$. Note that the computation for the Jacobian in the previous lemma is unchanged. Therefore from \eqref{eq:ADcondition} we see that if either \eqref{2cond1} or \eqref{2cond3} holds, then \eqref{eq:det} yields $\det{D_{\br}F(\br)}<0$ and hence at least one of the eigenvalues is negative and the Amplitude Death fixed point is unstable.

    With that in hand we are seeking a fixed point such that $r_{j,\infty}>0$ for all $j=1,...,N$. Therefore such a fixed point would satisfy
    \begin{align}
        0=(\a_j-r_{j,\infty}^2)r_{j,\infty}+\frac{\k}{N}\sum_{l=1}^N(r_{l,\infty}-r_{j,\infty})
    \end{align}
    as all $r_{j,\infty}>0$ we can divide through by $r_{j,\infty}$ and solve for $r_{j,\infty}^2$,
    \begin{align}
        r_{j,\infty}^2=\a_j-\k+\frac{\k}{N}\sum_{l=1}^NR_{lj}^{\infty}
    \end{align}
    where we recall that $R_{lj}=\frac{r_l}{r_j}$. Now in the computation of the determinant of the Jacobian we can plug in this fixed point for the $r_j^2$ in the formula for $g_j=\k+3r_j^2-\a_j$ so that
    \begin{align}
        g_j=\frac{3\k}{N}\sum_{l=1}^NR_{lj}^{\infty}>0.
    \end{align}
    Thus as $\det{D_{\br}F(\br)}=\prod_{j=1}^Ng_j\left(1-\frac{\k}{N}\sum_{j=1}^N\frac{1}{g_j}\right)$, we again need to show that $\frac{\k}{N}\sum_{j=1}^N\frac{1}{g_j}<1$.
    \begin{align}
        \frac{\k}{N}\sum_{j=1}^N\frac{1}{g_j}=\frac{1}{3}\sum_{j=1}^N\frac{1}{\sum_{l=1}^NR_{lj}^{\infty}}=\frac{1}{3}\sum_{j=1}^N\frac{r_j^{\infty}}{\sum_{l=1}^Nr_{l^{\infty}}}=\frac{1}{3}<1.
    \end{align}
    Thus $\det{D_{\br}F(\br)}>0$ and the computation of the principal minors similarly gives $M_n>0$ guaranteeing stability of the Active fixed point. In order to prove existence and uniqueness we utilize the Brouwer topological degree, see \cite{Cronin}. For a particular region $\cW$, the topological degree of a point $\bx$ is defined as
\begin{equation}
    \mathrm{deg}\{F,\cW,\bx\}=\sum_{\br \in F^{-1}(\boldsymbol{x})} \mathrm{sgn} (\det{D_{\br}F(\br)}).
\end{equation}
Now, since we have determined that  $\mathrm{sgn} (\det{D_{\br}F(\br)})=1$, for any $\br \in F^{-1}(\boldsymbol{0})$, if we can show that the degree is exactly one, then there must be a unique fixed point within the region  $\cW$.

Let us denote
\begin{align}
\langle \bx,\by\rangle= \sum_{j=1}^N  x_j y_j, \ \ \ \ \ \|\bx\|_3^3=\sum_{i=1}^N x_j^3.
\end{align}
We define,
\begin{align}
\cW=\{\bx: x_j \geq 0, \ \e\leq \|\bx\|_{\infty}, \|\bx\|_{3} \leq C\},
\end{align}
where $C>0$ is large, and $\e$ small, to be determined momentarily. We verify that the image of the boundary does not contain the origin, $0 \not\in F(\partial \cW)$. Recall that
\begin{align}
    F(\br)_j=-(\a_j-r_j^2)r_j-\frac{\k}{N}\sum_{l=1}^N(r_l-r_j).
\end{align}
First if $r_j=0$ for some $j$ then $F(\br)_j=-\frac{\k}{N}\sum_{l=1}^Nr_l<0$ since $\|\br\|\geq\e$. Now to check the other two parts of the boundary, if $\|\br\|_{\infty}=\e$ then there is a particular $j$ such that $r_j=\e$ and as $\hat{\br}=(r,...,r)$ for some $r\in \R_+$ is not a fixed point unless all $\a_j\equiv \a$, we also have particular $k$ such that $r_k>\e$. Thus
\begin{align}
    F(\br)_j=(\k-\a_j)\e+\e^3-\frac{\k}{N}\sum_{l=1}^Nr_l<0,
\end{align}
for $\e>0$ small enough.

On the other hand if $\|\br\|=C$ then
\begin{align}
    \frac{1}{N}\sum_{j=1}^N F(\br)_j=-\langle \alpha,\br\rangle+\|\br\|_3^3\geq \|\br\|_3^3-\|\br\|_3\|\a\|_{\frac{3}{2}}>0,
\end{align}
for $C$ large enough.

Therefore the value $\boldsymbol{0}$ of $F$ is regular, and its degree can be computed explicitly by
\begin{align}
\mathrm{deg}\{F,\cW,\boldsymbol{0}\}=\sum_{\br \in F^{-1}(\boldsymbol{0})} \mathrm{sgn} (\det{D_{\br}F(\br)}).
\end{align}
However, we have proved that all Jacobians for $\br \in F^{-1}(\boldsymbol{0})$ are strictly positive. Therefore uniqueness can be shown by proving $\mathrm{deg}\{F,\cW,\boldsymbol{0}\}=1$.

This is certainly true for $\hat{\boldsymbol{\alpha}}=(\a,\dots,\a)$ for a fixed $\a>0$. Indeed, this is because for identical inherent amplitudes, consensus of amplitude values is achieved at exactly the square root of the Hopf-parameter, $\sqrt{\a}$. Now, fix any such $\hat{\boldsymbol{\a}}$ and consider the homotopy of maps
\begin{align}
F^{(\tau)}:=F_{\tau \boldsymbol{\a} + (1-\tau)\hat{\boldsymbol{\a}}}.
\end{align}
First, the starting point $\hat{\boldsymbol{\a}}$ satisfies \eqref{2cond1}, and \eqref{2cond3} automatically. If the final fixed point we are investigating comes from only equation \eqref{2cond2} holding, then pick $\hat{\boldsymbol{\a}}$ to be equal to one of the $\a_j$ for which $\k<\a_j$ holds. In this way $\boldsymbol{0} \not\in F^{(\tau)}(\partial \cW)$ for any $\tau$, and the Invariance under Homotopy Principle applies and hence,
\begin{equation}
\mathrm{deg}\{F_{\boldsymbol{\a}},\cW,\boldsymbol{0}\}=\mathrm{deg}\{F_{\hat{\boldsymbol{\a}}},\cW,\boldsymbol{0}\}=1,
\end{equation}
and the proof of existence and uniqueness is finished.
\end{proof}

With this in hand we are ready to state the main theorem for this section.
\begin{theorem}\label{t:N>2synch2}
    Let $2\leq N<\infty$, $\a_j \in \R$ and $\w_j\equiv 0$ for each $j=1,...,N$. Let $\{z_j\}_{j=1}^N$ be strictly sectorial solutions to \eqref{SLhom}, then for all $\k>0$, $\max_{j,k} \Phi_{jk}\to 0$ exponentially fast, and the amplitude dynamics have each $r_j \to r_j^{\infty}$, where one of the two cases can occur:
    \begin{itemize}
        \item Amplitude death: $r_j^{\infty}=0$ for all $j=1,...,N$ if all the following conditions hold
        \begin{itemize}
        \item $\sum_{j=1}^N\a_j<0$,
        \item $\k>\a_j$, $\forall j=1,...,N$
        \item $\sum_{j=1}^N \frac{\a_j}{\k-\a_j}<0$
        \end{itemize}
        \item Active: $r_j^{\infty}>0$ for each $j=1,...,N$ if at least one of the following conditions hold:
    \begin{itemize}
        \item $\sum_{j=1}^N \a_j>0$,
        \item $\k<\a_j \ \ \text{for some } j$,
        \item $\sum_{j=1}^N \frac{\a_j}{\k-\a_j}>0.$
    \end{itemize}  
    \end{itemize}
    Further, $r_{j,\infty}\leq r_{j+1,\infty}$ for all $j$ and if all $\a_j>0$, then $r_j^{\infty} \in [\min_j \sqrt{\a_j},\max_j \sqrt{\a_j}]$.
\end{theorem}
\begin{proof}
    Lemma \ref{l:fullsynch} yields $\Phi_{jk} \to 0$ exponentially fast for strictly sectorial solutions, regardless of $\boldsymbol{\a}$ configuration. Lemmas \ref{msad} and \ref{l:actfullsynch} provide the existence, uniqueness and stability of the fixed point within the Amplitude Death and Active Regimes respectively for the system without the exponentially decaying contribution from the synchronization part $\Phi_{jk}\to 0$. Indeed the exponentially decaying term $\bE(t)$ does not affect the existence uniqueness or stability analysis. The final piece is to see convergence to the fixed point. Convergence is achieved via the gradient flow structure given by \eqref{eq:GF} and utilizing the Lojasiewicz gradient inequality \cite{L}. This type of convergence has already been shown in \cite{RT} for a similar system of opinion dynamics which will be studied in the following section. While the exponentially decaying term is appropriately dealt with in \cite{LRS}. We refer the reader to these sources for details.
\end{proof}

\begin{remark}
    Note that we did not treat the edge cases between Lemmas \ref{msad} and \ref{l:actfullsynch}. Indeed, these transitions represent the phase transition from having an Active synchronized state to an Amplitude Death synchronized state. 
\end{remark}

The results of this section provide full synchronization for $\omega$-homogeneous oscillators, but for initial data confined to one half-plane. In the final section we study the restriction of Stuart-Landau oscillators to the real line.

\section{Stuart-Landau on the Real line: A model of Opinion dynamics}\label{S:opinion}
The real-valued version of the Stuart-Landau model \eqref{eq:SL} reads
\begin{align}\label{eq:SLopinion1}
    \ddt x_j=(\a_j-x_j^2)x_j+\frac{\k}{N}\sum_{l=1}^N(x_l-x_j), \ \ j=1,..,N
\end{align}
with $\a_j \in \R$, $\k\geq0$ fixed and $x_j\in \R$.
In \cite{RT}, the system \eqref{eq:SLopinion} is proposed as a nonlinear model of opinion dynamics. In that model, the inherent amplitudes are used to model stubbornness parameters that anchored an opinion to that particular positive value. It was introduced as a nonlinear consensus model in line with the famous Degroot \cite{DeGroot} and Friedkin-Johnsen models \cite{FJ}. Within both \cite{RT,LRS} and the previous Section \ref{S:hom}, a restriction on the initial conditions leads to convergence to a positively oriented stable fixed point which is unique within that region. Such a fixed point is a Nash Equilibrium: a point where agents do not necessarily agree, but no agent can do better in the dynamical opinion game by changing their opinion. However, if we allow initial data to take both positive and negative values, we can see other stable asymptotic states can exist in weak coupling regimes. Further, allowing for the existence of subcritical oscillators (agreeable agents, $\a_j\leq 0$), then consensus in amplitude death can be recovered despite heterogeneity of stubbornness parameters.

First let us recall what the potential asymptotic states are of the model \eqref{eq:SLopinion1}.

\begin{definition}[Disagreement, Compromise, and Consensus]
    The asymptotic states $x_j^{\infty} \in \R$ of \eqref{eq:SLopinion1} can be characterized as
    \begin{itemize}
        \item Disagreement: If there exists $j,k$ such that $x_j^{\infty}<0$ and $x_k^{\infty}>0$.
        \item Compromise: If for all $j=1,..,N$, $x_j^{\infty}>0$ $(x_j^{\infty}<0)$ and there exists $k,l$ such that $x_k^{\infty}\neq x_l^{\infty}$.
        \item Consensus: If for all $j=1,...,N$, $x_j^{\infty}\equiv c$.
        \begin{itemize}
            \item Balanced Consensus: If $c=0$.
        \end{itemize}
    \end{itemize}
\end{definition}

Further we recall below the main result from \cite{RT} adapted to the current model in \eqref{eq:SLopinion1}:

\begin{theorem}[\cite{RT}]\label{t:opinion2}
For any set of parameters $\{\a_j\}_{j=1}^N$, such that $\a_j>0$, there exists a unique steady state $\bx^*\in \R^N_+$ to \eqref{eq:SLopinion} which is a locally exponentially stable Nash equilibrium. Moreover, any solution $\bx(t)\in \R^N_+$ with $\bx(0)\in \R^N_+$ converges to the unique Nash equilibrium.
\end{theorem}
 The original result is more general incorporating more aspects of opinion dynamics as well as allowing for a generic underlying graph topology of interactions between the oscillators. However, for the purpose of this work we focus on the case of fully interacting oscillators (corresponding to an underlying complete graph topology) and limit the number of the model's parameters to those relevant for the current study, characterizing model \eqref{eq:SLopinion}. 

The above result assumes positive initial data $x_j(0)$ and positive stubborness parameters $\alpha_j> 0$. The Nash equilibrium is a Compromise state if there is $\alpha$-heterogeneity, and Consensus if all $\alpha_j\equiv\a>0$. The exclusion of negative initial data along with the sectorial invariance of the previous section precluded the possibility of Disagreement states. In what follows, instead, we will study the case of both positive and negative initial values, $x_j(0) \in \R$, as well as the inclusion of nonpositive stubborness parameter values, $\a_j \in \R$. Values $\a_j\leq 0$ are considered agreeable or (moderates) as the forcing mechanism $(\a_j-x_j^2)x_j$ always drives the opinion value towards zero. In this way we will see the existence of stable Disagreement states as well as the recovery of a Consensus state while having $\a$-heterogeneity.

In the next subsection we see the effect of allowing for negative initial data, while for the moment continuing to have $\a_j>0$.

\subsection{Negative initial data: the two agents case} Let us begin with a simple situation involving only two agents with the same stubbornness parameter $\a_1=\a_2=\a>0$. Model \eqref{eq:SLopinion} then becomes
\begin{align}
    \ddt x_1&=(\a-x_1^2-\frac{\k}{2})x_1+\frac{\k}{2}x_2,\label{2coup1}\\
    \ddt x_2&=(\a-x_2^2-\frac{\k}{2})x_2+\frac{\k}{2}x_1.\label{2coup2}
\end{align}
Now for $\k=0$ the system is uncoupled and $x_1,x_2$ behave independently, and converge to one of their 3 possible equilibria, respectively, i.e. there are nine equilibria given by $(\pm\sqrt{\a}$, $\pm\sqrt{\a})$, $(\pm\sqrt{\a},0)$, $(0,\pm\sqrt{\a})$, $(0,0)$, for the following uncoupled system,
\begin{align}
    \ddt x_1&=(\a-x_1^2)x_1,\label{2uncoup1}\\
    \ddt x_2&=(\a-x_2^2)x_2.\label{2uncoup2}
\end{align}
Four of these fixed points are stable, $(\pm\sqrt{\a},\pm\sqrt{\a})$, four are saddle points, $(\pm\sqrt{\a},0),(0,\pm\sqrt{\a})$, and $(0,0)$ is fully unstable. Indeed the Jacobian is given by
\[
J(\bx)=\begin{bmatrix}
\a-3x_1^2 & 0 \\
0 & \a-3x_2^2 \\
\end{bmatrix}.
\]
Plugging in each of the nine fixed points grants the stability for each of them. Further, as $\det{J(\bx)}\neq 0$ for each fixed point, they are all hyperbolic fixed points. Thus considering $\k$ small as a perturbation, we know that all nine of the fixed points of of \eqref{2uncoup1}-\eqref{2uncoup2} are continuously shifted with the perturbation parameter $\k$ so that there are also nine fixed points of \eqref{2coup1}-\eqref{2coup2} which all retain the stability of the original fixed points, \cite{HKBook}. However, note that any fixed point satisfies
\begin{align}
    x_1(\k-\a+x_1^2)=\frac{\k}{2}(x_1+x_2)=x_2(\k-\a+x_2^2),
\end{align}
which for $\k>\a$, implies $\mathrm{sgn}(x_1)=\mathrm{sgn}(x_2)$. Thus there must exist at least one value $\k'(\a)\leq \a$ such that a bifurcation occurs eliminating the six fixed points for which $\mathrm{sgn}(x_1)\neq\mathrm{sgn}(x_2)$. In fact we will find two bifurcation values $\k_1'(\a)<\k_2'(\a)$. The following lemma provides us with the desired information about the bifurcations.
\begin{lemma}\label{l:2bifurcate}
    Let $\{x_1,x_2\}$ be a solution to \eqref{2coup1}-\eqref{2coup2}. Then if $\k<\frac{2}{3}\a$ there are exactly nine fixed points, of which four are stable, four are saddle points and one is repulsive. The system undergoes a pair of triple saddle-node bifurcations at $\k_1'(\a)=\frac{2}{3}\a$ where all six fixed points for which $\mathrm{sgn}(x_1)\neq\mathrm{sgn}(x_2)$ meet such that $x_j^2(\k_1')=\frac{1}{3}\a$. A second bifurcation occurs at $\k_2'(\a)=\a$ where the now two fixed points such that $\mathrm{sgn}(x_1)\neq\mathrm{sgn}(x_2)$ collapse to the zero fixed point and for $\k>\k'$ there exist only the three fixed points for which $\mathrm{sgn}(x_1)=\mathrm{sgn}(x_2)$ where the nonzero ones are stable and the zero fixed point is unstable.
\end{lemma}
\begin{proof}
    As the nine fixed points of \eqref{2uncoup1}-\eqref{2uncoup2} are hyperbolic, it is guaranteed for small values of $\k$ that these nine fixed points are preserved along with their stability. Let us now seek a value $\k'$ such that the hyperbolicity breaks down, i.e. fixed points such that at $\k'$ we have
    \begin{align}
        0=(\a-x_1^2-\frac{\k}{2})x_1+\frac{\k}{2}x_2,& \ \ \ \ \ 
    0=(\a-x_2^2-\frac{\k}{2})x_2+\frac{\k}{2}x_1,\label{bifurcate1}\\
    \det J(\bx)&=0,\label{bifurcate2}
    \end{align}
    where for the coupled system we have
    \[
J(\bx)=\begin{bmatrix}
\a-3x_1^2-\frac{\k}{2} & \frac{\k}{2} \\
\frac{\k}{2} & \a-3x_2^2-\frac{\k}{2} \\
\end{bmatrix}.
\]
Let us begin with \eqref{bifurcate1}. Note that $x_1=x_2=0$ is always a fixed point. Now by adding the two equations there we get the following condition for a fixed point.
\begin{align}\label{fpprod}
    0=(x_1+x_2)(\a-(x_1^2-x_1x_2+x_2^2)).
\end{align}
The first part of the product prompts us to investigate what happens when $x_1=-x_2$. Returning to the first equation of \eqref{bifurcate1} yields
\begin{align}
    0=(\a-\k-x_1^2)x_1,
\end{align}
which gives $x_1^2=\a-\k=x_2^2$ (or $x_1^2=0=x_2^2$ which has already been accounted for). Therefore we get the two points $(\pm \sqrt{\a-\k},\mp \sqrt{\a-\k})$ while $\k<\a$. Note that this automatically yields the bifurcation value $\k_2'(\a)=\a$ for when these fixed points collide with the zero fixed point.\\

Returning to \eqref{fpprod} the second part of the product gives the equation for an ellipse,
\begin{align}\label{ellipse}
    \a=x_1^2-x_1x_2+x_2^2.
\end{align}
The ellipse has been rotated $\frac{\pi}{4}$ radians so that the vertex of the main axis is found at $x_1=\pm\sqrt{\a}=x_2$ and the vertex of the minor axis at $x_1=\pm\sqrt{\frac{\a}{3}}=\mp x_2$. This gives the first bifurcation value $\k_1'(\a)=\frac{2\a}{3}$ by plugging in $x_1^2=\frac{\a}{3}$ and $x_1=-x_2$ into \eqref{bifurcate1},
\begin{align}
    0=(\frac{2}{3}\a-\k)\sqrt{\frac{\a}{3}}.
\end{align}

Therefore, any fixed point must be found either on the ellipse \eqref{ellipse}, or at the two points where $(x_1,x_2)=(\pm \sqrt{\a-\k},\mp \sqrt{\a-\k})$ while $\k<\a$, or at $x_1=x_2=0$. Note that all nine of the fixed points of \eqref{2uncoup1}-\eqref{2uncoup2} can be found there as well. Note that at the points we computed we still need to verify that \eqref{bifurcate2} is satisfied.\\

Computing the determinant  yields
\begin{align}
    \det J(\bx)=\a^2-\a\k-3(\a-\frac{\k}{2})(x_1^2+x_2^2)+9x_1^2x_2^2.
\end{align}
Plugging in the values found at $\k_1'(\a)=\frac{2\a}{3}$ with $x_1^2=\frac{\a}{3}$ and $x_1=-x_2$ as well as at $\k_2'(\a)=\a$ with $x_1=x_2=0$ both grant $\det J(\bx)=0$, while any other point on the ellipse or the line $x_1=-x_2$ yield $\det J(\bx)\neq 0$. Therefore there are exactly two bifurcation points and for $\k>\a$ the only fixed points are the stable Nash equilibria from Theorem \ref{t:opinion2} representing consensus at $(x_1,x_2)=(\pm \sqrt{\a},\pm \sqrt{\a})$ and the unstable zero fixed point.
\end{proof}
\begin{figure}
    \centering
    \includegraphics[width=0.75\linewidth]{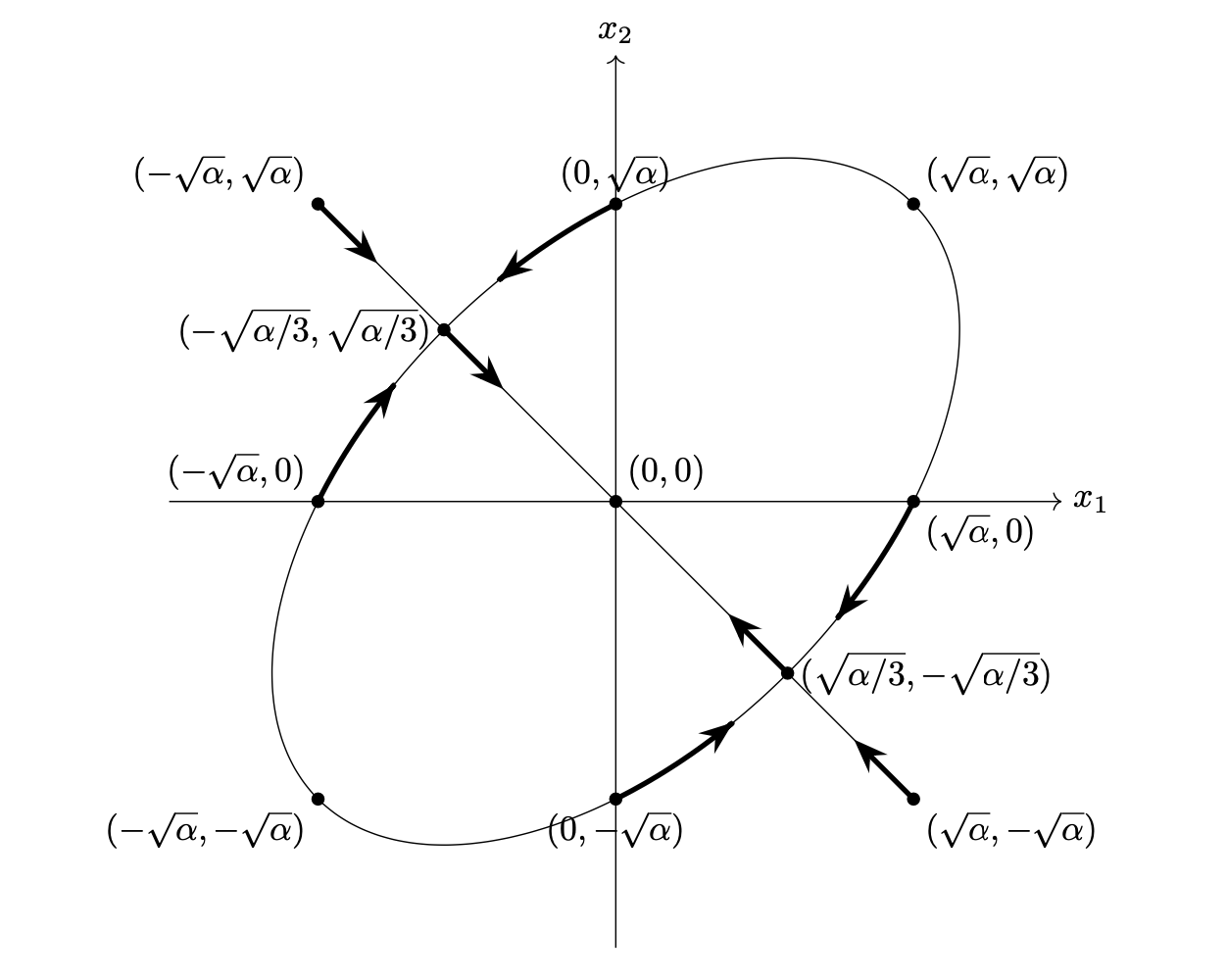}
    \caption{This figure demonstrates the evolution of the the nine fixed points of \eqref{2coup1}-\eqref{2coup2} as $\k$ increases from $\k=0$ to $\k=\a$ via the two bifurcation points $\k_1'$ and $\k_2'$. The fixed points $(\sqrt{\a},\sqrt{\a})$ and $(-\sqrt{\a},-\sqrt{\a})$ are stable and never move for any value of $\k>0$. The fixed point $(0,0)$ is unstable and never moves as well. The fixed points $(-\sqrt{\a},\sqrt{\a})$ and $(\sqrt{\a},-\sqrt{\a})$ are also stable, but as $\k$ increases they move along the line $x_1=-x_2$. Meanwhile the four saddle points which begin on each of the axes at distance $\sqrt{\a}$ away from the origin move along the ellipse given by \eqref{ellipse} in the direction of the line $x_1=-x_2$. The first bifurcation occurs at $\k_1'(\a)=\frac{2\a}{3}$ where the two saddle points in the second quadrant collide with the stable node at $(-\sqrt{\a/3},\sqrt{\a/3})$ producing one saddle point. And similarly in the fourth quadrant. The two remaining saddle points continue along the line $x_1=-x_2$ until at $\k_2'(\a)=\a$ it reaches the point $(0,0)$ and becomes a fully unstable fixed point.}
    \label{fig:Bifurcation_Ellipse}
\end{figure}
Figure \ref{fig:Bifurcation_Ellipse} gives the evolution of the fixed points as $\k$ increases giving rise to the two bifurcation points $\k_1'=\frac{2\a}{3}$ and $\k_2'=\a$.

\subsection{The general case with $N$ agents}
If all $\a_j>0$, then the previous argument for $N=2$ still holds, yielding the existence of multiple stable fixed points which are Disagreement states, where some $x_j<0$ and some $x_j>0$. Indeed, as the perturbation breaks down as the coupling strength increases, the basins of attraction of these fixed points shrink to zero and collide with the saddle points. There is a cascade of saddle-node bifurcations until for $\k\geq\a_{+}$ the only fixed points are the two stable positively (negatively) oriented Compromise (Consensus if all $\a_j\equiv \a$) states and the unstable repulsive Consensus state at zero.

However, letting $\a_j$ take negative values as well can lead to more diverse outcomes. In the context of opinion dynamics, $\a_j>0$ represents stubborn agents that want to keep their initial opinion values as close to $\sqrt{\a_j}$ if initially positive and $-\sqrt{\a_j}$ if initially negative. Meanwhile $\a_j< 0$ would indicate an agent that is not stubborn, but agreeable and always wants to be in the center.

Let us now investigate the linearization around the fixed points.

In this case we let the fixed point map be
\begin{align}
    F_j(\bx)=(\a_j-x_j^2-\k)x_j+\frac{\k}{N}\sum_{l=1}^Nx_l.
\end{align}
Then the Jacobian matrix is given  by
\[
D_{\bx}F(\bx)=\begin{bmatrix}
\a_1-3x_1^2-\frac{\k(N-1)}{N} & \frac{\k}{N} & \dots & \frac{\k}{N}\\
\frac{\k}{N} & \a_2-3x_2^2-\frac{\k(N-1)}{N} &  & \vdots \\
\vdots &  & \ddots & \vdots \\
\frac{\k}{N} & \dots & \dots & \a_N-3x_N^2-\frac{\k(N-1)}{N}\\
\end{bmatrix}
\]
\[
=\begin{bmatrix}
\a_1-3x_1^2 &  &  & \\
 & \a_2-3x_2^2 &  & \\
 &  & \ddots &  \\
 & &  & \a_N-3x_N^2\\
\end{bmatrix}-\frac{\k}{N}\begin{bmatrix}
N-1 & -1 & \dots & -1\\
-1 & N-1&  & \vdots \\
\vdots &  & \ddots & \vdots \\
-1 & \dots & \dots & N-1\\
\end{bmatrix}=A-E.
\]

Now, for the uncoupled system with $\k=0$, the eigenvalues are given by $\a_j-3x_j^2$, where the stable fixed points are achieved if for $\a_j>0$, then $x_j=\pm\sqrt{\a_j}$, and for $\a_j<0$, then $x_j=0$. Indeed, if $\a_j<0$, then the only steady state for an uncoupled oscillator is $x_j=0$, as this is the subcritical regime of the pitchfork bifurcation.

For $\k$ small, we know that as long as $\a_j\neq 0$ that every fixed point is hyperbolic and the existence and stability is retained up to each saddle-node bifurcation like the one in Lemma \ref{l:2bifurcate}. %We can further see, that as long as we assume $\a_j\neq \a_k$ for each $j,k=1,...,N$, then each eigenvalue of $D_{\bx}F(\bx)|_{\k=0}$ is simple and we can explicitly compute the the perturbation of each eigenvalue,
%\begin{align}
%    \tilde{\l}_j&=\l_j(A)-\frac{\k(N-1)}{N}+O(\|E\|^2),\\
 %   &=\a_j-3\tilde{x}_j^2-\frac{\k(N-1)}{N}+O(\|E\|^2).
%\end{align}

For general $N$ and with $\a$-heterogeneity it becomes too unwieldy to find all of the individual bifurcations. However, we can give a lower bound on when the final bifurcation has occurred.

\begin{lemma}\label{l:k*}
    There exists a $\k^*\geq 0$ such that for all $\k>\k^*$, the stable asymptotic state is either,
    \begin{itemize}
        \item Consensus: All $x_j\equiv 0$, if $\sum_{j=1}^N \a_j<0$,
        \item Compromise: All $x_j>0$ $(x_j<0)$, if $\sum_{j=1}^N \a_j>0$,
    \end{itemize}
     The Compromise state further can be a positive (negative) Consensus state if $\a_j\equiv \a>0$ for all $j=1,...,N$. Further, if $\a_1\leq 0$, then $k^*=0$ and Consensus at $x_j\equiv 0$ is guaranteed, while if $\max_j|\a_j|=\a_1>0$, then at the least $0<\k_1^*\leq\a_1$ where $\k_1^*$ represents the final bifurcation for the disappearance of Disagreement states. If $\sum_{j=1}^N\a_j>0$, then this agrees with the final bifurcation value $\k_1^*=\k^*$, while if $\sum_{j=1}^N\a_j<0$, then there is one further bifurcation where the stable Compromise states fall into the Consensus state with a final pitchfork bifurcation.
\end{lemma}
\begin{proof}
First let us see that if $\k>\max_j|\a_j|=\a_1$, then there exist only three possible fixed points, Consensus $(x_j\equiv 0)$, or the two compromise states such that all $x_j>0 \ (x_j<0)$ as determined in Theorem \ref{t:opinion2}.

Indeed, suppose that $\k>\a_1$ and consider the fixed point map for $x_1$ and $x_N$ respectively,
\begin{align}
    0=(\a_1-x_1^2)x_1+\frac{\k}{N}\sum_{l=1}^N(x_l-x_1),\\
    0=(\a_N-x_N^2)x_N+\frac{\k}{N}\sum_{l=1}^N(x_l-x_N),
\end{align}
letting $x_a=\frac{1}{N}\sum_{l=1}^Nx_l$ and rearranging terms yields,
\begin{align}
    x_1(k-\a_1+x_1^2)=kx_a=x_N(k-\a_N+x_N^2).
\end{align}
As $\k>\a_1\geq\a_N$, we see that $\mathrm{sgn}(x_1)=\mathrm{sgn}(x_N)=\mathrm{sgn}(x_j)$ for all $j=1,...,N$.

Therefore $\k_1^*\leq \max_j|\a_j|$ and Theorem \ref{t:N>2synch2} tells us that if $\sum_{j=1}^N \a_j>0$ then for $\k>\k_1^*$, there are exactly the two stable Compromise fixed points and the unstable Consensus at zero fixed point. On the other hand if $\sum_{j=1}^N \a_j<0$, then there exists a final $\k^*>\a_1\geq\k_1^*$ such that  $\sum_{j=1}^N \frac{\a_j}{\k^*-\a_j}=0$ at which the pitchfork bifurcation occurs and we are left only with the Consensus at zero fixed point.
\end{proof}

To see convergence to one of the stable fixed points one can apply a gradient flow argument similar to those in \cite{LRS,RT}, however due to the existence of multiple stable fixed points, the initial data plays a role as to which fixed point the system converges. Therefore the entirety of Theorem \ref{t:opinion1} has been proved. We provide here a corollary to further state the totality of results.

\begin{corollary}\label{c:opinion3}
    Let $N=N_1+N_2$ where there are $N_1$ agents with conviction parameter $\a_j>0$ and $N_2$ agents with $\a_j\leq 0$. Then in the weak coupling regime, $0<\k<\k_*$, where $\k_*$ represents the first saddle-node bifurcation, there exists $3^{N_1}$ fixed points of \eqref{eq:SLopinion1}, where $2^{N_1}$ are stable fixed points. If $N_1>0$, then of these stable fixed points, $2$ are representative of compromise where all $x_j>0$, $(x_j<0$, respectively), while the rest are disagreement fixed points. The remaining $3^{N_1}-2^{N_1}$ fixed points have exactly the $1$ unstable consensus at zero fixed point, and the rest as saddle point disagreement states. 
\end{corollary}

As a model of opinion dynamics the Stuart-Landau system is interesting as it allows for various asymptotic states depending on the parameters of the system. Indeed, the nonlinear nature allows for stable configurations that exhibit disagreement, compromise, or consensus in the weak coupling regime. While as the coupling strength increases, the ability to have stable disagreements falls away through a series of saddle-node bifurcations, leaving only the compromise and consensus states depending on if on average most agents are ``agreeable" $(\a_j<0)$ or ``stubborn" $(\a_j>0)$.\\

%%%% FT HERE %%%%%%%%%

\section{Discussion}\label{S:discussion}

The purpose of this work was to establish a rigorous mathematical analysis of the fundamental SL model of synchronization allowing for full freedom of parameter values $(\a_j,\omega_j,\k) \in \R \times \R \times \R_+$. In accomplishing this task for the case of $N=2$ oscillators, a novel regime of phase-locking, \textit{leader-driven synchronization}, was discovered. This regime can be found when the coupling strength $\k$ is found between the values of the respective Hopf-parameters $\a_1$ and $\a_2$, in particular $\k/2\in[\a_2,\a_1]$, where the $2$ is seen explicitly because $N=2$. Within the \textit{leader-driven synchronization} regime there is a unique active stable phase-locked state for any values $\o_1,\o_2\in \R$, and thus for any $\g>0$. This is found in direct contrast to the $\a$-homogeneous case, and in models of synchronization which have undergone a phase-reduction, e.g. the famous Kuramoto model \cite{ABPRS,Kur}, and its higher-order generalizations, Lohe matrix model \cite{HK_rev,PP}, and Schr\"odinger-Lohe model \cite{AR,AR2}. Indeed, in each of these models, for any configuration of natural frequencies $\o_j \in S$, for the appropriate set $S$, there exists a $\k^* \in \R_+$ such that for all $\k>\k^*$ there exists a unique stable phase-locked state, where $\k^*=0$ if and only if there is $\o$-homogeneity so that $\o_j\equiv \w$ or $\g=0$. Conversely, this implies that for any $\k>0$ one can find a configuration of $\o_j$ (i.e. $\g$ large enough) such that no phase-locked state exists. However, in the \textit{leader-driven synchronization} regime, this is impossible. Of further interest is the ability to study the synchronous (or incoherent) behavior of oscillators undergoing amplitude death. In the $\alpha$-heterogeneous regime, we saw that amplitude death implies phase-locking, however this is not the case in the $\alpha$-homogeneous setting where incoherent phase dynamics can persist as amplitude death occurs. Investigation into the persistence of these regimes for $N>2$ will be the topic of future works.

Beyond the $N=2$ case, a rigorous convergence to the fully synchronized state for $\a$-heterogeneous, $\o$-homogeneous SL oscillators is provided under a half-plane initial data configuration. Further, the phase-transition for whether the synchronized state is active or converging to amplitude death is explicitly given.

Last, restricting SL oscillators to the real line ($\g=0)$ gives a model of opinion dynamics, which in low coupling regimes exhibits multistablitiy. Indeed, depending on $(\a_j,\k) \in \R \times \R_+$ there can be coexisting, stable disagreement, compromise, and consensus fixed points, a realistic result not seen in other continuous models of opinion dynamics like the Taylor and Abelson models \cite{PT1,PT2}. Increasing $\k$ gives rise to a series of saddle-node bifurcations which eliminate all disagreement fixed points, until only one (up to multiplication by -1) stable fixed point exists (either compromise or consensus depending on the distribution of $\{\a_j\}_{j=1}^N$).

\section*{Acknowledgments}
\begin{comment}
This project was partially supported by Modeling Nature (MNAT) Research Unit (project QUAL21-011), and the National Science Centre, Poland, grant number 2023/50/A/ST1/00447.
APM was supported by Grant No. PID2023-149174NB-I00 financed by the Spanish Ministry and Agencia
Estatal de Investigaci\'on MICIU/AEI/10.13039/501100011033 and ERDF funds (European Union), and by the ’Ram\'on y Cajal’ program of the Spanish Ministry of Science and Innovation (Grant RYC2021-031241-I).
The authors would also like to thank Mihail Hurmuzov for their technological skills in the transfer from chalkboard to digital files of Figures \ref{fig:Manifold_stability} and \ref{fig:Bifurcation_Ellipse}
\end{comment}

DNR and DP have been partially supported by the Modeling Nature (MNat) Research Unit, project QUAL21-011.
DNR has received funding from the National Science Centre, Poland, grant number 2023/50/A/ST1/00447.
DP has been supported by Grant C-EXP-265-UGR23 funded by Consejeria de Universidad, Investigacion e Innovacion  ERDF/EU Andalusia Program, and by Grant PID2022-137228OB-I00 funded by the Spanish Ministerio de Ciencia, Innovacion y Universidades, MICIU/AEI/10.13039/501100011033  ``ERDF/EU A way of making Europe''.
APM was supported by Grant No. PID2023-149174NB-I00 financed by the Spanish Ministry and Agencia Estatal de Investigaci\'on MICIU/AEI/10.13039/501100011033 and ERDF funds (European Union), and by the ’Ram\'on y Cajal’ program of the Spanish Ministry of Science and Innovation (Grant RYC2021-031241-I).
FT is partially funded by the PRIN-MUR project MOLE code 2022ZK5ME7, and by PRINPNRR project FIN4GEO within the European Union’s Next Generation EU framework, Mission 4, Component 2, CUP P2022BNB97.

The authors would also like to thank Mihail Hurmuzov for their technological skills in the transfer from chalkboard to digital files of Figures \ref{fig:Manifold_stability} and \ref{fig:Bifurcation_Ellipse}

\end{document}